\documentclass[twoside,11pt]{article}
\usepackage{jmlr2e}

\usepackage{hyperref}
\usepackage{pdfsync}
\usepackage{color}
\usepackage{graphicx}
\usepackage{caption}
\usepackage{fullpage}
\usepackage{cite}
\usepackage{array}

\usepackage{mathtools}
\mathtoolsset{showonlyrefs}

\usepackage{amsmath,amsfonts,amssymb}
\usepackage{graphicx}%

\usepackage[normalem]{ulem}

\makeatletter
\newcommand*{\T}{%
  {\mathpalette\@T{}}%
}
\newcommand*{\@T}[2]{%
  \raisebox{\depth}{$\m@th#1\intercal$}%
}
\makeatother

\newtheorem{assumption}{Assumption}

\newtheorem{newassumption}{Assumption}

\def\vu{{\bf u}}  \def\vw{{\bf w}} \def\vx{{\bf x}}
\def\vy{{\bf y}} \def\vz{{\bf z}}

  \def\calC{\mathcal{C}}
\def\calD{\mathcal{D}}  \def\calF{\mathcal{F}}

 \def\calN{\mathcal{N}} 
\def\calP{\mathcal{P}}  
\def\calS{\mathcal{S}} \def\calT{\mathcal{T}} \def\calU{\mathcal{U}}

\newcommand{\e}{\varepsilon}

\newcommand{\ds}{\, ds}

\newcommand{\R}{\mathbb{R}}

\newcommand{\E}{\mathbb{E}}

\renewcommand{\P}{\mathbb{P}}

\newcommand{\grad}{\nabla}
\newcommand{\lambdaMin}{\lambda_{\textup{min}}}
\newcommand{\lambdaMax}{\lambda_{\textup{max}}}

\newcommand{\xperp}{x_{nc}}

\newcommand{\ones}{{\bf 1}}
\newcommand{\myDiag}{\textup{diag}\,}
\newcommand{\ddt}{\frac{d}{dt}}

\newcommand{\dtau}{\,d\tau}
\newcommand{\myspan}{\textup{span}}

\newcommand{\C}{\mathcal{C}}

\newcommand{\diag}{\text{diag}}
\newcommand{\ns}{{n_s}}
\newcommand{\Proj}{\textup{P}}
\newcommand{\Hess}{\nabla^2}
\newcommand{\CP}{\textup{CP}}
\newcommand{\linG}{W}
\newcommand{\dist}{\textup{dist}}

\begin{document}

\title{Distributed Stochastic Gradient Descent: Nonconvexity, Nonsmoothness, and Convergence to Local Minima}

\author{\name Brian Swenson \email swenson@psu.edu\\ \addr Applied Research Laboratory\\
Pennsylvania State University\\
State College, PA
\AND
\name Ryan Murray \email rwmurray@ncsu.edu\\
\addr Department of Mathematics\\ 
North Carolina State University\\ 
Raleigh, NC
\AND 
\name H. Vincent Poor \email poor@princeton.edu\\ 
\addr Department of Electrical Engineering\\
Princeton University\\
Princeton, NJ
\AND
\name Soummya Kar \email soummyak@andrew.cmu.edu\\
\addr Department of Electrical and Computer Engineering\\
Carnegie Mellon University\\ 
Pittsburgh, PA
}

\editor{}

\maketitle

\begin{abstract}
In centralized settings, it is well known that stochastic gradient descent (SGD) avoids saddle points and converges to local minima in nonconvex problems. However, similar guarantees are lacking for distributed first-order algorithms.
The paper studies \emph{distributed} stochastic gradient descent (D-SGD)---a simple network-based implementation of SGD. Conditions under which D-SGD avoids saddle points and converges to local minima are studied. First, we consider the problem of computing critical points. Assuming loss functions are  nonconvex and possibly nonsmooth, it is shown that, for each fixed initialization, D-SGD converges to critical points of the loss with probability one. Next, we consider the problem of avoiding saddle points. In this case, we again assume that loss functions may be nonconvex and nonsmooth, but are smooth in a \emph{neighborhood} of a saddle point. It is shown that, for any fixed initialization, D-SGD avoids such saddle points with probability one. 
Results are proved by studying the underlying (distributed) gradient flow, using the ordinary differential equation (ODE) method of stochastic approximation, and extending classical techniques from dynamical systems theory such as stable manifolds.
Results are proved in the general context of subspace-constrained optimization, of which D-SGD is a special case. 
\end{abstract}


\section{Introduction} \label{sec:intro}
Nonconvex optimization has come to the forefront of machine learning, data science, and signal processing in recent years.
Applications in these areas often involve large-scale optimization problems that are addressed using first-order optimization techniques, i.e., gradient descent and its variants. First-order algorithms are particularly popular because
they are (relatively) computationally efficient, easy to implement, scalable, and achieve excellent results in practice \citep{jin2019nonconvex,hardt2016train,kingma2014adam}.

In this paper we study distributed (i.e., network-based) variants of gradient descent for nonconvex optimization. Distributed algorithms are an important tool for handling large-scale optimization problems that cannot be accommodated on a single machine \citep{lian2017can}. Such algorithms can also be useful for federated learning \citep{konevcny2015federated}, and have applications in a range of important domains including vehicular networks \citep{chang2020distributed}, edge computing \citep{wang2019adaptive}, and distributed control \citep{duchi2011dual}. 

We will consider the following setup:
Suppose there are $N$ computing nodes, or agents. Each agent $n=1,\ldots,N$ possesses a private function $f_n:\R^d\to\R$, that is nonconvex and possibly nonsmooth, and accessible only to agent $n$. Agents are assumed to be equipped with an overlaid communication network which may be used to communicate with neighboring agents. We are interested in optimizing the sum function $f:\R^d\to\R$ given by
\begin{equation} \label{eqn:f-def}
f(x)\coloneqq \sum_{n=1}^N f_n(x).
\end{equation}

Problems of this form are common in practice, with the most prominent example being empirical risk minimization. 
Concretely, suppose that $\calD_n = \{(x_i,y_i)\}_i$ represents a local dataset collected or stored by agent $n$. Let $\ell(\cdot,\cdot)$ denote some predefined loss function 
and let $h(\cdot,\theta)$ denote a parametric hypothesis class, with parameter $\theta$. In empirical risk minimization, the objective is to minimize the empirical risk over the data held by \emph{all} agents, i.e., solve the optimization problem
\begin{equation}\label{eq:opt-prob-ERM}
\min_\theta \sum_{(x,y)\in \bigcup_n \calD_n} \ell(h(x,\theta),y) ~ = ~ \min_\theta \sum_{n=1}^N \sum_{(x,y)\in\calD_n} \ell(h(x,\theta),y),
\end{equation}
where the objective above fits the form of \eqref{eqn:f-def} with $f_n(\theta) = \sum_{(x,y)\in\calD_n} \ell(h(x_i,\theta),y_i)$. Nonconvexity in this setting is typically induced by the function $h$ \citep{goodfellow2016deep}.

A basic but critical theoretical guarantee for \emph{nonconvex} optimization is that an algorithm converges to local minima. This is known to occur for many centralized first-order algorithms \citep{lee2016gradient,lee2019first,pemantle1990nonconv}. However, this issue is not well understood in the distributed setting. Despite the importance of distributed algorithms for multi-agent and large-scale processing, basic theoretical guarantees are largely lacking and there is a paucity of tools for gaining insight into the fundamental structure of \emph{distributed} optimization dynamics near saddle points. 
Motivated by these issues, in this paper we consider distributed stochastic gradient descent (D-SGD)---a simple distributed variant of SGD. We will characterize fundamental convergence properties of D-SGD (namely, avoidance of saddle points and convergence to local minima). 

The main results of the paper will be formally presented in the next section, but may be summarized as follows:
\begin{itemize}
    \item [1.] D-SGD converges to critical points of \eqref{eqn:f-def} when each $f_n$ is nonconvex and nonsmooth (but locally Lipschitz; see Theorems \ref{thrm:discrete-conv} and \ref{nonsmooth-conv-to-cp}).
    \item [2.] D-SGD avoids saddle points of \eqref{eqn:f-def} if each $f_n$ is sufficiently smooth in a neighborhood of the saddle point (see Theorems \ref{thrm:DT-nonconvergence} and \ref{thrm:DT-nonconvergence-gen}). 
\end{itemize}

Regarding the first contribution, there have been many excellent prior works considering the convergence of D-SGD (and other distributed algorithms) to critical points. See Section \ref{sec:related-work} for a detailed discussion. Motivated by applications in machine learning, here we consider convergence to critical points under more general nonsmoothness assumptions than previously considered. 
Previous work on distributed nonconvex nonsmooth optimization typically assumes that nonsmoothness enters the global loss function via an additive convex (or difference of convex) nonsmooth regularizer \citep{di2016next,scutari2019distributed}. This allows for nonsmooth regularization, but precludes important cases such as training neural networks with nonsmooth activation functions. In this paper we consider a substantially relaxed notion of nonsmoothness that encompasses a wide range of popular neural network architectures.

Showing nonconvergence to saddle points is more challenging than showing convergence to critical points. To our knowledge, these results represent the first rigorous analysis of D-SGD showing nonconvergence to saddle points. More precisely, we will show nonconvergence to saddle points with nonsingular Hessian. In the optimization literature, a point is sometimes called \emph{second-order stationary} if it meets the second-order criteria for optimality (the gradient is zero and the Hessian is positive semidefinite). Together, our results demonstrate, for the first time, convergence of D-SGD to second-order stationary points.
We note that several recent works including \citep{vlaski2019distributed1,vlaski2019distributed2,daneshmand2018second,daneshmand2018second-b,hong2018gradient} have obtained important related results concerning saddle point avoidance in distributed settings. See Section \ref{sec:related-work} for a detailed positioning of our results in the literature.

Our analysis techniques for establishing these results notably rely on studying underlying ordinary differential equations (ODEs). 
There has been a line of recent work that analyses optimization algorithms using an ODE approach \citep{su2014differential,krichene2015accelerated, shi2021understanding,davis2018stochastic}.
This is a powerful technique that can significantly simplify the analysis and distill the problem to a setting where fundamental intuition is far more transparent. 
The analysis approach we take in this paper follows that tradition. 


\bigskip
\noindent \textbf{Organization}. The remainder of the paper is organized as follows. Section \ref{sec:main-results} presents the main results, reviews related literature, and introduces notation to be used in the proofs. Sections \ref{sec:general-setup}--\ref{sec:stochastic-analysis-D-SGD} prove the main results (see Section \ref{sec:pf-strategy} for an overview of these sections and the general proof strategy). 
Section \ref{sec:conclusions} concludes the paper.

\section{Setup and Main Results} \label{sec:main-results}
We will now present the D-SGD algorithm and main results of the paper. We will be begin in Sections \ref{sec:DT main results}--\ref{sec:avoiding-saddles-intro} by presenting our main results in the context of smooth loss functions. Subsequently, in Section \ref{sec:nonsmooth-extension} we will generalize our results to the case where loss functions are nonsmooth.\footnote{The results of Section \ref{sec:nonsmooth-extension} are more general than those of Sections \ref{sec:DT main results}--\ref{sec:avoiding-saddles-intro}. However, because handling nonconvex nonsmooth loss functions requires concepts and notation that are not mainstream, we have elected to present these results separately to improve readability.}  To simplify the presentation, Sections \ref{sec:DT main results}--\ref{sec:nonsmooth-extension} only contain the statements of the main results without detailed discussion. In Sections \ref{sec:intro-discussion}--\ref{sec:assumptions-discussion} we will give an in-depth discussion of the results and assumptions used. Section \ref{sec:related-work} discusses related work. Section \ref{sec:high-level-pf-discussion} gives a high-level overview of proof techniques, including our use of ODE-based methods. Section \ref{sec:pf-strategy} gives a detailed roadmap of the proof strategy to be used through the rest of the paper. Finally, Section \ref{sec:prelims} sets up the notation to be used through the rest of the paper.

In D-SGD, agents will be assumed to be equipped with a communication network, represented by an undirected, unweighted graph $G=(V,E)$, where the set of vertices $V$ represents the set of agents and an edge $(i,j)\in E$ between vertices represents the ability of agents to communicate.

\bigskip
\noindent \textbf{D-SGD Algorithm}.
Let $k\geq 1$ be an integer (representing discrete time steps) and let $x_n(k)$ denote agent $n$'s estimate of a minimizer of \eqref{eqn:f-def} at iteration $k$. The D-SGD algorithm is defined agentwise by the recursion
\begin{equation} \label{dynamics_DT}
x_n(k+1) = x_n(k) - \alpha_k\big(\nabla f_n(x_n(k)) + \xi_n(k+1)\big) + \beta_k\sum_{\ell\in \Omega_n}\big(x_\ell(k) - x_n(k)\big),
\end{equation}
for $n=1,\ldots,N$, where $\{\alpha_k\}_{k\geq 1},\{\beta_k\}_{k\geq 1} \subset (0,1]$ are scalar weight parameters, $\xi_n(k)$ is zero-mean noise, and $\Omega_n$ represents the set of neighbors of agent $n$ in the graph $G$. The algorithm is initialized by setting the vector $(x_n(0))_{n=1}^N$ to some point $x_0\in \R^{Nd}$. 

\bigskip
\noindent \textbf{Intuition.} The D-SGD algorithm above follows the discrete-time consensus+innovations form~\citep{karmouraramanan2012distributed} and is related to the class of diffusion~\citep{chen2012diffusion} and distributed gradient descent (DGD)~\citep{nedic2009distributed} processes for distributed optimization.

In order to see how D-SGD relates to classical (centralized) SGD, observe that the algorithm consists of two components: a consensus term $\beta_k\sum_{\ell\in \Omega_n}(x_\ell(k) - x_n(k))$ and a local (stochastic) gradient descent term $- \alpha_k\left(\nabla f_n(x(k)) + \xi_n(k+1)\right)$. The consensus term is related to well-studied consensus algorithms \citep{dimakis2010gossip} (in particular, if one sets $f_n\equiv 0$, then \eqref{dynamics_DT} reduces to a standard consensus algorithm).
Intuitively, the consensus term asymptotically forces each $x_n(k)$ towards the network mean $\bar x(k) := \frac{1}{N} \sum_{n=1}^N x_n(k)$. In turn, the network mean behaves (nearly) like a classical stochastic gradient descent process. To see this, one takes the average over agents on both sides of \eqref{dynamics_DT} to obtain
\begin{equation} \label{eq:mean-D-SGD}
\bar x(k+1) \approx -\alpha_k\left(\nabla f(\bar x(k)) + \bar \xi(k+1)\right),
\end{equation}
where $\bar \xi(k) := \frac{1}{N}\sum_{n=1}^N \xi_n(k)$ is the network-averaged noise and $f$ is given by \eqref{eqn:f-def}.\footnote{The approximate equality is due to the fact that it deals with $\nabla f(\bar x(k))$ rather than $\sum_{i=1}^N \nabla f_n(x_n(k))$. This is made rigorous in Section \ref{sec:conv-to-CP}.}
Thus, all together, we will see that $\bar x(k)$ behaves like classical SGD and $x_n(k)\to \bar x(k)$ for each agent $n$.

\subsection{Convergence to Critical Points} \label{sec:DT main results}
We will now consider convergence to critical points when loss functions are smooth. 
We will make the following assumptions, recalling that a detailed discussion of assumptions can be found in Section \ref{sec:assumptions-discussion}. When we make an assumption concerning $f_n$, we mean that the assumption holds for each $f_n$ in \eqref{eqn:f-def}.
\begin{assumption} \label{a:smooth-stuff}
$f_n:\R^d\to \R$ is continuously differentiable with locally Lipschitz continuous gradient.
\end{assumption}
\begin{assumption} \label{a:coercive-sorta}
There exists a radius $R>0$ and constants $C_1,C_2>0$ such that
\begin{equation} 
\big\langle  \frac{\nabla f_n(x)}{\|\nabla f_n(x)\|},\frac{x}{\|x\|} \big\rangle \geq C_1 \quad \mbox{ and } \quad \|\nabla f_n(x)\|\leq C_2\|x\|
\end{equation}
for all $\|x\|\geq R$.
\end{assumption}
\begin{assumption} \label{a:grab-bag}
$~$
\begin{enumerate}
    \item The graph $G=(V,E)$ is undirected, unweighted, and connected.
   \item $\alpha_k = \Theta\left( k^{-\tau_\alpha}\right)$ and $\beta_k = \Theta\left(k^{-\tau_\beta}\right)$ with $0< \tau_\beta < \tau_\alpha$, $\frac{1}{2} < \tau_\alpha \leq 1$.
   \item The gradient noise at each agent $n$ 
   satisfies
  $$
  \E(\xi_{n}(k)\vert \calF_{k-1}) = 0 \quad \mbox{ and }  \quad \E(\|\xi_{n}(k)\|^2\vert \calF_{k-1})\leq B
  $$
  for some $B>0$, and all $k\geq 1$.
\end{enumerate}
\end{assumption} 
\begin{assumption} \label{a:CP-meas-zero}
Let $\CP_f\subset \R^d$ denote the set of critical points of $f$. The set $\R\backslash f(\CP_f)$ is dense in $\R$.
\end{assumption}
Assumption \ref{a:smooth-stuff} is a standard smoothness assumption. Assumption \ref{a:coercive-sorta} ensures that the gradient points outwards asymptotically, and imposes a mild restriction on the growth rate of the gradient. 
Note that this assumption is satisfied in the case of $\ell_1$ and $\ell_2$ regularization. 
Assumption \ref{a:grab-bag}, part 1 ensures that information can diffuse through the network. Part 2 assumes a convenient form for the update weights (see Section \ref{sec:prelims} for a formal definition of $\Theta(\cdot)$ notation). Part 3 makes a standard assumption that the gradient noise is zero mean with bounded variance. 

Finally, in Assumption \ref{a:CP-meas-zero} we recall that
a point $x^*$ is said to be a critical point of $f$ if $\nabla f(x^*) = 0$. The set of critical \emph{values} is given by the image of the set of critical points, $f(\CP_f)$. Assumption \ref{a:CP-meas-zero} simply states that the set of non-critical values is dense in $\R$. While this assumption is quite technical, we emphasize that it is a mild assumption commonly used obtain convergence to critical points in stochastic approximation procedures; cf. \citep{davis2018stochastic,duchi2018stochastic,benaim2005stochastic}. The assumption 
holds in many practical circumstances of interest involving nonsmooth loss functions \citep{davis2018stochastic}.

Our first main result, stated next, is that D-SGD achieves consensus and converges to critical points under these assumptions. 
\begin{theorem}[Convergence to Critical Points]\label{thrm:discrete-conv}
Let $\{(x_n(k))_{n=1}^N\}_{k\geq 1}$ be a D-SGD process  \eqref{dynamics_DT} with $f$ given by \eqref{eqn:f-def}. Suppose Assumptions \ref{a:smooth-stuff}--\ref{a:CP-meas-zero} hold. Then, given any fixed initial condition, for each $n=1,\ldots,N$ the following hold with probability 1:
\begin{enumerate}
  \item [(i)] Agents achieve consensus in the sense that $\lim_{k\to\infty} \|x_n(k) - x_\ell(k)\|=0$ for all $\ell=1,\ldots,N$.
  \item [(ii)]$x_n(k)$ converges to the set of critical points of $f$.
\end{enumerate}
\end{theorem}

We note that in this result and many of the following results, $x_n(k)$ converges to the \emph{set} of critical points. This is because the critical point set may contain a connected set. See Section \ref{sec:prelims} for the definition of setwise convergence. An extension of Theorem \ref{thrm:discrete-conv} for the case of nonsmooth loss functions will be given in Theorem \ref{nonsmooth-conv-to-cp}.

\subsection{Avoiding Saddle Points} \label{sec:avoiding-saddles-intro}
Next, we consider the issue of avoiding saddle points when loss functions are smooth. We say that $x^*\in \R^d$ is a saddle point of $f$ if $\nabla f(x^*)=0$ and $x^*$ is neither a local maximum or minimum of $f$. We will consider saddle points satisfying the following notion of regularity.
\begin{definition} [Regular Saddle Point] \label{def:strict saddle}
A saddle point $x^*$ of $f$ will be said to be \emph{regular} (or \emph{nondegenerate}) if the Hessian $\Hess f(x^*)$ is an invertible matrix.
\end{definition}
Note that this is equivalent to requiring all eigenvalues of $\nabla^2 f(x^*)$ to be nonzero.\footnote{The term \emph{nondegenerate} is commonly used for this concept in the optimization community. However, since we will deal with \emph{nonconvergence} to these points, we prefer to use the term ``regular'' in this paper to avoid frequent (and confusing) use of double negatives.}

To ensure D-SGD avoids saddle points, we will require the following mild technical assumption.  
\begin{assumption}[Differentiability of Eigenvectors] \label{a:eigvec-continuity}
Suppose $x^*\in\R^d$ is a saddle point of \eqref{eqn:f-def}.
For each $n$, the eigenvectors of $\Hess f_n(x)$ are differentiable at $x^*$ in the sense that, for each $x$ in a neighborhood of $x^*$, there exists an orthonormal matrix $U_n(x)$ that diagonalizes $\Hess f_n(x)$ such that $x\mapsto U_n(x)$ is differentiable in a neighborhood of $x^*$.
\end{assumption}
We emphasize that this assumption is relatively innocuous and should be satisfied by most functions encountered in practice. The assumption does not arise in centralized optimization but is needed to rule out certain highly pathological cases that can arise in the distributed setting, e.g., see Example 
19 in \citep{SMKP-TAC2020}. 
The assumption is, in fact, guaranteed to hold under more familiar (and less technical) conditions. For example, the assumption always holds when each eigenvalue of $\Hess f_n(x^*)$ is unique or if each $f_n$ is analytic \citep{katoBook}. However, we have chosen to state Assumption \ref{a:eigvec-continuity} in its present form in order to keep it as unrestrictive as possible. Further discussion of this assumption and and a simple illustrative example can be found in Section \ref{sec:assumptions-discussion}. Discussion of why the assumption is needed in the distributed setting can be found in Section \ref{sec:CT-stable-manifold} (see Remark \ref{remark:eigval-cont}).

In order to prevent D-SGD from getting trapped in ``bad'' sets (e.g., saddle points or stable manifolds), we must assume that the noise provides some minimum excitation. 
\begin{assumption}
\label{a:SGD}
For every unit vector $\theta \in \R^{d}$ and some constant $c_1>0$ we have
$$\E((\bar \xi(k)^\T \theta)^+\vert \calF_{k-1}) \geq c_1,
$$
where $\bar \xi(k):= \frac{1}{N}\sum_{n=1}^N \xi_n(k)$ and for $a\in\R$ we let $(a)^+:=\max\{0,a\}$.
\end{assumption}
Intuitively, Assumption \ref{a:SGD} simply asserts that the networked-average noise occasionally perturbs in all directions. The assumption is easily satisfied, for example, if each $\xi_n(k)$ is independently drawn from a nondegenerate (with positive definite covariance) Gaussian distribution or a uniform distribution over the $d$-dimensional unit sphere at each agent. 

Finally, we will assume a slightly stronger smoothness condition than required earlier for convergence to critical points (cf. Assumption \ref{a:smooth-stuff}).\footnote{In the optimization literature, smoothness often means that a function has a Lipschitz continuous gradient, as in Assumption \ref{a:smooth-stuff}. In the differential equations literature, the degree of smoothness refers to the number of times a function is continuously differentiable. We will use the term smoothness to refer to both concepts. However, note that if a function is more than twice continuously differentiable, it has a locally Lipschitz gradient, so Assumption \ref{a:C3} is stronger than Assumption \ref{a:smooth-stuff}.}
\begin{assumption} \label{a:C3}
$f_n:\R^d\to\R$ is three times continuously differentiable.
\end{assumption}

The next theorem gives our second main result. The theorem shows that the critical point reached by D-SGD cannot be a regular saddle point.

\begin{theorem}[Nonconvergence to Saddle Points]\label{thrm:DT-nonconvergence}
Suppose $\{(x_n(k))_{n=1}^N\}_{k\geq 1}$  is a D-SGD process \eqref{dynamics_DT}, $f$ is given by \eqref{eqn:f-def}, and $x^*$ is a regular saddle point of $f$. Suppose Assumptions \ref{a:coercive-sorta}--\ref{a:grab-bag} and \ref{a:eigvec-continuity}--\ref{a:C3} hold. 
Then, for each $n=1,\ldots,N$, and for any initialization $x_0 \in \R^{Nd}$
$$
\P(x_n(k)\to x^*) = 0.
$$
\end{theorem}

Finally, as an immediate consequence of Theorems \ref{thrm:discrete-conv} and \ref{thrm:DT-nonconvergence}, we obtain the following local-minimum convergence guarantee. 
\begin{theorem} [Convergence to Local Minima] \label{thrm:conv-to-mins}
  Suppose $\{(x_n(k))_{n=1}^N\}_{k\geq 1}$ is a D-SGD process \eqref{dynamics_DT} with $f$ given by \eqref{eqn:f-def}.
  Suppose Assumptions  
  \ref{a:coercive-sorta}--\ref{a:C3}, 
  hold 
  and that every saddle point of $f$ is regular. Then for each $n=1,\ldots,N$, and for any initial condition, $\lim_{k\to\infty} \|x_n(k) - x_\ell(k) \| = 0$ for all $\ell$, and $x_n(k)$ converges to the set of local minima of $f$ with probability 1.
\end{theorem}

An extension of Theorems \ref{thrm:DT-nonconvergence} and \ref{thrm:conv-to-mins} to the case of nonsmooth loss functions will be given in Theorems \ref{nonsmooth-conv-to-cp}--\ref{thrm:DT-nonconvergence-gen}.

\subsection{Extension: Nonsmooth Loss Functions} \label{sec:nonsmooth-extension}

In the previous two sections we focused on smooth loss functions. However, nonsmooth loss functions play a critical role in many machine learning applications, e.g., deep learning with ReLU activation functions or $\ell_1$ regularization. In this section we will obtain the following generalizations of the above results:
\begin{enumerate}
    \item [1.] If $f_n$ is nonsmooth, then D-SGD still converges to critical points. 
    \item [2.] If $f_n$ is nonsmooth in general, but is smooth in a \emph{neighborhood} of a saddle point, then we still obtain nonconvergence to the saddle point.
\end{enumerate} 

We will formalize these generalizations now. 
Instead of Assumption \ref{a:smooth-stuff}, consider the following assumption. 
\begin{assumption} \label{a:f_n-Lipschitz}
$f_n$ is locally Lipschitz continuous.
\end{assumption}
Under this assumption it is a well-known consequence of Rademacher's theorem that $f_n$ is differentiable almost everywhere \citep[Ch. 3]{evans2015measure}. Thus one may define the following  generalization of the gradient for such functions. 
\begin{definition}[Generalized Gradient] \label{def:gen-grad}
Given a locally Lipschitz continuous function $h$, the generalized gradient of $h$ is given by
\begin{equation}
    \partial h(x) := \textup{co}\left\{\lim_{i\to\infty} \nabla h(x_i): \, x_i\to x, ~ \nabla h(x_i) \textup{ exists}\right\},
\end{equation}
where co denotes the convex hull. 
\end{definition}
Properties of the generalized gradient and further discussion of this definition can be found in \citep{clarke1990optimization}. Of particular note, if $h$ is locally Lipschitz continuous, then $\partial h(x)$ is a nonempty compact convex set for all $x$. If $h$ is continuously differentiable, then $\partial h(x)$ coincides with the traditional gradient and if $h$ is convex, then $\partial h(x)$ simply gives the subgradient set. 

Given Definition \ref{def:gen-grad}, a point $x^*$ is said to be a critical point of $f$ if $0\in \partial f(x^*)$ and $x^*$ will be called a saddle point if it is a critical point but not a local maximum or minimum. 

We will make the following direct generalization of Assumption \ref{a:coercive-sorta} for nonsmooth functions.
\begin{assumption} \label{a:coercive-sorta-nonsmooth}
There exists a radius $R>0$ and constants $C_1,C_2>0$ such that
\begin{equation} \label{eq:assump-grad-alignment}
\left\langle \frac{v}{\|v\|},\frac{x}{\|x\|} \right\rangle \geq C_1\quad \mbox{ and } \quad \| v\|\leq C_2\|x\|
\end{equation}
for all $v\in \partial f_n(x)$ and $\|x\|\geq R$. 
\end{assumption}

In this context, D-SGD is given by the recursion
\begin{equation} \label{dynamics_DT-DI}
x_n(k+1) = x_n(k) + \beta_k\sum_{\ell\in \Omega_n}\big(x_\ell(k) - x_n(k)\big) - \alpha_k\big(v(k) + \xi_n(k+1)\big),
\end{equation}
where $v(k)\in \partial f_n(x_n(k))$. 

Finally, in order to ensure that D-SGD will descend the objective function, we require the following technical assumption. 
\begin{assumption}[Chain rule] \label{a:chain-rule}
For any absolutely continuous function $\vx:[0,\infty)\to \R^{d}$,
$f$ satisfies the chain rule
$$
\ddt f(\vx(t)) = \left\langle v, \ddt \vx(t) \right\rangle, 
$$
for some $v\in \partial f(\vx(t))$, and almost all $t\geq 0$. 
\end{assumption}
This assumption is a technical regularity condition needed to avoid pathological cases---we expect it to be satisfied by most functions encountered in practice. Further intuition for why the assumption is needed can be found in \citep{daniilidis2020pathological} (see also Remark \ref{remark:chain-rule} below).
In particular, the assumption is guaranteed to hold for a wide range of functions including common nonsmooth neural network architectures \citep[Sec. 5]{davis2018stochastic}. 

Under these relaxed assumptions we obtain convergence to critical points of nonsmooth loss functions. The following result generalizes Theorem  \ref{thrm:discrete-conv}.
\begin{theorem}[Convergence to Critical Points] \label{nonsmooth-conv-to-cp}
Suppose $\{(x_n(k))_{n=1}^N\}_{k\geq 1}$  satisfies \eqref{dynamics_DT-DI} with $f$ given by \eqref{eqn:f-def}.
Suppose Assumptions \ref{a:grab-bag}--\ref{a:CP-meas-zero} and \ref{a:f_n-Lipschitz}--\ref{a:chain-rule} hold. 
Then for each agent $n$ we have that: (i) agents reach consensus in the sense that $\|x_n(k) - x_\ell(k)\|\to 0$ for each $\ell=1,\ldots,N$ and (ii) $x_n(k)$ converges to the set of critical points of $f$. 
\end{theorem}


The next theorem generalizes Theorem \ref{thrm:DT-nonconvergence}.
Critically, note that in the theorem we only assume smoothness in a \emph{neighborhood} of $x^*$.
\begin{theorem} [Nonconvergence to Saddle Points] \label{thrm:DT-nonconvergence-gen}
  Suppose $\{(x_n(k))_{n=1}^N\}_{k\geq 1}$  satisfies \eqref{dynamics_DT-DI} with $f$ given by \eqref{eqn:f-def}. 
  Suppose Assumptions \ref{a:grab-bag}, \ref{a:eigvec-continuity}--\ref{a:SGD}, and \ref{a:f_n-Lipschitz}--\ref{a:chain-rule} hold. Suppose $x^*$ is a regular saddle point and there exists some neighborhood of $x^*$ on which Assumption \ref{a:C3} holds. Then for each $n=1,\ldots,N$ and for any initialization $x_0\in \R^{Nd}$ we have
  $$
  \mathbb{P}\left(x_n(k)\to x^* \right) = 0.
  $$
\end{theorem}
Finally, combining Theorems \ref{nonsmooth-conv-to-cp} and \ref{thrm:DT-nonconvergence-gen} we immediately obtain the following result, generalizing Theorem \ref{thrm:conv-to-mins}.
\begin{theorem} [Convergence to Local Minima] \label{thrm:conv-to-mins-DI}
  Suppose $\{(x_n(k))_{n=1}^N\}_{k\geq 1}$ satisfies \eqref{dynamics_DT-DI} with $f$ given by \eqref{eqn:f-def}.
  Suppose Assumptions  
  \ref{a:grab-bag}--\ref{a:SGD}, and \ref{a:f_n-Lipschitz}--\ref{a:chain-rule}
  hold.
  Suppose that for every saddle point $x^*$, $x^*$ is regular and \ref{a:C3} holds in some neighborhood of $x^*$.
  Then for each $n=1,\ldots,N$, and for any initial condition $x_0\in \R^{Nd}$, $\lim_{k\to\infty} \|x_n(k) - x_\ell(k)\| = 0$ for each $\ell$, and $x_n(k)$ converges to the set of local minima of $f$ with probability 1.
\end{theorem}

\subsection{Discussion} \label{sec:intro-discussion}

The main novelty of these results is that they show (i) convergence to critical points of D-SGD under  weak nonsmoothness assumptions and (ii) nonconvergence to saddle points. Regarding (i), 
convergence to critical points is shown under the assumption of Lipschitz continuous loss functions (Assumption \ref{a:f_n-Lipschitz}), which is essentially the weakest assumption one can make while still ensuring the generalized gradient (the analog of the subgradient for nonconvex functions) still exists. Additional technical assumptions are made, but as discussed in Section \ref{sec:assumptions-discussion}, these are all fairly mild. 
From a practical standpoint, our result showing convergence to critical points (Theorem \ref{nonsmooth-conv-to-cp}) applies to most nonsmooth neural network architectures used in practice today, cf. \citep{davis2018stochastic}.

Regarding (ii), most previous work on D-SGD has focused on showing convergence to critical points under various assumptions but has not dealt with nonconvergence to saddle points (see Section \ref{sec:related-work} for a discussion of related work). Analyzing nonconvergence to saddle points is challenging as it requires dealing with tools that are not standard in optimization. In particular, one typically deals with stable manifolds---a concept from classical dynamical systems theory \citep{coddington1955theory}. Stable manifolds have been used to study gradient dynamics in the centralized regime and show nonconvergence to saddle points \citep{lee2016gradient,pemantle1990nonconv}. 
However, showing nonconvergence of D-SGD to saddle points is more challenging than in the centralized case because the classical theory for stable manifolds does not apply. This is because the scaling factors $\alpha_k$ and $\beta_k$ in \eqref{dynamics_DT} make the dynamics nonautonomous (meaning the right hand side depends on time). The nonautonomous nature of the dynamics is an intrinsic part of their operation---by changing the relative weight placed on the consensus vs gradient descent terms over time, the consensus term gradually ``overpowers'' the gradient descent term, ensuring that consensus is reached, but at a point which also optimizes the loss function. Classical stable manifold theory, which is traditionally applied to show nonconvergence to saddle points, does not apply to nonautonomous systems.

To our knowledge, no analogous theory for stable manifolds exists for generic nonautonomous systems, which necessitates much of the analysis in the present paper. In particular, the backbone of our proof is a novel stable manifold theorem for (continuous-time) distributed gradient flow (DGF) developed in \citep{SMKP-TAC2020}. 



Our analysis, both for showing convergence to critical points and nonconvergence to saddle points, will rely on studying the ordinary differential equation (ODE) that underlies D-SGD.
ODE-based methods for studying optimization dynamics have grown in popularity recently \citep{su2014differential,krichene2015accelerated,wibisono2016variational}.  These powerful techniques often allow for much simpler analysis and provide deep insights by characterizing the qualitative properties of the underlying ODE.
A high level discussion of proof techniques and our use of ODE-based methods will be given in Section \ref{sec:high-level-pf-discussion} and a more detailed roadmap of the proof strategy will be given in Section \ref{sec:pf-strategy}. 

It is worth noting that the problem of minimizing \eqref{eqn:f-def} in a distributed setting can be viewed as a subspace-constrained optimization problem (see Section \ref{sec:general-setup}).
Subspace-constrained optimization problems have many applications in engineering including distributed optimization and inference \citep{hajinezhad2019perturbed,nassif2019adaptation-I,nassif2019adaptation-II}. While this paper is motivated by the application of distributed nonconvex optimization and our main results will be stated in this context, all of our results are in fact proved for the broader class of subspace-constrained optimization problems. See Sections \ref{sec:pf-strategy} and \ref{sec:general-setup} for more details.

\subsection{Discussion of Assumptions} \label{sec:assumptions-discussion}

In order to keep the presentation of the main results concise, we have deferred an in-depth discussion of assumptions to this section. 
For convenience, a table summarizing the assumptions and their usage in the main results can be found in Figures \ref{fig:assumptions-summary} and \ref{fig:results-summary}.

We will begin by discussing assumptions related to convergence to critical points. Theorem \ref{nonsmooth-conv-to-cp} is a strict generalization of Theorem \ref{thrm:discrete-conv}, hence we will focus on the assumptions made in that theorem. Assumption \ref{a:f_n-Lipschitz} is a standard notion of continuity. (Note that it is  weaker than the standard smoothness assumption \ref{a:smooth-stuff} used in Theorem \ref{thrm:discrete-conv}.) Assumption \ref{a:f_n-Lipschitz} is essentially the weakest assumption under which the generalized gradient can be ensured to exist, and hence under which D-SGD dynamics are well defined. For a discussion of other notions of nonsmoothness, see \citep{li2020understanding}.

Assumption \ref{a:coercive-sorta-nonsmooth} (the generalization of \ref{a:coercive-sorta} to nonsmooth functions) is nonstandard, but is relatively straightforward. The assumption comes in two parts. The first statement in the assumption is essentially a weak form of coercivity. A standard coercivity assumption in the context of optimization is to assume that $f(x) \to\infty$ as $\|x\|\to \infty$. In contrast, the first part of Assumption \ref{a:coercive-sorta-nonsmooth} merely assumes that the gradient points outward asymptotically. On the other hand, the second part of Assumption \ref{a:coercive-sorta-nonsmooth} is not related to coercivity. It is a restriction on the asymptotic growth rate of the gradient. The assumption is effectively equivalent to requiring that $f_n$ be bounded by a quadratic function the sense that $f_n(x) \leq C\|x\|^2$ for some $C>0$. This relationship can be obtained by integrating the inequality in \ref{a:coercive-sorta-nonsmooth} and applying the fundamental theorem of calculus. This assumption is needed to rule out unbounded oscillations that can occur due to discretization error. The only place in the paper it is directly applied is the proof of Lemma \ref{lemma:technical-inner-prod} (see in turn, Lemma \ref{lemma:compact-set}) where we show that iterates of D-SGD remain in a compact set. This is required to show convergence to critical points using the ODE method of stochastic approximation---see Theorem \ref{thrm:tame-functions}, item 1 and \citep{davis2018stochastic}. Notably, the assumption is not required to obtain convergence to critical points in the continuous-time settings---see \citep{SMKP-TAC2020} Assumption A.3 and Theorem 3. To the best of our understanding, the fact that we needed this assumption is an artifact of discretization. 

Assumption \ref{a:chain-rule} is a technical assumption that is required to ensure that a generalized gradient descent process actually descends the objective function when it is nonsmooth (i.e., under Assumption \ref{a:f_n-Lipschitz}). To understand why the assumption is needed, note that in the context of smooth optimization, if $h$ is a differentiable function and $\vx(t)$ is a gradient flow trajectory, so that $\ddt\vx(t) = -\nabla h(\vx(t))$, we have 
\begin{equation}\label{eq:descent-example}
    \frac{d}{dt} h(\vx(t)) = \langle \nabla h(\vx(t), \ddt\vx(t) \rangle
     = -\|\nabla h(\vx(t))\|^2,
\end{equation}
where the first equality follows from the standard chain rule. This equality establishes the critical relationship that a gradient flow trajectory descends the objective. For generalized gradient descent processes on nonsmooth functions (e.g., satisfying \ref{a:f_n-Lipschitz}), the chain rule need not hold and it is possible that a descent relationship such as \eqref{eq:descent-example} does not hold. Examples where this occurs are typically pathological and are not likely be encountered in practice \citep{daniilidis2020pathological}. Hence, the assumption is quite mild, but is required for technical reasons. 

Assumption \ref{a:grab-bag} is standard. Assumption \ref{a:CP-meas-zero} is technical, but is a standard assumption for obtaining convergence to critical points in stochastic approximation algorithms \citep{duchi2018stochastic, davis2018stochastic} and the assumption is quite mild.

We will now discuss the assumptions made for nonconvergence to saddle points. We will focus on the Assumptions of Theorem \ref{thrm:DT-nonconvergence-gen} as this generalizes Theorem \ref{thrm:DT-nonconvergence}.
Assumption \ref{a:eigvec-continuity} is not a standard assumption and is specifically required to handle the setting of \emph{distributed} optimization. To clarify the assumption, the following example provides an illustration of what is meant by ``differentiable eigenvectors'' in the assumption.
\begin{example}
Consider the function $f(x,y) = \frac{x^4}{12} + xy + \frac{y^2}{2}$. The Hessian of this function will take the form
\[
\nabla^2 f = \begin{bmatrix}x^2 & 1 \\ 1 & 1 \end{bmatrix}.
\]
We may compute that the eigenvectors of this matrix are given by 
\[
v_1(x) = \begin{bmatrix} \frac{x^2-1-\sqrt{5-2x^2 + x^4}}{2} \\ 1\end{bmatrix},\qquad v_2(x) = \begin{bmatrix} \frac{x^2-1+\sqrt{5-2x^2 + x^4}}{2} \\ 1\end{bmatrix}
\]
If we define
\[
U(x,y) = \begin{bmatrix} \frac{v_1(x)}{|v_1(x)|} & \frac{v_2(x)}{|v_2(x)|} \end{bmatrix}
\]
then $U$ is a unitary matrix which diagonalizes $\nabla^2 f $. One can verify that the components of this matrix are differentiable in $(x,y)$, and hence this matrix would satisfy the ``differentiable eigenvectors'' assumption. We notice that $U(x,y)$ is not the only matrix diagonalizing $\nabla^2 f$ at any given point, but this particular choice does give a differentiable function. Assumption \ref{a:eigvec-continuity} is known to hold for analytic functions \citep[p. 111]{katoBook}, and counterexamples where the assumption fails to hold must be carefully constructed to demonstrate the pathology \citep[Example 19]{SMKP-TAC2020} \citep[p. 111]{katoBook}. Hence, we anticipate it holds in many common situations.  
\end{example}
Regarding Assumption \ref{a:eigvec-continuity}, it is worth noting that the property we will require in proofs is that the Hessian of the sum function $(x_n)_{n=1}^N \mapsto \sum_n f_n(x_n)$ has differentiable eigenvectors. However, because the Hessian of this function is block diagonal, it is equivalent to assume that the property holds for each block, which is what is done is Assumption \ref{a:eigvec-continuity}. This is why the assumption deals with the individual functions and not the sum function.

In words, Assumption \ref{a:SGD} assumes that the noise occasionally perturbs in all directions. It is the same assumption made in \citep{pemantle1990nonconv}. The assumption enables gradient dynamics to escape saddle points by knocking the optimization process out of any undesirable sets it could get trapped in. Finally, to ensure nonconvergence to a saddle point, Assumption \ref{a:C3} requires that each $f_n$ is three times continuously differentiable (i.e., $C^3$) in a neighborhood of the saddle point.
The reason the assumption is needed is because we use the existence of a stable manifold near the saddle point to establish nonconvergence to the saddle point. The techniques for establishing the existence of stable manifolds typically require the function to be locally $C^3$.

\begin{figure}
\begin{center}

\begin{tabular}{|m{6em}|m{25em}|}
\hline
Assumption & Short description of assumption \\
 \hline
 A.1 & $f_n$ is $C^1$ with Lipschitz gradient. \\
 \hline
 A.2 & $\nabla f_n$ points outward as $x \to \infty$\\
 \hline
 A.3 & The connectivity graph is connected, weights $\alpha_k,\beta_k$ have appropriate decay, and the noise injection is mean zero and finite variance.\\
 \hline
 A.4 & The set of critical values of $f$ has a dense complement.\\
 \hline
 A.5 & The matrix $\nabla^2 f_n(x)$ may be diagonalized using a differentiable matrix $U(x)$.\\
 \hline
 A.6 & The noise provides some minimal excitation in all directions.\\
 \hline
 A.7 & $f_n$ is thrice differentiable.\\
 \hline
 A.8 & $f_n$ is locally Lipschitz continuous.\\
 \hline
 A.9 & The generalized gradient also points outward as $x \to \infty$.\\
 \hline
 A.10 & A generalized chain rule holds for gradient flow paths.\\
 \hline
\end{tabular}
\end{center}
\caption{Table of main assumptions, with brief descriptions.}
\label{fig:assumptions-summary}
\end{figure}

\begin{figure}
\begin{center}

\begin{tabular}{|m{4em}|m{20em}|m{6em}|} 
\hline
Theorem & Short description of theorem & Assumptions used \\
 \hline
 1 & D-SGD achieves consensus and converges to critical points (when $f_n$ has Lipschitz gradient). & A.1-A.4 \\
 \hline
 3 & D-SGD does not converge to saddle points which are regular (assuming $f_n$ is $C^3$) & A.2-A.3, A.5-A.7 \\
 \hline
 4 & D-SGD converges to local minimizers assuming that all saddle points are regular and $f_n$ is $C^3$. & A.2-A.7 \\
 \hline
 6 & D-SGD achieves consensus and converges to critical points in non-smooth case (generalizes Theorem 1) & A.3-A.4, A.8-A.10 \\
 \hline
 7 & D-SGD does not converge to saddle points which are regular and thrice differentiable near the saddle point (generalizes Theorem 3). & A.3, A.5-A.6, A.8-A.10 \\
 \hline
 8 & D-SGD must converge to local minimizers in non-smooth case, assuming all saddle points are regular enough to apply Theorem 7. & A.3-A.6, A.8-A.10 \\
 \hline
\end{tabular}
\end{center}
\caption{Table of main results, with brief descriptions.}
\label{fig:results-summary}
\end{figure}

\subsection{Related Work} \label{sec:related-work}
There has been a significant body of recent research on first-order algorithms for nonconvex optimization in classical (centralized) settings. Research on saddle-point nonconvergence and saddle-point escape time in centralized gradient methods includes
\citep{lee2018distributed,ge2015escaping,du2017gradient,jin2019nonconvex,murray2017revisiting,du2017gradient,du2018gradient}. Reference \citep{pemantle1990nonconv} considers nonconvergence to unstable points (such as saddle points) in \emph{autonomous} stochastic recursive dynamical systems (such as centralized gradient descent). Some of the key techniques used in the present paper to study D-SGD are inspired by the techniques developed in \citep{pemantle1990nonconv}. We stress, however, that the nonautonomous nature of the distributed dynamics makes the problem here more challenging and precludes the use of classical stable manifold theorems that crucially underpin all of above results.

Distributed gradient algorithms for \emph{convex} optimization have been the subject of intensive research over the past decade, see e.g., \citep{nedic2009distributed,rabbat2004distributed,duchi2011dual,ram2010distributed,scaman2019optimal,jakovetic2014fast,nedic2014distributed} and references therein. Important considerations include time-varying vs. static communication graphs \citep{nedic2014distributed}, directed vs. undirected communication graphs \citep{nedic2014distributed}, communication efficiency \citep{mota2013d}, rate of convergence \citep{jakovetic2014fast,scaman2019optimal,jakovetic2018unification,uribe2020dual,scaman2019optimal,mokhtari2016dsa}, and mitigating communication overhead \citep{koloskova2019decentralized}. 

Distributed algorithms for \emph{nonconvex} optimization have been a subject of more recent focus. The majority of work on distributed gradient methods for nonconvex optimization have focused on addressing various challenging issues related to convergence to critical points 
\citep{bianchi2012convergence,di2016next,sun2016distributed,scutari2019distributed,tatarenko2017non,tian2018asy,sun2019distributed,wai2017decentralized}.
More recent work has focused on refined convergence guarantees.
References \citep{daneshmand2018second,daneshmand2018second-b} consider deterministic DGD and nonconvergence to saddle points. It is shown that for sufficiently small constant step size, DGD avoids regular saddle points and converges to the neighborhood of a second-order stationary point from almost all initializations. The result relies on the classical stable manifold theorem applied to an appropriate Lyapunov function that captures both the consensus dynamics and the gradient dynamics descending \eqref{eqn:f-def} given a fixed step size. 
In a similar vein, the work \citep{vlaski2019distributed1,vlaski2019distributed2} considers a constant-step size gradient-based algorithm for distributed stochastic optimization. It is shown that the algorithm avoids saddle points, and a polynomial escape time bound is established. The work \citep{vlaski2019second} considers relaxed conditions on gradient noise variance to escape saddle points.
Our work differs from these in that we consider a decaying step size (and stochastic, vis-\`{a}-vis \citep{daneshmand2018second}) version of D-SGD which is able to asymptotically handle noise and obtain convergence to consensus and convergence to local minima, rather than a neighborhood of local minima or recurrence to minima. Moreover, we explicitly characterize the mean dynamics of D-SGD (i.e., continuous-time DGF) which is a powerful technical tool and provides significant insight into the structure of distributed gradient algorithms near saddle points. 
We also note that a primal-dual method for distributed nonconvex optimization with local minima convergence guarantees was considered in \citep{hong2018gradient}. In this paper, however, we focus explicitly on properties of distributed gradient descent. 

The present work is closely related to \citep{bianchi2012convergence} which considers a slightly different variant of D-SGD  for constrained optimization. In order to handle the constraint, \citep{bianchi2012convergence} utilizes a projection step integrated with each D-SGD iteration. Reference \citep{bianchi2012convergence} showed convergence of this algorithm to the set of KKT points. Our work may be seen as a partial generalization of \citep{bianchi2012convergence} in that, we consider \emph{unconstrained} optimization of \eqref{eqn:f-def} (thus omitting the projection step) for nonsmooth functions and study the problem saddle point avoidance.  
We note that, in order focus on the essential problems in evasion of saddle points in distributed optimization, we have decided to focus on the relatively simple setting of time-invariant directed communication graphs. 
Future work may consider extensions to undirected and time-varying graphs. 

The topic of distributed optimization when loss functions are both \emph{nonconvex} and \emph{nonsmooth} has received limited attention in the literature. References \citep{di2016next} and \citep{scutari2019distributed} consider convergence to critical points when the distributed objective is the sum of a smooth nonconvex component and a nonsmooth convex component  (or difference-of-convex with smooth convex part). 
More recently, \citep{kungurtsev2019distributed} studied a variant of D-SGD assuming agents' private loss functions are obtained as the maximum of the set of smooth loss functions and have bounded Lipschitz constant. This assumption captures many problems of interest but does not permit, for example, $\ell_2$ regularization. 
In contrast, centralized SGD for \emph{nonsmooth nonconvex} optimization was studied in the recent paper \citep{davis2018stochastic}, which demonstrated convergence of SGD to critical points under very broad assumptions that include nonsmooth neural network architectures and typical regularization schemes. The current paper demonstrates convergence of D-SGD to critical points under similar assumptions to \citep{davis2018stochastic}, with the exception that, due to the distributed setting, we also require Assumption \ref{a:eigvec-continuity} (See Theorem \ref{nonsmooth-conv-to-cp} above).\footnote{Aside from Assumption \ref{a:eigvec-continuity}, our assumptions are nearly identical to those made in \citep{davis2018stochastic}, modulo minor differences; e.g., \citep{davis2018stochastic} assumes iterates remain bounded, while we explicitly assume loss functions to be coercive (Assumption \ref{a:coercive-sorta-nonsmooth}) in order to ensure bounded iterates (via Lemma \ref{lemma:compact-set}).}  
As discussed above, we reiterate this additional assumption is quite weak.

In a companion paper \citep{SMKP-TAC2020} (see also \citep{swenson2019allerton}), we study the problem of nonconvergence to saddle points for (continuous-time) DGF and establish a stable manifold theorem for DGF. The stable manifold for DGF plays a key role in the analysis of saddle point avoidance for D-SGD in the present paper. 
We also remark that a distributed gradient-based algorithms for computing global minima have recently been considered in
\citep{swenson2019CDC,swenson2019ICASSP}. In these algorithms, noise is deliberately added in order to escape local minima and seek out global minima.






\begin{center}
\begin{figure}
\begin{center}
\includegraphics[width=8cm]{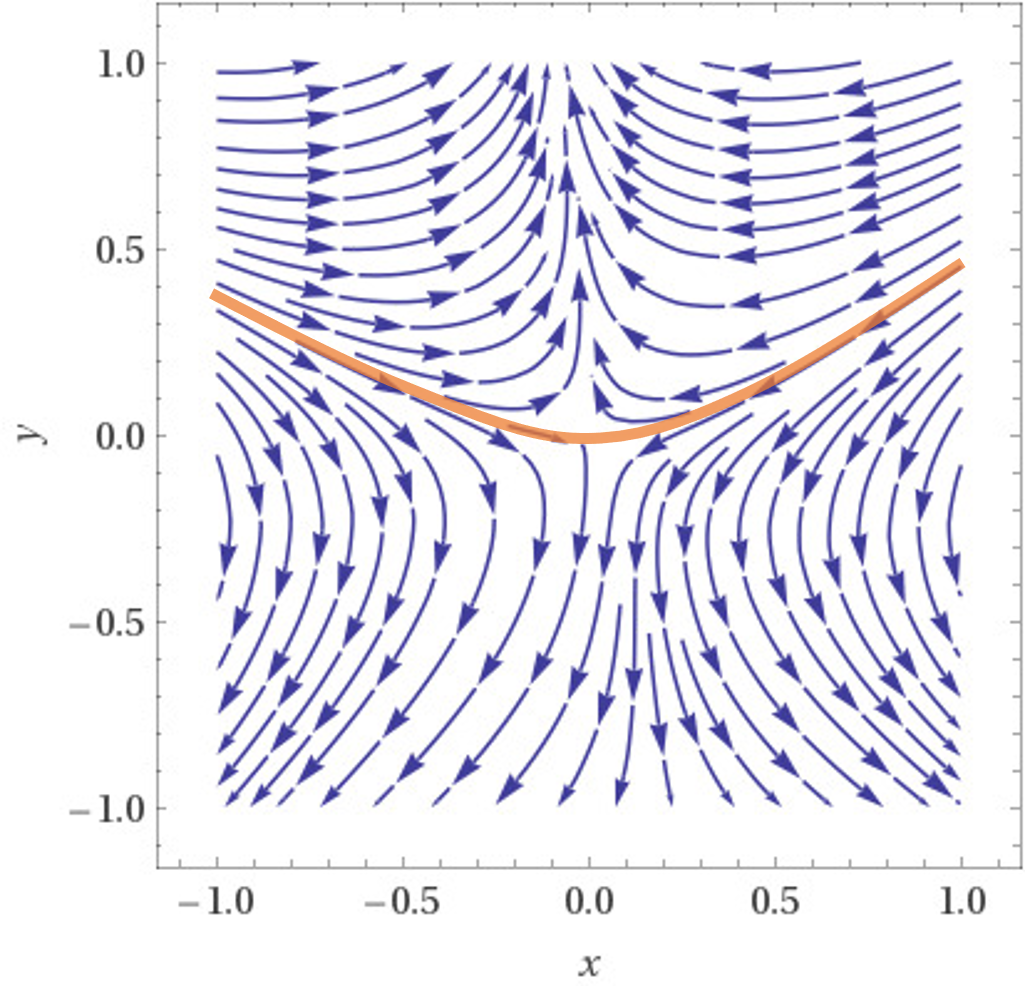}
\end{center}
\caption{In orange, the stable manifold for continuous GD applied to a fixed (autonomous) function over $\R^2$ with a saddle point at the origin. 
The proof that D-SGD avoids saddle points (Theorems \ref{thrm:DT-nonconvergence} and \ref{thrm:DT-nonconvergence-gen}) is intricate, but 
the underlying intuition is simple and is captured in this figure. First, we will see that D-SGD \eqref{dynamics_DT} (or \eqref{dynamics_DT-DI}) approximates some continuous gradient flow as $k\to\infty$. The underlying gradient flow will be referred to as DGF. We compute the \emph{stable manifold} for DGF near saddle points---a low-dimensional surface passing through the saddle point. The stable manifold for DGF is the key object in our analysis: It is precisely the set that D-SGD is repelled from (asymptotically, and in expectation). 
Utilizing this property, we prove that D-SGD avoids saddle points. The figure visualizes only classical centralized GD in two dimensions. In D-SGD, and in higher dimensions, the stable manifold is more complex (and time varying), but the basic idea is the same. This technique originated in \citep{pemantle1990nonconv}.}
\label{fig:stable-man1}
\end{figure}
\end{center}

\subsection{High-Level Discussion of Proof Techniques} \label{sec:high-level-pf-discussion}

Our approach for proving the main results will be to study a differential equation (or inclusion) that approximates the dynamics \eqref{dynamics_DT-DI} asymptotically. A standard technique in the ODE method of stochastic approximation is to relate the limit set of a discrete-time process to the limit set of its continuous-time analog. Typically, one shows that the limit set of the discrete-time process is contained in the limit set of the continuous-time analog. Proving convergence of the discrete-time process to some set then reduces to simply proving convergence of the continuous-time analog to the set. This is the approach that we will take to proving convergence to critical points. In particular, we will leverage the results of \citep{davis2018stochastic} which provides powerful tools for analyzing processes when the continuous-time analog is a differential inclusion.

Proving \emph{nonconvergence} to saddle points requires a different tack. Showing nonconvergence of the continuous-time analog to a point is not sufficient to ensure that the discrete-time process will not converge to it. Instead, we wish to identify some underlying structure from which the discrete-time dynamics are ``repelled.'' Near a saddle point, one can show the existence of a stable manifold for gradient flows \citep{chicone2006ordinary}. 
In \citep{pemantle1990nonconv} it was shown that for autonomous systems (ODEs where the right-hand side does not depend on time) discrete-time dynamics are asymptotically repelled from the stable manifold of the underlying continuous-time system. We will prove our results following a similar strategy. As noted earlier, because D-SGD is nonautonomous, classical theory does not apply. In \citep{SMKP-TAC2020} the existence of stable-manifolds was proved for the continuous-time variant of D-SGD. In this paper, we will show that D-SGD is repelled from the stable manifold of its continuous-time analog. This will imply nonconvergence to the saddle point, almost surely.

\subsection{Roadmap of Proof Strategy} \label{sec:pf-strategy}
Our approach to proving the main results (Sections \ref{sec:DT main results}--\ref{sec:nonsmooth-extension}) will be as follows. First, we note that we will focus on proving results in the nonsmooth setting, as results in the smooth setting follow as a special case. Thus, we will prove Theorems \ref{nonsmooth-conv-to-cp} and \ref{thrm:DT-nonconvergence-gen}. 

Next, we observe that the problem of minimizing \eqref{eqn:f-def} in a distributed setting is a special case of a general subspace-constrained optimization problem. (This observation has been made in several other recent papers, e.g., \citep{hong2018gradient}.) Rather than focus only on the particular setting of distributed gradient descent, we will study the broader problem of subspace-constrained optimization, and we will prove our main results about D-SGD as direct corollaries to results in this general framework.
We emphasize that the move to a more general framework will not come at the cost of a more complex presentation. 
The effect is the opposite---the general framework allows us to dispense with cumbersome notation associated with distributed consensus processes. The D-SGD dynamics \eqref{dynamics_DT-DI} simply become the gradient descent dynamics for an objective function plus a quadratic penalty. 
Proofs are simplified and intuition is more transparent. 

In Section \ref{sec:general-setup} we introduce the general subspace-constrained optimization problem. We also introduce dynamics for solving said optimization problem that generalize \eqref{dynamics_DT-DI}. 

In Section \ref{sec:conv-to-CP} we prove convergence to critical points (i.e., Theorem \ref{nonsmooth-conv-to-cp}). Convergence to critical points can be studied by approximating D-SGD with an appropriate gradient flow system. We will prove our results using the so-called ODE approach to stochastic approximation \citep{benaim1996dynamical}. More precisely, in order to prove results in the context of nonsmooth loss functions, we will study differential inclusions, rather than ODEs.
A key tool in this analysis will be results developed in \citep{davis2018stochastic}.

We will then turn our attention to the problem of showing nonconvergence to saddle points. Our approach to addressing this problem will again rely on approximating D-SGD with a continuous-time gradient flow, which we refer to as DGF.
However, unlike the case of studying convergence to critical points, there are no standard tools to be applied in this case. To address this problem, we first consider the DGF ODE.\footnote{In fact, we will consider a continuous-time version of our generalized dynamics introduced in Section \ref{sec:gen-framework}. But for simplicity we refer to this as continuous-time DGF here.} In \citep{SMKP-TAC2020} it was shown that a stable manifold exists for DGF near regular saddle points. In classical ODE theory it is well known that the stable manifold is a Lyapunov unstable set. The same holds for the stable manifold for DGF. As a consequence, D-SGD will be shown to be ``repelled'' from from the DGF stable manifold. Intuitively, leveraging this property we will be able to show almost sure nonconvergence to saddle points. See Figure \ref{fig:stable-man1} for further intuition and a low-dimensional visualization of this idea for classical GD. 

More precisely, In Section \ref{sec:DT-non-conv} we state our general saddle-point nonconvergence result for the subspace-constrained setting (Theorem \ref{thrm:saddle-instability}) and briefly review how this implies Theorem \ref{thrm:DT-nonconvergence-gen}.

In Section \ref{sec:CT-stable-manifold} we review the stable-manifold theorem for DGF from \citep{SMKP-TAC2020}. Here we will also prove an important smoothness property of the stable manifold (Lemma \ref{lemma:manifold-C2}) which will be necessary in proving nonconvergence to saddle points.

In Section \ref{sec:D-SGD-geometry} we further characterize the stable manifold for DGF and prove a key inequality (see \eqref{eq:eta-ineq} and Proposition \ref{prop:eta-properties}). Informally, this inequality states that D-SGD is repelled from the stable manifold. We note that this is the most technically involved aspect of the paper. 

Finally, in Section \ref{sec:stochastic-analysis-D-SGD} we perform the stochastic analysis of D-SGD and formally prove nonconvergence to saddle points (Theorem \ref{thrm:DT-nonconvergence-gen}). Importantly, we note that our approach to proving Theorem \ref{thrm:DT-nonconvergence-gen} builds upon the techniques developed in \citep{pemantle1990nonconv} to characterize nonconvergence to unstable points in autonomous dynamical systems. However, in studying D-SGD we are required to study nonautonomous dynamics which makes the problem substantially more challenging.




\subsection{Notation} \label{sec:prelims}
Throughout the paper, we use bold face letters, e.g., $\vx(t)$ to refer to continuous-time processes where $t\in\R$ is the time index, and we use non-bold letters, e.g., $x(k)$ to refer to discrete-time processes, where $k$ is a positive integer. When referring to weight parameters such as $\alpha_t$ and $\beta_t$, in order to reduce clutter we will place the time index in the subscript (since there is no need to specify an agent $n$ associated with these quantities). 
We say that $g\in C^r(\R^m; \R^n)$, for integer $r\geq 1$, if $g:\R^m\to\R^n$ is a function that is $r$-times continuously differentiable. When the domain and codomain are clear from the context, we simply use the shorthand $g\in C^r$ or say $g$ is $C^r$.
If $g$ is $C^1$, we use the notation $D[g,x]$ to denote the derivative of $g$ at the point $x$. Treating $D[g,x]:\R^m\to\R^n$ as a linear operator, we use the notation $D[g,x](y)$, $y\in \R^m$ to indicate the action of $D[g,x]$ on $y$.
When the meaning is clear from the context, we will sometimes use the shorthand $Dg(x)$ to denote $D[g,x]$, and treat $D[g,x]$ and $Dg(x)$ as $n\times m$ matrices. When $g\in C^2$, then we use $D^2[g,x]$ to indicate the second derivative of $g$ at $x$, where $D^2[g,x]:\R^m\times\R^m\to\R^n$ is a bilinear operator, and we use the notation $D^2[g,x](y,z)$ to indicate its action on inputs $y,z\in\R^m$.
Moreover, in the case that $g:\R^n\to\R$ is $C^2$, we often use the standard notation $\nabla g$ and $\Hess g$ to refer to the gradient and Hessian of $g$ respectively.
Given functions $g,h:\R\to\R$ we say that $g(t) = \Theta(h(t))$ if there exist constants $a,b>0$ such that $a h(t) \leq g(t) \leq bh(t)$ for all $t$ sufficiently large. 

At a few points in the manuscript we move to Einstein notation, in which any repetition in subscripts denote an implied summation over that index: i.e. $A_{ij} b_j = \sum_j A_{ij} b_j$. This convention condenses tensor operations, and is particularly useful for chain rule computations for vector fields. As this notation is less standard in the optimization literature, we only use it in specific situations where matrix notation would become cumbersome, e.g. Lemma \ref{lemma:manifold-C2}.

We will use $\|\cdot\|$ to denote the standard Euclidean norm.  Given a set $S\subset \R^d$ and point $x\in \R^d$, we let $\dist(x,S) := \inf_{y\in S} \|x-y\|$ and for $\delta>0$ we let $B_\delta(S) := \{x: \dist(x,S)<\delta\}$. When we say $x(k)\to S$ as $k\to\infty$, we mean that $\lim_{k\to\infty} \dist(x(k),S)=0$. Given $a,b\in\R$, $a\wedge b$ is the minimum of $a$ and $b$. $A \otimes B$ indicates the Kronecker product of matrices $A$ and $B$.  Given a matrix $A\in \R^{d\times d}$, $\diag(A)$ is the $d$-dimensional vector containing the diagonal entries of $A$. In an abuse of notation, given a vector $v\in \R^d$, we also use $\diag(v)$ to denote the $d\times d$ diagonal matrix with entries of $v$ on the diagonal. 

Given a graph $G=(V,E)$, the set of vertices $V=\{1,\ldots,N\}$ will be used to denote the set of agents and an edge $(i,j)\in E$ will denote the ability of two agents to exchange information. In this paper we will assume $G$ is undirected, meaning that $(i,j) \in E$ implies that $(j,i) \in E$. We let $\Omega_n$ denote the set of neighbors of agent $n$, namely $\Omega_n = \{i \in 1 \dots N : i \neq n, (i,n) \in E\}$, and we let $d_n =|\Omega_n|$.
The graph Laplacian is given by the $N\times N$ matrix $L = D-A$, where $D=\diag(d_1,\ldots,d_N)$ is the degree matrix and $A=(a_{ij})$ is the adjancency matrix defined by $a_{ij} = 1$ if $(i,j)\in E$ and $a_{ij}=0$ otherwise. Further details on spectral graph theory can be found in \citep{chung1997spectral}.

Suppose that $F:\R^d\to\R$ is of class $C^1$, and consider the general gradient-descent differential equation
\begin{equation} \label{eqn:generic-ode}
\dot \vx = -\nabla F(\vx),
\end{equation}
where $\vx:\R\to\R^d$ and $\dot \vx$ denotes $\ddt \vx(t)$. We say $\vx$ is a solution to \eqref{eqn:generic-ode} with initial condition $x_0$ at time $t_0$ if $\vx$ is $C^1$, satisfies $\vx(t_0) = x_0$, and satisfies \eqref{eqn:generic-ode} for all $t\geq t_0$.

It is known that the generalized gradient (Definition \ref{def:gen-grad}) is upper semicontinuous when the function in question is locally Lipschitz continuous \citep{clarke1990optimization}. This concept will be important in the subsequent analysis and is formally defined as follows.
\begin{definition} \label{def:upper-semicont}
A set-valued function $G:\R^m\to\R^m$ is said to be \emph{upper semicontinuous} at $x$ if for any $\e>0$ there exists a $\delta > 0$ such that for all $y\in B_\delta(x)$, $G(y)\subset B_{\e}(G(x))$
\end{definition}

We will consider recursive stochastic processes $\{y(k)\}_{k=1}^\infty$ of the form $y(k+1) = y(k) + G(y(k),\xi(k+1),k)$, where $G:\R^d\times\R^d\times \R$, $\xi(k+1)$ denotes a noise term and $k$ denotes the iteration number. We use $\calF_k = \sigma(\{x(j),\xi(j)\}_{j=1}^k)$ to denote the filtration representing the information available at iteration $k$.

A list of shorthand symbols commonly used throughout the paper (e.g., $N=$ number of agents) can be  found in the appendix.


\section{Generalized Setup: Subspace-Constrained Optimization} \label{sec:general-setup}
The problem of minimizing \eqref{eqn:f-def} in a distributed setting is equivalent to the subspace-constrained optimization problem
\begin{equation} \label{eq:dist-opt-prob}
\min_{\substack{x_n\in \R^d\\ n=1,\ldots,N\\ x_1 = x_2 = \cdots = x_N }} \sum_{n=1}^N f_n(x_n).
\end{equation}
Rather than focus on the particular problem \eqref{eq:dist-opt-prob}, we will study the broader class of all \emph{subspace-constrained optimization problems}, of which \eqref{eq:dist-opt-prob} is a special case. There are several advantages to taking this approach. Notation is simplified, proofs and intuition are more transparent, and results apply in a more general context.

In this section, we introduce a simple subspace-constrained optimization problem (generalizing \eqref{eq:dist-opt-prob}) that will be considered in the remainder of the paper. We then set up optimization dynamics extending \eqref{dynamics_DT-DI} which address the general subspace-constrained optimization problem.

The general optimization problem and dynamics will be set up in Section \ref{sec:gen-framework}.
However, before setting up the general framework, we will first briefly review continuous-time DGF and discuss a key time change operation in Section \ref{sec:time-changes}.
In particular, after a time-change, D-SGD and DGF will admit a clean and intuitive interpretation in terms of gradient descent with respect to a penalty function. This interpretation will be described in Section \ref{sec:gen-framework}.

\subsection{Continuous-Time DGF and Time Changes} \label{sec:time-changes}
The approach we will take for analyzing \eqref{dynamics_DT} throughout the paper will be to approximate the algorithm with a continuous-time ODE that is easier to study (i.e., DGF). 
Let $\vx_n(t)$ represent agent $n$'s estimate of a minimizer of \eqref{eqn:f-def} at time $t\in[0,\infty)$. We define 
DGF agentwise by the ODE
\begin{equation} \label{dynamics_CT}
\dot \vx_n(t) = \beta_t \sum_{\ell \in \Omega_n} (\vx_\ell(t) -\vx_n(t)) - \alpha_t \grad f_n(\vx_n(t))
\end{equation}
for $n=1,\ldots,N$, where $t\mapsto \alpha_t \in (0,1]$ and $t\mapsto \beta_t \in (0,1]$ are weight parameters. 

The ODE \eqref{dynamics_CT} may be expressed compactly as
\begin{equation} \label{eq:ODE-vec-form}
\dot \vx = \beta_t(L\otimes I_{d})\vx - \alpha_t(\nabla f_n(\vx_n))_{n=1}^N,
\end{equation}
where we let $\vx:\R\to\R^{Nd}$ be the vectorization $\vx := (\vx_1,\ldots,\vx_N)$, where $\vx_n:\R\to\R^d$ represents the state of agent $n$, and, as before, we assume $\alpha_t = o(\beta_t)$. It will often be convenient to study this ODE under a time change. In particular, assuming $\alpha_t>0$ for $t\geq 0$, set $S(t) = \int_{0}^t \alpha_r \,dr$ and let $T(\tau)$ denote the inverse of $S(t)$, so that $T(S(t)) = t$. Letting $\vy(\tau) = \vx(T(\tau))$ we have
\begin{equation} \label{eq:ODE-alt-form1}
\frac{d}{d\tau} \vy(\tau) = \gamma_\tau(L\otimes I_{d})\vy(\tau) - (\nabla f_n(\vy(\tau)))_{n=1}^N,
\end{equation}
where $\gamma_\tau = \frac{\beta_{T(\tau)}}{\alpha_{T(\tau)}} \to \infty$ as $\tau\to\infty$. Likewise, if we set $S(t) = \int_{0}^t \beta_r \,dr$ and let $T(\tau)$ denote the inverse of $S(t)$ we have
\begin{equation} \label{eq:ODE-alt-form2}
\dot \vy(\tau) = (L\otimes I_{d})\vy(\tau) - \tilde \gamma_\tau(\nabla f_n(\vy(\tau)))_{n=1}^N,
\end{equation}
where $\tilde \gamma_\tau = \frac{\alpha_{T(\tau)}}{\beta_{T(\tau)}} \to 0$ as $\tau\to\infty$. Thus, processes of the form \eqref{eq:ODE-alt-form1} or \eqref{eq:ODE-alt-form2}, with $\gamma_t \to \infty$ or $\tilde \gamma_t\to 0$ respectively, generalize dynamics of the form \eqref{dynamics_CT}.
When convenient, we will study \eqref{eq:ODE-alt-form1} or \eqref{eq:ODE-alt-form2} (with associated parameter $\gamma_\tau$ or $\tilde \gamma_\tau$) in lieu of \eqref{dynamics_CT}. 

\subsection{Subspace-Constrained Optimization Framework} \label{sec:gen-framework}
Consider the optimization problem
\begin{align} \tag{P.1} \label{eq:opt-prob}
\min_{x\in R^M} & \quad h(x)\\
\text{subject to} & \quad  x^\T Q x = 0,
\end{align}
where $h:\R^M\to\R$ is a $C^1$ function and $Q\in\R^{M\times M}$ is a positive semidefinite matrix. We will often denote the constraint set by
\begin{equation} \label{eq:constraint-def}
\C \coloneqq \left\{x\in \R^M:\,x^\T Qx = 0\right\}.
\end{equation}
Since $Q$ is positive semidefinite, $\C$ is precisely $\{x: Qx = 0\}$, i.e., the nullspace of $Q$; we write the constraint in its quadratic form because we will solve this problem using a penalization approach that connects directly with the quadratic form.
In the remainder of the paper we will focus on algorithms for computing local minima of \eqref{eq:opt-prob}.

\bigskip
\noindent \textbf{Continuous-Time Dynamics}. Suppose $h$ is $C^1$ and consider the following continuous-time dynamical system for solving \eqref{eq:opt-prob}:\footnote{The behavior of DGF with nonsmooth loss functions was considered in \citep{SMKP-TAC2020}. In this paper, when we study DGF it will be in the neighborhood of saddle points where we will always assume $h$ is smooth. Thus, when studying the descent process \eqref{eq:ODE1} in this paper we will always consider the gradient $\nabla h$, not the generalized gradient $\partial h$. However, when discussing discrete-time processes below, we will consider nonsmooth loss functions and generalized gradients.}
\begin{equation}\label{eq:ODE1}
\dot \vx = - \gamma_t Q\vx -\nabla h (\vx) ,
\end{equation}
where the weight $\gamma_t\to\infty$.
Note that these may be viewed as the gradient descent dynamics associated with the (time-varying) function $x\mapsto h(x) + \gamma_t x^\T  Q x$, i.e., $$
\dot \vx = -\nabla_x \left(h(\vx) + \gamma_t\vx^\T  Q\vx\right).
$$
The term $\gamma_t x^\T  Q x$ may be thought of as a quadratic penalty term that punishes deviations from $\C$ with increasing severity as $t\to\infty$.

\bigskip
\noindent \textbf{Discrete-Time Dynamics}. Suppose $h$ is locally Lipschitz continuous, and consider the following discrete-time dynamics for solving \eqref{eq:opt-prob}:
\begin{equation} \label{dynamics_DT3}
x(k+1) = x(k) - \alpha_k\Big(v(k) + \gamma_kQx(k) + \xi(k+1) \Big),
\end{equation}
where $v(k) \in \partial h\left(x(k)\right)$, $\gamma_k\to\infty$, $\alpha_k\to 0$,
and $\{\xi(k)\}_{k\geq 1}$ represents noise given by a Martingale difference sequence (i.e., conditionally zero-mean noise as in Assumption \ref{a:grab-bag}). If $h\in C^1$, the recursion \eqref{dynamics_DT3} may be viewed as a perturbed discretization of the ODE \eqref{eq:ODE1} with (diminishing) step size $\alpha_k$, in the sense that, the expected update satisfies
$$
\E(x(k+1)\vert\calF_{k}) = x(k) + \alpha_k\Big(- \gamma_kQx(k)-\nabla h\left(x(k)\right)  \Big).
$$
Letting $\beta_k = \alpha_k\gamma_k$, using similar reasoning to the continuous-time case, the discrete-time D-SGD process  \eqref{dynamics_DT-DI} may be seen as a special case of \eqref{dynamics_DT3}, where we use $\gamma_k$ to be consistent with the penalty interpretation of Section \ref{sec:time-changes}.

\subsection{Interpreting Results: From General Framework to D-SGD Framework} \label{sec:gen-to-DGD}
All of the results in the remainder of the paper will be proved in the context of Problem \eqref{eq:opt-prob} and optimization dynamics \eqref{dynamics_DT3}.  
Under Assumptions \ref{a:grab-bag}--\ref{a:C3}, the problem of optimizing \eqref{eqn:f-def} under the D-SGD dynamics  \eqref{dynamics_DT-DI} is a special case of this general framework in which we let the dimension be $M= Nd$, the state $x\in \R^{Nd}$ is given by the vectorization of all agents' states $x=(x_n)_{n=1}^N$, the objective function is given by $h(x) = \sum_{n=1}^Nf_n(x_n)$, and the penalty term is generated by the matrix $Q = (L\otimes I_d)$, where $I_d$ is the $d\times d$ identity matrix and $L$ denotes the graph Laplacian of $G$. Under Assumption  \ref{a:grab-bag}, the nullspace of $(L\otimes I_d)$ is the \emph{consensus subspace} $\{(x_n)_{n=1}^N \in \R^{Nd} :x_1=\cdots = x_N\}$. Thus, the constraint space $\C$ in \eqref{eq:constraint-def} is given by the consensus subspace in the context of D-SGD. 

We now introduce some conventions that will simplify presentation.
Unless otherwise noted, without loss of generality
assume the coordinate system is rotated so that the constraint space is given by
\begin{equation} \label{eq:C-rotation}
\calC = \{x\in\R^M: x_{d+1}=\cdots =x_M = 0\},
\end{equation}
where we let $d=\dim\C$. 
Given a vector $x\in \R^M$, we will use the decomposition
\begin{equation} \label{eq:x-decomposition1}
x = (x_c,x_{nc}),
\end{equation}
$x_c\in \R^d$, $x_{nc}\in\R^{M-d}$, where the subscripts indicate the ``constraint'' and ``not constraint'' components respectively.
In a slight abuse of notation, given $x_c\in \R^{d}$ we let
$$
h\vert_\C(x_c) := h(x_c,0).
$$
Given $x\in \R^M$ define $\partial_{x_c} h(x) := \{z\in \R^{d}: (z,y)\in \partial h(x), ~\mbox{ for some } y\in \R^{M-d}\}.$
In a slight abuse of terminology, we say that $x^*=(x_c^*,x_{nc}^*)\in\R^M$ is a critical point of $h\vert_\C$ if $0\in \partial_{x_c} h(x_c^*,0)$, or equivalently, if $0\in \partial h\vert_\C(x_c^*)$.


In the D-SGD framework of Section \ref{sec:main-results}, convergence of D-SGD to consensus corresponds here to convergence of \eqref{dynamics_DT3} to $\calC$. Likewise, critical points of \eqref{eqn:f-def} correspond to critical points of $h\vert_\calC$.

In order to show the main results in Section \ref{sec:main-results} (Theorems \ref{nonsmooth-conv-to-cp} and \ref{thrm:DT-nonconvergence-gen}), we will show that the following hold with probability 1:  \eqref{dynamics_DT3} converges to $\calC$ (see Theorem \ref{thrm:conv-to-cp-DT-h}); \eqref{dynamics_DT3} converges to critical points of $h\vert_{\C}$ (see Theorem \ref{thrm:conv-to-cp-DT-h}); \eqref{dynamics_DT3} does not converge to regular saddle points of $h\vert_\C$ (see Theorem \ref{thrm:saddle-instability}).

\section{Convergence to Critical Points} \label{sec:conv-to-CP}
In this section we show that \eqref{dynamics_DT3} converges to critical points of $h\vert_\C$.
We begin by presenting several assumptions. We emphasize that the assumptions made through the remainder the paper are distinct from all assumptions made thus far in that all subsequent assumptions apply to the general subspace-constrained optimization framework of \eqref{eq:opt-prob} and \eqref{dynamics_DT3}. We will make sufficiently broad assumptions so that the D-SGD framework of Section \ref{sec:main-results} is special case of the general framework. 
To distinguish the assumptions pertaining to each framework, the assumptions for the distributed framework are numbered A.1, A.2, etc., while the assumptions for the general subspace-constrained framework will be numbered 
B.1, B.2, etc.


\begin{newassumption} \label{a:h-loc-lip}
$h$ is locally Lipschitz continuous.
\end{newassumption}

\begin{newassumption} 
\label{a:coercive-h}
For some radius $R>0$ and constants $C_1,C_2>0$ the following conditions hold:
\begin{equation} \label{eq:lemma-grad-alignment}
\langle x, v\rangle \geq C_1\|x\|\|v\| \quad\quad \mbox{ and }  \quad\quad \|v\| \leq C_2\|x\|,
\end{equation}
for all $v\in \partial h(x)$ and $\|x\|\geq R$. 
\end{newassumption}
Note that the first part of \eqref{eq:lemma-grad-alignment} simply ensures that asymptotically, the negative gradient points inwards. The second part of \eqref{eq:lemma-grad-alignment} ensures that the function is asymptotically subquadratic in the sense that (selections of) the generalized gradient grow more slowly than the gradient of some quadratic function.

\begin{newassumption} \label{a:cp-meas-0-h}
Let $\CP_{h\vert_\C}$ be the set
of critical points of $h\vert_\C$. Assume the set $\R\backslash h\vert_\C(\CP_{h\vert_\C})$is dense in $\R$. 
\end{newassumption}
\begin{newassumption} \label{a:chain-rule-h}
For any absolutely continuous function $\vz:[0,\infty)\to \R^{d}$, 
$h\vert_\C$ satisfies the chain rule
$$
\ddt h\vert_\C(\vz(t)) = \big\langle v, \ddt \vz(t) \big\rangle, 
$$
for some $v\in \partial f(\vz(t))$, and almost all $t\geq 0$. 
\end{newassumption}
\begin{newassumption} 
\label{a:Q-PSD1}
  $Q\in \R^{M\times M}$ is positive semidefinite with at least one zero eigenvalue.
\end{newassumption}

\begin{newassumption} \label{a:weights-h}
$\alpha_k = \Theta(k^{-\tau_\alpha})$ and $\gamma_k = \Theta(k^{\tau_\gamma})$ where  $\frac{1}{2}< \tau_\gamma <\tau_\alpha \leq 1$, $\alpha_k,\gamma_k\not=0$.
\end{newassumption}
Finally, we will assume the following regarding the noise process $\{\xi(k)\}_{k\geq 1}$ in \eqref{dynamics_DT3}. We recall that we assume a coordinate rotation so $\calC = \{x\in\R^M: x_{d+1}=\ldots = x_M = 0\}$, and let $\xi(k)$ be decomposed as $\xi(k) = (\xi_c(k),\xi_{nc}(k))$.
\begin{newassumption} \label{a:noise2}
For all $k\geq 1$, $\E(\xi_c(k)\vert \calF_{k-1}) = 0$ and $\E(\|\xi(k)\|^2\vert \calF_k)<B$ for some $B>0$.
\end{newassumption}

We will prove the following result.
\begin{theorem} \label{thrm:conv-to-cp-DT-h}
Let $\{x(k)\}_{k\geq 1}$ be a solution to \eqref{dynamics_DT3} and suppose that Assumptions \ref{a:h-loc-lip}--\ref{a:noise2} hold. Then,
\begin{itemize}
    \item [(i)] $x(k) \to \C$ as $k\to\infty$.
    \item [(ii)] $x(k)$ converges to the set of critical points of $h\vert_{\C}$ as $k\to\infty$.
\end{itemize}
\end{theorem}
\begin{remark} [Proving Theorem \ref{nonsmooth-conv-to-cp}]
By Section \ref{sec:gen-to-DGD} we see that optimizing \eqref{eqn:f-def} using D-SGD is a special case of optimizing \eqref{eq:opt-prob} using \eqref{dynamics_DT3}. Note that Assumption \ref{a:grab-bag}, part 1 implies that $L\otimes I_d$ has exactly $d$ zero eigenvalues, and thus is a special case of Assumption \ref{a:Q-PSD1}. From here it is straightforward to verify that the remaining Assumptions hold and see that Theorem \ref{thrm:conv-to-cp-DT-h} implies Theorem \ref{nonsmooth-conv-to-cp} (which, in turn, implies Theorem \ref{thrm:discrete-conv}).
\end{remark}

\begin{remark}[On Assumption \ref{a:chain-rule-h}] \label{remark:chain-rule}
Suppose $g:\R^d\to\R$ is some smooth function and $\vx(t)$ a gradient flow trajectory for $g$, i.e., $\ddt \vx(t) = \nabla g(\vx(t))$. Then by the chain rule we have
$$
\ddt g(\vx(t)) = \big\langle \nabla g(\vx(t), \ddt \vx(t) \big\rangle = -\|\nabla g(\vx(t))\|^2.
$$
This condition allows us to guarantee that $\vx(t)$ \emph{descends} $g$, i.e., when $\vx(t)$ is not at a critical point of $g$ we have $\ddt g(\vx(t)) < 0$. This property is crucial for Lyapunov-function based analysis. 
When dealing with nonsmooth loss functions, ensuring that gradient flow trajectories possess the descent property is nontrivial. Assumption \ref{a:chain-rule-h} (respectively, Assumption \ref{a:chain-rule}) ensures that we have a descent property for gradient flows of $h\vert_\C$ (respectively, $f$). 
\end{remark}

The proof of Theorem \ref{thrm:conv-to-cp-DT-h} will be given in Sections \ref{sec:DT-conv-to-C}--\ref{sec:DT-conv-to-cp} below. In particular, the theorem will follow immediately from Lemmas \ref{lemma:DT-conv-to-C} and \ref{lemma:dt-conv-to-cp-h} below.

\subsection{Convergence to the Constraint Space} \label{sec:DT-conv-to-C}
We now prove that $x(k)$ converges to $\calC$. 
Note that under the coordinate change in \eqref{eq:C-rotation}, $Q$ has block diagonal form
\begin{equation} \label{eq:Q-diag-form}
Q = \begin{pmatrix}
0 & 0\\
0 & \widehat Q
\end{pmatrix}
\end{equation}
where $\widehat Q\in\R^{(M-d)\times (M-d)}$ is positive definite and here $0$ denotes a zero matrix of appropriate dimension.
Let $x(k)$ be decomposed as
\begin{equation} \label{eq:x-decomposition}
x(k) =
\begin{pmatrix}
x_c(k)\\
\xperp(k)
\end{pmatrix},
\end{equation}
where $x_c(k)\in \R^d$ and $\xperp(k) \in \R^{M-d}$. 

We begin with the following technical lemma. 
\begin{lemma} \label{lemma:technical-inner-prod}
Suppose Assumptions \ref{a:coercive-h}, \ref{a:Q-PSD1}, and \ref{a:weights-h} hold.
Then for all $k$  sufficiently large, $\|x\|\geq R$ and $v\in \partial h(x)$ we have
$$
\langle x- \frac{1}{2}\alpha_t (v - \gamma_k Qx),  v + \gamma_k Qx \rangle > 0.
$$ 
\end{lemma}
\begin{proof}
Throughout the proof we let $v\in \partial h(x)$.
We break the analysis into two cases. First, suppose that $Qx = 0$. Expanding the inner product we obtain
\begin{align}
    \langle x- \frac{1}{2}\alpha_k (v - \gamma_k Qx),  v + \gamma_k Qx \rangle & = \langle x,  v \rangle - \frac{\alpha_k}{2}\|v\|^2\\
    & \geq C_1\|x\|\|v\| - \frac{\alpha_k}{2}\|v\|^2\\
    & \geq \frac{C_1}{C_2}\|v\|^2 - \frac{\alpha_k}{2}\|v\|^2  > 0,
\end{align}
where the second line follows using the first part of \eqref{eq:lemma-grad-alignment}, the third line from the second part of \eqref{eq:lemma-grad-alignment}, and the last inequality holds for $t$ large as $\alpha_k\to 0$ as $k\to\infty$.

Suppose now that $Qx\not=0$. Let $\lambdaMin$ denote the magnitude of the smallest nonzero eigenvalue of $Q$ and let $\lambdaMax$ denote the magnitude of the largest eigenvalue of $Q$. Expanding the inner product and again employing \eqref{eq:lemma-grad-alignment} as above we obtain
\begin{align}
    \langle x- \frac{1}{2}\alpha_k (v - \gamma_k Qx),  v + &  \gamma_k  Qx  \rangle\\
    = & \langle x,  v + \gamma_k Qx \rangle \frac{\alpha_k}{2}\|v + \gamma_k Qx \|^2\\
    \geq & \langle x,v\rangle + \gamma_k\langle x,Q x\rangle 
    - \frac{\alpha_k}{2}\left( \|v\|^2 + \gamma_k^2\lambdaMax\|x\|^2 + 2\gamma_k\langle v,Qx\rangle \right)\\
    \geq & C_1\|x\|\|v\| + \gamma_k \lambdaMin\|x\|^2 - \frac{\alpha_k}{2}\|v\|^2 - \frac{\alpha_k\gamma_k^2}{2}\lambdaMax\|x\|^2 - \alpha_k \gamma_k\lambdaMax\|v\|\|x\|\\
    \geq & \frac{C_1}{C_2}\|v\|^2 + \gamma_k \lambdaMin\|x\|^2 - \frac{\alpha_k}{2}\|v\|^2 - \frac{\alpha_k\gamma_k^2}{2}\lambdaMax\|x\|^2 - C_2\alpha_k\gamma_k\lambdaMax \|x\|^2\\
    = & \left(\frac{C_1}{C_2} - \alpha_k \right)\|v\|^2 + \gamma_k\left(\lambdaMin - \frac{\alpha_k\gamma_k}{2}\lambdaMax - C_2\alpha_k\lambdaMax \right)\|x\|^2\\
    > & 0,
\end{align}
where the last inequality holds for $t$ sufficiently large as $\alpha_k\to 0$ and $\alpha_k\gamma_k\to 0$ as $k\to\infty$.
\end{proof}

The next lemma is a key result showing that iterates of \eqref{dynamics_DT3} remain some compact set.
\begin{lemma} \label{lemma:compact-set}
Let $\{x(k)\}_{k\geq 1}$ be a solution to \eqref{dynamics_DT3} and 
suppose that Assumptions  \ref{a:h-loc-lip}--\ref{a:coercive-h} and \ref{a:Q-PSD1}--\ref{a:noise2} hold. Then with probability 1, there exists a compact set $K\subset \R^M$ such that $x(k)\in K$ for all $k\geq 1$.
\end{lemma}
\begin{proof}
By our recursion relation \eqref{dynamics_DT3}, we have
\begin{align}
\|x(k+1)\|^2 
& =  \|x(k)\|^2 - 2\langle x(k) - \frac{1}{2}\alpha_k w(k),\alpha_k w(k) \rangle - 2\langle \alpha_k \xi(k+1), x(k) - \alpha_k w(k) \rangle + \alpha_k^2 \| \xi(k+1) \|^2,
\end{align}
where we let $w(k) = v(k) + \gamma_kQx(k)$, $v(k)\in \partial h(x(k))$. By Lemma \ref{lemma:technical-inner-prod}, for $\|x(k)\|>R$  we obtain
\[
\|x(k+1)\| \leq \sqrt{\|x(k)\|^2 - 2\langle \alpha_k \xi(k+1), x(k) - \alpha_k w(k) \rangle + \alpha_k^2 \| \xi(k+1) \|^2 }.
\]
This then gives
\[
\|x(k+1)\| \leq \|x(k)\| \sqrt{ 1 - 2\langle \alpha_k \xi(k+1), \frac{x(k) - \alpha_k w(k)}{\|x(k)\|^2} \rangle + \alpha_k^2 \frac{\| \xi(k+1) \|^2}{\|x(k)\|^2}}.
\]
Note that $\sqrt{y} \leq y$, for all $y\geq 0$. This follows by observing that $y\mapsto \sqrt{y}$ is concave and estimating with a first-order approximation at $y=1$. Using this in the display above we obtain
\begin{equation}\label{eqn:norm-bound}
\|x(k+1)\| \leq \|x(k)\| \left(1 - 2\langle \alpha_k \xi(k+1), \frac{x(k) - \alpha_k w(k)}{\|x(k)\|^2} \rangle + \alpha_k^2 \frac{\| \xi(k+1) \|^2}{2\|x(k)\|^2}\right).
\end{equation}
Consider the random variable
\[
z(k) = \max (\|x(k)\|,R) + \sum_{i=k}^\infty \alpha_i^2 \frac{\|\xi(i+1)\|^2}{2R}.
\]
As the $\xi(k)$ have bounded variance and the $\alpha_k$ are square summable, the second sum is almost surely finite \citep[Sec. 12.2]{williams1991probability}. Then, by \eqref{eqn:norm-bound}, we have that
\[
z(k+1) \leq z(k)  - 2\langle \alpha_k \xi(k+1), \frac{x(k) - \alpha_k w(k)}{\|x(k)\|} \rangle
\]
As the $\xi(k)$ are mean zero (Assumption \ref{a:noise2}), this implies that $\{z(k)\}_{k\geq 1}$ is a non-negative supermartingale. By Doob's supermartingale inequality \citep[Section 4.4]{durrett2005probability}, we see that $\P(\sup_{k\geq 1} z(k) \geq \lambda) \leq \frac{E(z(1))}{\lambda}$ for every $\lambda > 0$. Sending $\lambda \to \infty$ we see that $\P(\sup_{k\geq 1} z(k) = \infty\} = 0$. 
In turn this implies that $\sup_{k\geq 1}\|x(k)\|<\infty$, almost surely.
\end{proof}

The following result from \citep{kar2013distributed} will be useful.
\begin{lemma}[\citet{kar2013distributed}, Lemma 4.1]
\label{lemma-kar}
Let $\{z_k\}_{k\geq 1}$ be an $\R^+$ valued sequence satisfying
$$
z_{k+1} \leq (1-r_1(k))z_k + r_2(k),
$$
where $\{r_1(k)\}_{k\geq 1}$ and $\{r_2(k)\}_{k\geq 1}$ are deterministic sequences with
$$
\frac{a_1}{(k+1)^{\delta_1}} \leq r_1(k) \leq 1 \quad \mbox{ and } \quad r_2(k) \leq \frac{a_2}{(k+1)^{\delta_2}},
$$
and $a_1,a_2>0$, $0\leq \delta_1 <1$, $\delta_2>0$. Then, if $\delta_1<\delta_2$, $(k+1)^{\delta_0}z_k\to 0$ as $k\to\infty$ for all $0\leq \delta_0 <\delta_2-\delta_1$. 
\end{lemma}

The next result shows that \eqref{dynamics_DT3} converges to $\C$.
\begin{lemma}\label{lemma:DT-conv-to-C}
Let $\{x(k)\}_{k\geq 1}$ be a solution to \eqref{dynamics_DT3} and suppose that Assumptions \ref{a:h-loc-lip}--\ref{a:coercive-h} and \ref{a:Q-PSD1}--\ref{a:noise2} hold. Then $x(k)\to \C$ as $k\to\infty$, with probability 1.
\end{lemma}
\begin{proof}
Let $\C$ and $Q$ be as given in \eqref{eq:C-rotation} and \eqref{eq:Q-diag-form} and let $x(k)$ be decomposed as \eqref{eq:x-decomposition}.
By \eqref{dynamics_DT3} we have
\begin{equation} \label{eq:x-perp-update}
\xperp(k+1) = \xperp(k) - \alpha_k\gamma_k\widehat Q\xperp(k) + \alpha_k r(x(k)) + \alpha_k \xi_{nc}(k+1),
\end{equation}
where $r(x(k)) \in -\left[\partial h(x(k))\right]_{i=d+1}^M$, and $\xi_{nc}(k) = [\xi_i(k)]_{i=d+1}^M$.
By Lemma \ref{lemma:compact-set}, $x(k)$ remains in a compact set $K$ with probability 1. Using Assumption \ref{a:h-loc-lip}, let $L>0$ be the local Lipschitz constant for $h$ over the set $K$ so that $\frac{\|h(x) - h(y)\|}{\|x-y\|} \leq L$ for all $x,y\in K$, and, in particular, $\|\partial h(x)\|\leq L$ for all $x\in K$. 
By Assumption \ref{a:noise2}, we may choose the previous $L$ sufficiently large so that we also have $\|\xi_{nc}(k)\| < L$ for all $k$.
Let $\lambda_{\text{min}}$ be the smallest eigenvalue of $\widehat Q$ (where $\lambda_{\textup{min}}$ is necessarily positive, since, by construction, $\widehat Q$ is positive definite). From \eqref{eq:x-perp-update} we have 
$$
\|\xperp(k+1)\| \leq (1-\alpha_k\gamma_k\lambda_{\text{min}})\|\xperp(k)\| + \alpha_k2L.
$$
Invoking the step size requirement in Assumption \ref{a:weights-h}, Lemma \ref{lemma-kar} implies that $\|\xperp(k)\|\to 0$.
\end{proof}

\subsection{Convergence to Critical Points} \label{sec:DT-conv-to-cp}

We now prove that $x(k)$ converges to critical points of $h\vert_\C$. This result will follow as a simple consequence of the following result from \citep{davis2018stochastic} and the results of the previous section. In the statement of the theorem we will use the convention that for a (possibly set-valued) function $H$ mapping from $\R^{m_1}$ to $\R^{m_2}$, $H^{-1}$ denotes the preimage of $H$, i.e., $H^{-1}(z) = \{x\in \R^{m_1}: z\in H(x)\}$. 
\begin{theorem}[\citet{davis2018stochastic}, Theorem 3.2] \label{thrm:tame-functions}
Let $m\geq 1$ be an integer and suppose that $\{w_k\}_{k\geq 1}\subset \R^m$ is a sequence satisfying the recursive relationship
\begin{equation} \label{eq:tame-fn-recursion}
w_{k+1} = w_k + a_k\left(y_k + \eta_k \right),
\end{equation}
where $w_k,y_k,\eta_k\in \R^m$ and $a_k$ is a scalar.
Suppose that $G:\R^m\to\R^m$ is a set-valued map and that the following conditions hold:
\begin{enumerate}
    \item $\sup_{k} \|w_k\| < \infty$ and $\sup_k \|y_k\| < \infty$.
    \item $\{a_k\}_{k\geq 1}$ satisfies $a_k\geq 0$, $\sum_k a_k = \infty$ and $\sum_k a_k^2 < \infty$.
    \item $\lim_{n\to\infty} \sum_{k=1}^n a_k\eta_k$ exists, with probability 1.
    \item For any bounded increasing sequence $\{k_j\}_{j\geq 1}$ such that $w_{k_j}$ converges to a point $\bar w$ it holds that
    $$
    \lim_{n\to\infty} \textup{dist}\left(\frac{1}{n}\sum_{j=1}^n y_{k_j},G(\bar w) \right) = 0.
    $$
    \end{enumerate}
Suppose, moreover, that $\phi:\R^m\to\R$ is a candidate Lyapunov function satisfying the following conditions:
\begin{enumerate}
 \item [5.] For a dense set of values $r\in \R$, the intersection $\phi^{-1}(r)\cap G^{-1}(0)$ is empty. 
 \item [6.] Whenever $\vz:\R\to\R^m$ is a trajectory of the differential inclusion
 \begin{equation} \label{eq:z-diff-inclusion}
 \ddt \vz(t) \in G(\vz),
 \end{equation}
 and $0\not\in G(\vz(0))$, there exists a real $T>0$ satisfying 
 $$
 \phi(\vz(T) < \sup_{t\in [0,T]} \phi(\vz(t)) \leq \phi(\vz(0)). 
 $$
\end{enumerate}
Then every limit point of $\{w_k\}_{k\geq 1}$ lies in $G^{-1}(0)$ and $\phi(w_k)$ converges to a limit as $k\to\infty$. 
\end{theorem}

\begin{lemma} \label{lemma:dt-conv-to-cp-h}
Let $\{x(k)\}_{k\geq 1}$ be a solution to \eqref{dynamics_DT3}
and suppose that Assumptions \ref{a:h-loc-lip}--\ref{a:noise2} hold.
Then $x(k)$ converges to the set of critical points of $h\vert_{\C}$, almost surely.
\end{lemma}
\begin{proof}
The problem at hand fits the template of Theorem \ref{thrm:tame-functions} as follows. Let the dimension $m$ from Theorem \ref{thrm:tame-functions} be given by $m = \dim\C$. Let $x(k)$, $v(k)$, and $\xi(k)$ be as given in \eqref{dynamics_DT3}, and decompose each of these as in 
\eqref{eq:C-rotation}--\eqref{eq:x-decomposition1} to obtain $x_c(k)$, $v_c(k)$, and $\xi_c(k)$. Let $x_k$, $y_k$ and $\eta_k$ in \eqref{eq:tame-fn-recursion} be assigned as follows. Let $x_k = x_c(k)$; let $y_k =v_c(k)$; and let $\eta_k = \xi_c(k)$. 
For  $x_c\in \R^{\dim\C}$, let $\phi$ and $G$ from Theorem \ref{thrm:tame-functions} be assigned as
$\phi(x_c) = h\vert_\C(x_c)$  and 
$G(x_c) = \partial h\vert_\C(x_c)$. 
 

We now verify that the conditions of Theorem \ref{thrm:tame-functions} hold. Condition 1 holds as a consequence of Lemma \ref{lemma:compact-set}. Condition 2 holds by Assumption \ref{a:weights-h} where we let $a_k = \alpha_k$. To verify that Condition 3 holds, recall that if $X_1,X_2,\ldots$ are random variables with $\E X_i = 0$ and $\sum_{i=1}^\infty \textup{Var}(X_i) < \infty$, then $\sum_{i=1}^\infty X_i$ converges with probability 1  \citep[Sec. 1.8, Theorem (8.3)]{durrett2005probability}. It then follows that Condition 3 holds by Assumptions \ref{a:weights-h} and \ref{a:noise2}. That Condition 4 holds follows from the fact that $\partial h(\cdot)$ is upper semicontinuous (Definition \ref{def:upper-semicont}) and convex \citep{clarke1990optimization}, and the fact that $x_{nc}(k)\to 0$ (Lemma \ref{lemma:DT-conv-to-C}). Condition 5 holds by Assumption \ref{a:cp-meas-0-h}. To verify that Condition 6 holds, note that in the present context, the differential inclusion \eqref{eq:z-diff-inclusion} is given by
$$
\ddt \vz(t) \in \partial_{x_c} h(\vz(t)). 
$$
Suppose that $0\not\in \partial h\vert_\C(\vz(0))$. By upper semicontinuity of $\partial h\vert_\C(\cdot)$ we have that $0\not\in \partial h\vert_\C(z)$ for $z$ in an open neighborhood of $\vz(0)$. Using Assumption \ref{a:chain-rule-h}, it follows that $t\mapsto \hat h(\vz(t))$ is monotonically nonincreasing for $t\geq 0$ and for any $t> 0$ we have
$$
h\vert_\C(\vz(t)) < h\vert_\C(\vz(0)). 
$$
Thus, Condition 6 holds.
Having verified the conditions of Theorem \ref{thrm:tame-functions}, the result immediately follows. 
\end{proof}




\section{Non-Convergence to Saddle Points: General Result} \label{sec:DT-non-conv}





We now state our main result concerning nonconvergence to saddle points in the context of the subspace-constrained setup of Section \ref{sec:general-setup}. (In the context of D-SGD, the next result implies Theorem \ref{thrm:DT-nonconvergence-gen}).  Before stating the result we will require the following three additional assumptions. 
\begin{newassumption} \label{a:noise3}
For some constant $c_1>0$ there holds $$\E((\xi(k)^\T \theta)^+\vert \calF_{k-1}) \geq c_1,$$ for every unit vector $\theta \in \C$.
\end{newassumption}
In the next two assumptions, $x^*$ denotes some saddle point of interest. 
\begin{newassumption}\label{a:eigvec-continuity-h}
Assume that the eigenvectors of $\Hess h(x)$ are differentiable near $x^*$ in the sense that for each $x$ near $x^*$, there exists an orthonormal matrix $U(x)$ that diagonalizes $\Hess h(x)$ such that $x\mapsto U(x)$ is differentiable at $x^*$. 
\end{newassumption}
This assumption is used to make the analysis more tractable. In particular, it is needed to apply the stable manifold theorem. (See Section \ref{sec:CT-stable-manifold} and Remark \ref{remark:eigval-cont} below for more details.) The assumption is relatively mild an only needed to rule out pathological functions (see Example 19 in \citep{SMKP-TAC2020}). 
Finally, we assume that $h$ is slightly more regular in a neighborhood of $x^*$ as follows.
\begin{newassumption} \label{a:h-C3}
$h\in C^3$ in a neighborhood of $x^*$.
\end{newassumption}
 
\begin{theorem} \label{thrm:saddle-instability}
Let $\{x(k)\}_{k\geq 1}$ satisfy \eqref{dynamics_DT3}. 
Let  $x^*$ be a regular saddle point of $h\vert_\C$.
Suppose Assumptions \ref{a:h-loc-lip}--\ref{a:coercive-h}, \ref{a:Q-PSD1}--\ref{a:noise3} hold and \ref{a:eigvec-continuity-h}--\ref{a:h-C3} hold relative to $x^*$.
Then for any initialization,
$$\mathbb{P}\left(\lim_{k\to\infty} x(k) = x^*\right) = 0.$$
\end{theorem}

Theorem \ref{thrm:saddle-instability} is the focal point of the paper. It implies Theorem  \ref{thrm:DT-nonconvergence-gen} (and, of course, Theorem \ref{thrm:DT-nonconvergence}). See Remark \ref{remark:Thrm-nonconvergence} below for more details. 
The remainder of the paper will focus on proving Theorem \ref{thrm:saddle-instability}. In \citep{SMKP-TAC2020} it was shown that a time-varying stable manifold exists for DGF near saddle points. The key idea of the proof of Theorem \ref{thrm:saddle-instability} will be to show that solutions of \eqref{dynamics_DT3} are repelled from this stable manifold. 

The proof of Theorem \ref{thrm:saddle-instability} will proceed as follows: In Section \ref{sec:CT-stable-manifold} we will review the stable manifold theorem for DGF and will prove a key smoothness property of the stable manifold (Theorem \ref{lemma:manifold-C2}).
Next, in Section \ref{sec:D-SGD-geometry} we will consider the relationship of the discrete-time process \eqref{dynamics_DT3} to the stable manifold for DGF and establish a key inequality relating to the instability of the stable manifold (Proposition \ref{prop:eta-properties}, Property 4). Finally, in Section \ref{sec:stochastic-analysis-D-SGD} we will carry out the stochastic analysis of \eqref{dynamics_DT3} required to prove Theorem \ref{thrm:saddle-instability}. 

We emphasize that our approach to proving Theorem \ref{thrm:saddle-instability} is based upon the techniques developed in \citep{pemantle1990nonconv}. We note, however, that \citep{pemantle1990nonconv} studies \emph{autonomous} dynamical systems. The dynamics \eqref{dynamics_DT3} (and \eqref{eq:ODE1}) are \emph{non-autonomous}, and the approach used in \citep{pemantle1990nonconv} requires substantial modification to address the non-autonomous setting.

\begin{remark} [Proving Theorem \ref{thrm:DT-nonconvergence-gen}] \label{remark:Thrm-nonconvergence}
Recall that \eqref{dynamics_DT3} is a generalization of \eqref{dynamics_DT-DI} where $Q = L\otimes I_d$ (see Section \ref{sec:gen-framework}). Note that assumption \ref{a:grab-bag}, part 1 implies that $L\otimes I_d$ has exactly $d$ zero eigenvalues, and thus is a special case of Assumption \ref{a:Q-PSD1}. From here, it is straightforward to verify the remaining assumptions and see that Theorem \ref{thrm:saddle-instability} implies Theorem \ref{thrm:DT-nonconvergence-gen}.


\end{remark}


\section{Continuous-Time Dynamics and the Stable-Manifold Theorem} \label{sec:CT-stable-manifold}

The analysis of Section \ref{sec:D-SGD-geometry} will build off of the stable-manifold theorem derived in \citep{SMKP-TAC2020}. In this section, we review the key results from \citep{SMKP-TAC2020} that will be required later in the paper. In particular, in Section \ref{sec:stab-man-review} we review the stable-manifold theorem for DGF.\footnote{In fact, the stable manifold applies to the process \eqref{eq:ODE1}. However, to simplify nomenclature we refer to it as the stable manifold for DGF.} 
In Section \ref{sec:stable-man-key-defs} we review some key definitions from \citep{SMKP-TAC2020} that will be required in the subsequent analysis.\footnote{In order to minimize the overlap between this paper and \citep{SMKP-TAC2020}, Sections \ref{sec:stab-man-review}--\ref{sec:stable-man-key-defs} review only the material from \citep{SMKP-TAC2020} that is required to keep the present paper self contained.} In Section \ref{sec:stabl-man-C3-smooth} we prove an important smoothness property of the stable manifold that will be required later.

\subsection{The Stable-Manifold Theorem for DGF} \label{sec:stab-man-review}

The following result states the stable manifold theorem for DGF. 
\begin{theorem}[\citep{SMKP-TAC2020}, Theorem 21] \label{thrm:main-continuous}
Suppose that $h$ is $C^2$ in a neighborhood of $x^*$ and $x^*$ is a regular saddle point of $h\vert_{\C}$. Suppose that Assumptions \ref{a:Q-PSD1} and \ref{a:eigvec-continuity-h} hold and the weight function $t\mapsto \gamma_t$ is $C^1$ and satisfies $\gamma_t\to\infty$. Let $q$ denote the number of negative eigenvalues of $\Hess h\vert_{\C}(x^*)$. Then there exists a $C^1$ manifold $\calS\subset [0,\infty)\times\R^{M}$ with dimension $M-q+1$ such that the following holds: For all $t_0$ sufficiently large, a solution $\vx$ to \eqref{eq:ODE1} converges to $x^*$ if and only if $\vx$ is initialized on $\calS$, i.e., $\vx(t_0) = x_0$, with $(t_0,x_0)\in\calS$.
\end{theorem}

\begin{remark}[On the eigenvector differentiability assumption] \label{remark:eigval-cont}
We require Assumption \ref{a:eigvec-continuity-h} to establish a stable manifold theorem for DGF, but an analogous assumption is not needed in the classical setting. At a high level, the reason Assumption \ref{a:eigvec-continuity-h} is needed here is because we are dealing with constrained optimization and penalty methods that operate from the \emph{exterior} of the constraint set. That is, trajectories to our optimization dynamics generally reside outside the constraint set. 
As trajectories are brought towards a point in the constraint set $\C$ (or the consensus subspace in the case of DGF), it can occur that that the eigenvectors along the path rotate without converging to a limit. e.g., Example 19 in \citep{SMKP-TAC2020}. Assumption \ref{a:eigvec-continuity-h} allows us to rule out such pathological behavior. 
\end{remark}

\subsection{Stable Manifold: Key Definitions and Notation} \label{sec:stable-man-key-defs}

In \citep{SMKP-TAC2020} the stable manifold is constructed using a change of coordinates. The stable manifold is simpler to express and analyze under this change of coordinates. For the sake of subsequent manipulations, we will review the coordinate change now and review other key properties of the stable manifold for DGF.

For convenience, we will consider critical point of $h\vert_\C$ residing at the origin. 
The following result from \citep{SMKP-TAC2020} establishes the existence of a time-varying critical point of the penalized function $h(x) + \gamma_tx^\T Qx$ near $0$. 
\begin{lemma} [\citep{SMKP-TAC2020}, Lemma 23] \label{lemma:g-existence}
Suppose that $h$ is $C^2$ in a neighborhood of $0$ and $0$ is a regular saddle point of $h\vert_{\C}$. Suppose that Assumptions \ref{a:Q-PSD1} and \ref{a:eigvec-continuity-h} hold and the weight function $t\mapsto \gamma_t$ is $C^1$ in $t$ and satisfies $\gamma_t\to\infty$.
Then there exists a constant $\gamma_0>0$ and a function $g:[\gamma_0,\infty)\to \R^M$ such that (i)
$\nabla h(g(\gamma)) - g(\gamma)^\T Q = 0$ for all $\gamma\geq \gamma_0$ and (ii) $g(\gamma)\to 0$ as $\gamma\to\infty$. Moreover, the arc length of $\{g(\gamma): \gamma\geq \gamma_0\}$ is finite, where $\gamma_0$ is a sufficiently large constant, i.e.,
\begin{equation} \label{eq:g-bdd-arc-length}
\int_{\gamma_0}^\infty |g'(s)|\ds < \infty.
\end{equation}
\end{lemma}

Having established the existence of the time-varying critical point $g(t)$, it will be convenient to consider a time-varying coordinate change that recenters the coordinate system about $g(t)$ and rotates the coordinate system so as to make the Hessian of $h(x) + \gamma_t x^\T Qx $ diagonal at $g(t)$. We will do this next. We note that this same change of coordinates is used in \citep{SMKP-TAC2020} (see proof of Lemma 24 therein). However, we repeat the construction of the coordinate change here because several of the intermediate quantities defined during this construction will be critical to the analysis later in the paper.

For $t\geq 0$ let
\begin{equation} \label{eq:A-matrix-def}
A(t) \coloneqq -\Hess_x \left( h(x) + \gamma_t \frac{1}{2}x^\T  Q x \right)\bigg\vert_{x = g(\gamma_t)}
\end{equation}
be the linearization of the right hand side of \eqref{eq:ODE1} centered about $g(\gamma_t)$. 
Letting $y=x-g(\gamma_t)$, we let
\begin{equation} \label{eq:F-def}
F(y,t) \coloneqq -\nabla_x h(y + g(\gamma_t)) - \gamma_t Q(y + g(\gamma_t)) - A(t)y
\end{equation}
represent the error between the linearized and nonlinear dynamics. 

For each $t\geq 0$, let $U(t)$ be a unitary matrix that diagonalizes $A(t)$ (which is possible as $A(t)$ is always symmetric), so that
\begin{equation} \label{def:Lambda-t}
\Lambda(t) \coloneqq U(t)A(t)U(t)^\T ,
\end{equation}
where $\Lambda(t)$ is diagonal.
Since $\gamma_t\in C^1$, 
by Assumption \ref{a:eigvec-continuity-h}  we may construct $U(t)$ as a differentiable function with $U(t)$ converging to some fixed matrix
as $t\to\infty$ (see \citep{SMKP-TAC2020}).
The stable manifold can be constructed by recentering about $g(\gamma_t)$ and rotating by $U(t)$. In particular, letting $\vz(t) := U(t)\left( \vx(t) - g(\gamma_t)\right)$ we obtain the convenient coordinate-changed ODE
\begin{align}
    \dot \vz(t) & = 
    U(t)\Big( A(t)U(t)^\T \vz(t) + F(U(t)^\T \vz(t),t)
- g'(\gamma_t)\dot \gamma_t \Big) + \dot{U}(t) U(t)^\T  \vz(t)\\
    & =: H(\vz(t),t). \label{eq:H-vec-field-def}
\end{align}

Let the diagonal matrix $\Lambda(t)$ above be decomposed as\footnote{To simplify notation in this paper, the ordering of coordinates in the following equations has been changed from that used in \citep{SMKP-TAC2020}.}
\begin{equation} \label{eq:Lambda-decomposition}
\Lambda(t) =
\begin{pmatrix}
  \Lambda^{u}(t) & 0 \\
  0 & \Lambda^s(t)
\end{pmatrix}
\end{equation}
where $\Lambda^{s}(t)\in \R^{\ns\times \ns}$ and $\Lambda^u(t) \in \R^{(M-\ns)\times (M-\ns)}$ denote the `stable' and `unstable' diagonal submatrices respectively, and $\ns$ denotes the number of `stable' coordinates. 
It can be shown that for $t$ sufficiently large, all entries in $\Lambda^s(t)$ are less than some constant $c<0$ and all entries in $\Lambda^u(t)$ are greater than 0 ; see \citep{SMKP-TAC2020}, proof of Lemma 24. With this in mind, when defining the dimension $n_s$ above, we implicitly take time sufficiently large.\footnote{Alternatively, $n_s$ can be computed as the number of negative eigenvalues in $\nabla^2 h\vert_\C(x^*)$ plus $M-\dim\C$ (where $M-\dim\C =$  the number of ``off-constraint'' directions). Intuitively, the penalty makes all off-constraint directions stable for $t$ sufficiently large.}

Let the matrices
\begin{equation} \label{eq:V-def}
V^u(t_2,t_1) \coloneqq
\begin{pmatrix}
e^{\int_{t_1}^{t_2} \Lambda^u(\tau)\,d\tau} & 0\\
0 & 0\\
\end{pmatrix},\quad\quad
V^s(t_2,t_1) \coloneqq
\begin{pmatrix}
0 & 0\\
0 & e^{\int_{t_1}^{t_2} \Lambda^s(\tau)\,d\tau}\\
\end{pmatrix}
\end{equation}
denote the respective evolution operators corresponding to the stable and unstable elements in $\Lambda(t)$. 

Similar to \eqref{eq:F-def}, after recentering and rotating, the error between the linearized and nonlinear dynamics is given by
\begin{equation} \label{eq:def-F-tilde}
\tilde F(z,t) \coloneqq U(t)F(U(t)^\T  z,t) + \dot U(t) U(t)z 
\end{equation}
Note that $\tilde F(0,t) = 0$ and $\tilde F(z,t) = o(|z|^2)$ for $t\geq 0$. Consequently, for any $\epsilon>0$ there exists an $r>0$ and a $T\geq 0$ such that for all $t\geq T$ we have
\begin{equation} \label{eq:Lip-F}
|\tilde F(z,t) - \tilde F(\tilde z,t)| \leq \e |z-\tilde z|, \quad\quad \forall ~z,\tilde z\in B_r(0).
\end{equation}
Let $t_0\in \R$, $a^s\in \R^{n_s}$, and $t\geq t_0$. 

Solutions of the following integral equation may be used to compute the stable manifold. 
\begin{align} \label{eq:integral-eq0}
\vu(t,(t_0,a^s)) = & V^s(t,t_0)
\begin{pmatrix}
0\\
a^s
\end{pmatrix}\\
& + \int_{t_0}^t V^s(t,\tau)\bigg(\tilde F(\vu(\tau,(t_0,a^s)),\tau) -U(\tau)g'(\gamma_\tau)\dot\gamma_\tau \bigg)\,d\tau\\
& - \int_{t}^\infty V^u(t,\tau)\bigg( \tilde F(\vu(\tau,(t_0,a^s)),\tau) -U(\tau)g'(\gamma_\tau)\dot\gamma_\tau  \bigg)\,d\tau,
\end{align}
where $\vu:\R\times \R\times \R^{n_s}\to\R^M$ and 0 is a vector of zeros of appropriate dimension. To be precise, we have included the initial time $t_0$ as a parameter in $\vu$. However, in most of our analysis $t_0$ will be fixed. Thus, in an abuse of notation, we will generally suppress the $t_0$ argument and only specify $\vu$ in terms of the arguments $t$ and $a^s$, i.e., $\vu(t,a^s)$. 

It is important to note that \eqref{eq:integral-eq0} is not only useful for constructing the stable manifold (via Picard iteration \citep{coddington1955theory}) but it also provides an extremely useful representation formula for the stable manifold that will be used extensively in the sequel. 

We may choose constants $\sigma> 0$ and $K>0$ such that the following estimates hold
\begin{align}
\label{eq:ex-estimate-s}
\|V^s(t_2,t_1)\| & \leq Ke^{-(\nu+\sigma)(t_2-t_1)}, \quad \quad t_2\geq t_1\\
\label{eq:ex-estimate-u}
\|V^u(t_2,t_1)\| & \leq Ke^{\sigma (t_2-t_1)}, \quad\quad\quad~~~ t_2\leq t_1.
\end{align}
One useful property of solutions to \eqref{eq:integral-eq0} 
is that they are exponentially stable in the sense that\footnote{Recent work has considered optimization dynamics near ``strict'' saddle points, where the Hessian has at least one negative eigenvalue (but may have zero eigenvalues) \citep{ge2015escaping}. An important reason we consider regular rather than strict saddle points, is one cannot in general guarantee that this estimate holds near saddle points that are merely strict.}
\begin{equation} \label{eq:u-bound}
|\vu(t,a^s)| \leq 2K(1+|a^s|)e^{-\alpha(t-t_0)},
\end{equation}
for some constant $\alpha >0$.

Given the above construction, the stable manifold is defined as follows. 
For each $t\in[T,\infty)$ define the component map $\psi_j:\R\times \R^{n_s}\to \R$ by
\begin{equation}\label{eq:psi-S-def}
\psi_j(t_0,z_0^s) \coloneqq \vu_j(t_0,(t_0,z_0^s)), \quad j=n_s+1,\ldots,M,
\end{equation}
and let $\psi = (\psi_j)_{j=n_s+1}^M$.
The stable manifold is defined with respect to the $z$-coordinate system as
$$
\calS \coloneqq\{(t_0,z_0^s,\psi(t_0,z_0^s)), t_0\geq T, \, z_0^s\in \R^k\cap B_{\frac{r}{3}}(0)\}.
$$
The stable manifold  for \eqref{eq:ODE1} in the original coordinates, denoted here by $\tilde \calS$, is obtained by an appropriate change of coordinates, $\tilde \calS \coloneqq\{(t,x)\in \R\times\R^M:~ U(t)(x-g(\gamma_t)) \in \calS \}.$

We will rely heavily on the  definition of the stable manifold in terms of \eqref{eq:psi-S-def} (and  \eqref{eq:integral-eq0}) in the sequel.

\subsection{Stable-Manifold: Higher-Order Smoothness} \label{sec:stabl-man-C3-smooth}
In this section we will strengthen the continuity assumption made in Theorem \ref{thrm:main-continuous} slightly. We will show that if $h$ is one degree smoother than assumed in Theorem \ref{thrm:main-continuous}, then the stable manifold is also one degree smoother (see Assumption \ref{a:h-C3}). To simplify the presentation, in the following proof we (sparingly) use Einstein notation, wherein repeated indices denote summation over that index.
\begin{lemma}\label{lemma:manifold-C2}
Suppose that $x^*$ is a regular saddle point of $h\vert_{\C}$ and Assumptions \ref{a:Q-PSD1} and \ref{a:eigvec-continuity-h}--\ref{a:h-C3} hold.
Assume the weight function $t\mapsto \gamma_t$ is $C^1$ and satisfies $\gamma_t\to\infty$. Then the stable manifold $\calS$ is $C^2$, uniformly in $t$. That is, the functions $\psi_i$ used to define $\calS$ in \eqref{eq:psi-S-def} are $C^2$ and the second derivatives of each $\psi_i$ are bounded uniformly in $t$.
\end{lemma}
\begin{proof}
The proof of this result is an extension of the proof in the classical autonomous case \citep[Sec. 13.4]{coddington1955theory} to the present setting.
Let $\vu(t,a^s)$ be the solution to \eqref{eq:integral-eq0} given $a^s\in \R^\ns$. Given $\vu$, fix $i\in \{1,\ldots,n_s\}$ and suppose $\vz(t,a^s)$ solves
\begin{align} \label{eq:integral-eqz}
  \vz(t,a^s) = & V^s(t,t_0)e_i\\
  & + \int_{t_0}^t V^s(t,\tau) D_x\tilde F(\mathbf{u}(\tau,a^s),\tau) \vz(\tau,a) \,d\tau\\ 
  & - \int_t^\infty V^u(t,\tau) D_x\tilde F(\mathbf{u}(\tau,a^s),\tau) \vz(\tau,a) \,d\tau.
\end{align}
In \citep{SMKP-TAC2020}, Lemma 24, it was shown that there exists a unique solution to \eqref{eq:integral-eqz} for $\|a^s\|$ sufficiently small. Moreover, it was shown that if $\vz(t,a_s)$ solves \eqref{eq:integral-eqz} then $\vz(t,a_s) = (\frac{\partial \vu_j(t,a_s)}{\partial a_i^s} )_{j=1}^M$; this is precisely how $\calS$ was shown to be $C^1$. 

In addition to $i$ fixed above, fix $j\in \{1,\ldots,\ns\}$. We will compute the vector of partial derivatives $(\frac{\partial \vu_\ell(t,a^s)}{\partial a^s_i \partial a^s_j})_{\ell=1}^M$.
Consider the integral equation, for vector $\vw$ with coordinates $w_\ell$, using Einstein notation and writing $V_{\ell, m}^s$ and $V_{\ell,m}^u$ to be the entries of the matrices $V^s$ and $V^u$:
\begin{align} \label{eq:integral-eqw}
w_\ell(t,a^s) = & \quad \int_{t_0}^t V_{\ell, m}^s(t,\tau) \left[ \vz^T(\tau,a^s)D^2_{x}\tilde F_m(\vu(\tau,a^s),\tau)\vz(\tau,a^s) + D_{x} \tilde F_m (\vu(\tau,a^s),\tau)\vw(\tau,a^s) \right]\dtau\\
& \quad - \int_{t}^\infty V_{\ell,m}^u(t,\tau) \left[ \vz^T(\tau,a^s)D^2_{x}\tilde F_m(\vu(\tau,a^s),\tau)\vz(\tau,a^s)  + D_{x} \tilde F_m(\vu(\tau,a^s),\tau)\vw(\tau,a^s) \right]\dtau.
\end{align}
Again using standard successive approximation techniques, i.e., Picard iterations, we see that for $a^s$ sufficiently small and $t_0$ sufficiently large, there exists a unique solution (in the class of continuous functions) to \eqref{eq:integral-eqw}.

Using the same reasoning used to show that $\vz(t,a^s) = \left(\frac{ \vu_j(t,a^s)}{\partial a_i^s}\right)_{j=1}^M$ in the proof of Lemma 25
in \citep{SMKP-TAC2020}, it follows that
$\vw(t,a^s) = \left(\frac{\partial \vz_j(t,a^s)}{\partial a_k^s}\right)_{\ell=1}^{M} = \left( \frac{\partial \vu_\ell(t,a^s)}{\partial a_j^s \partial a_i^s}\right)_{\ell=1}^{M}$. 

We now show that $\vw(t,a^s)$ is uniformly bounded in $t$. Using \eqref{eq:integral-eqw} we have
\begin{align} \label{eq-lemma-w-1}
\|\vw(t,a^s)\| \leq & \quad \int_{t_0}^t \|V^s(t,\tau)\| \|D^2_{x}\tilde F(\vu(\tau,a^s),\tau)\|\|\vz(\tau,a^s)\|^2\dtau\\
&  + \int_{t_0}^t \|V^s(t,\tau)\| \|D_{x} \tilde F(\vu(\tau,a^s),\tau)\|\|\vw(\tau,a^s)\|\dtau \\
& + \int_{t}^\infty \|V^u(t,\tau)\| \|D^2_{x}\tilde F(\vu(\tau,a^s),\tau)\|\|\vz(\tau,a^s)\|^2\dtau\\
& + \int_{t}^\infty \|V^u(t,\tau) \|\|D_{x} \tilde F(\vu(\tau,a^s),\tau)\|\|\vw(\tau,a^s)\|\dtau.
\end{align}

Recall that $F$ is defined in \eqref{eq:F-def}, and $\tilde F$ is obtained from $F$ by \eqref{eq:def-F-tilde}. From \eqref{eq:F-def}, and the assumption that $h \in C^3$, we see that $\|D^2_x F(x,t)\|$ is uniformly bounded for all $t\geq t_0$ and $x$ in a neighborhood of $0$. Since $\vu(t,a^s)\to 0$ as $t\to\infty$ and $U(t)$ is a unitary transformation, this implies that $\|D^2_{x}\tilde F(\vu(t,a^s),t)\|$ is uniformly bounded for all $t\geq t_0$.

It can be shown that $\vz(t,a_s)$ satisfies 
\begin{equation} \label{eq:z-to-zero}
|\vz(t,a^s)| \leq 2K|a^s|e^{-\alpha (t-t_0)}
\end{equation}
\citep[Equation (30)]{SMKP-TAC2020}.
Thus, we may bound the $\vw$-independent components of \eqref{eq-lemma-w-1} as
\begin{align}
& \int_{t_0}^t \|V^s(t,\tau)\| \|D^2_{x}\tilde F(\vu(\tau,a^s),\tau)\|\|\vz(\tau,a^s)\|^2\dtau\\
& +
\int_{t}^\infty \|V^u(t,\tau)\| \|D^2_{x}\tilde F(\vu(\tau,a^s),\tau)\|\|\vz(\tau,a^s)\|^2\dtau \\
\leq & C\int_{t_0}^t e^{-(\sigma+\alpha)(t-\tau)}\dtau + C\int_{t}^\infty e^{\sigma(t-\tau)}\dtau \leq  C\frac{2}{\sigma},
\end{align}
for some constant $C>0$.

Let $\e>0$ be such that $\frac{K\e}{\sigma} < \frac{1}{2}$.
Let $t_0$ be such that $\|D_x \tilde F(\vu(t,a^s),t)\| < \e$ for all $t\geq t_0$ and all $a^s$ sufficiently small.
Returning to \eqref{eq-lemma-w-1} we have
\begin{align}
\|\vw(t,a^s)\| \leq \frac{2C}{\sigma} +
K\e\int_{t_0}^t e^{-(\alpha+\sigma)(t-\tau)} \|\vw(\tau,a^s)\|\dtau +
K\e\int_{t}^\infty e^{\sigma(t-\tau)}\|\vw(\tau,a^s)\|\dtau.
\end{align}
Letting $M = \sup_{t\geq t_0} \|\vw(t,a^s)\|$ the above yields
$$
M \leq \frac{2C}{\sigma} + \frac{2K\e M}{\sigma}.
$$
Since $\frac{2K\e M}{\sigma} < \frac{M}{2}$, we get $M < \frac{4C}{\sigma},$
which concludes the proof.
\end{proof}


\section{A Key Inequality: D-SGD is Repelled from the Stable Manifold} \label{sec:D-SGD-geometry}
In this section our goal is to characterize the manner in which state-time pairs $(x,t)$ are repelled from the DGF stable manifold $\calS$ under the dynamics \eqref{dynamics_DT3}. For convenience, let 
\begin{equation} \label{eq:J-def}
J(x,t) := -\nabla h(x) - \gamma_t Q_x
\end{equation}
denote the right hand side of \eqref{eq:ODE1}. Informally, our goal is to show an inequality of the following form:
\begin{equation} \label{eq:eta-ineq}
\textup{dist}((x+\e J(x,t),t), \calS) \geq (1+c_2\e)\textup{dist}((x,t),\calS) -c_3\e^2,
\end{equation}
where here, $\dist$ refers to a notion of distance that we have not yet defined, $\e\in[0,1]$, and $c_2,c_3>0$ are constants. The idea is that the expected movement of a step of \eqref{dynamics_DT3} (which is an Euler approximation of \eqref{eq:ODE1}) pushes state-time iterates $(x,t)$ away form the stable manifold $\calS$. This will be formalized below.\footnote{Here, we treat the nonautonomous vector field as an augmented autonomous vector field with $\ddt t = 1$. Also, we note that the $\e^2$ term is simply a consequence of handling nonlinearities in the dynamics. As $\e\to 0$, this intuition holds arbitrarily close to $\calS$.}


Proving an inequality of this form will be main goal of this section. This will be accomplished in Proposition \ref{prop:eta-properties}, property 4.\footnote{The stable manifold is most easily studied under a change of coordinates. $T$ in property 4 represents the appropriate change of coordinates. The function $\eta(x,t)$, defined in \eqref{eta-def}, represents the ``distance'' from a point $(x,t)$ to the stable manifold. Note that $\eta$ is defined with respect to the coordinate change $T$, so points must pass through $T$ before being fed to $\eta$.}
(The remaining properties in Proposition \ref{prop:eta-properties} are straightforward, but are included as they will be required in Section \ref{sec:stochastic-analysis-D-SGD}.)

In broad strokes, the main ideas underlying the proof of \eqref{eq:eta-ineq} (or rather, Proposition \ref{prop:eta-properties}, property 4) are as follows. 
\begin{itemize}
    \item To simplify the analysis, we will consider a change of coordinates that ``straightens out'' the stable manifold of \eqref{eq:ODE1}. 
    This is accomplished in \eqref{eq:phi-def}--\eqref{eq:ODE-straightened}. The vector field of the ``rectified'' ODE is given by $G$, defined in \eqref{eq:ODE-straightened}. 
    \item The key inequality \eqref{eq:eta-ineq} will follow by examining a linearization of $G$. In particular, we will need the matrix representing the linearization of $G$, denoted by $\linG_t$ (see \eqref{eq:T-t-def}), to have a  \emph{uniform spectral gap}, i.e., for all $t$ sufficiently large, $\linG_t$ has at least one positive and one negative eigenvalue, and there exists some constant $c>0$ (independent of $t$) such that any eigenvalue $\lambda \in \sigma(\linG_t)$ satisfies $|\lambda| \geq c$. Once this property is obtained, it is clear that there is an exponential dichotomy near the saddle point, and the proof of \eqref{eq:eta-ineq} follows readily. Thus, a large portion of this section will focus on deriving the existence of the spectral gap for $\linG_t$. 
    \item In order to show the existence of a spectral gap, it is helpful to characterize the geometry of $\calS$ as $t\to\infty$. This is because the rectified ODE is explicitly defined to ``straighten out'' $\calS$. By showing that time slices $\calS_{t_0} := \{(x,t)\in\calS:t=t_0\}$ converge in some sense as $t_0\to\infty$, one may deduce that $D_x G$ (and hence $\linG_t$) has a simple structure as $t\to\infty$ (see Lemma \ref{lemma:T-eigvals}). 
    \item The structure of $\calS$ is elucidated by studying the ``constrained'' dynamics obtained by restricting the ODE \eqref{eq:ODE1} to $\calC$ (see \eqref{eq:ODE-autonomous}). In particular, one sees that as $t_0\to \infty$, $\calS_{t_0}$ approximates the stable manifold of the constrained system in an appropriate sense (see Lemma \eqref{lemma:Phi-limit}).
    
\end{itemize}
Above, we have tried to capture the motivation and main ideas of how we approach the proof. The precise steps we will follow are these:
\begin{enumerate}
    \item We define the rectified ODE \eqref{eq:ODE-straightened}.
    \item We define the autonomous in-constraint ODE \eqref{eq:ODE-autonomous}.
    \item We show that time slices of the non-autonomous stable manifold, given by $\calS_{t_0}$, converge to the stable manifold of \eqref{eq:ODE-autonomous} as $t_0\to\infty$ (Lemma \ref{lemma:Phi-limit}).
    \item We show that $\linG_t$ has a uniform spectral gap (Lemma \ref{lemma:T-eigvals}). 
    \item We show that the desired inequality for $\eta$ holds in the context of the rectified system (Lemma \ref{lemma:d-ineq}).
    \item We prove \eqref{eq:eta-ineq} (Proposition \ref{prop:eta-properties}, item 4).
\end{enumerate}



\bigskip
We now proceed as outlined above. 
Without loss of generality we assume $x^*=0$ and let $\C = \myspan\{e_1,e_2,\ldots,e_{\dim\C}\}$, where $e_i$ denotes the $i$-th canonical vector in $R^M$. Let the constraint space $\C$ be decomposed as
$$
\C = E_{s} + E_{u},
$$
where
$$
E_s := \myspan\left\{x\in \C: \Hess h(0)x = \lambda x, \lambda < 0\right\} \quad \text{ and } \quad E_u := \myspan\left\{x\in \C: \Hess h(0)x = \lambda x , \lambda \geq 0\right\}.
$$
Here, $E_s$ and $E_u$ correspond to the stable and unstable eigenspaces of the gradient-flow system restricted to $\C$ and linearized about the origin.

The (nonautonomous) stable manifold $\calS$ constructed in Section \ref{sec:CT-stable-manifold} may be represented locally  as a function $\psi:E_{s}\times \C^{\perp}\times [0,\infty)\to E_{u}$.
More precisely, the stable manifold may be represented locally as\footnote{Note that a stable input $a_s$ for \eqref{eq:psi-S-def} corresponds to $a_s = (x_s,x_{nc})$. The idea is that for $t$ sufficiently large, these coordinates are the stable coordinates in the sense that they correspond to negative eigenvalues in $\Lambda(t)$, given in \eqref{def:Lambda-t}. }
$$
\calS = \{(x,t):\, x_u = \psi(x_s,x_{nc},t),\, t\geq t_0, \|x_s\| < \delta,  \|x_{nc}\| <  \delta\},
$$
for some $\delta > 0$. Note that this representation of the stable manifold is with respect to the coordinate change discussed in Section \ref{sec:stable-man-key-defs}.

For convenience, we now construct a map which flattens out the stable manifold. Namely, we define
\begin{equation} \label{eq:phi-def}
\Phi(x,t) \coloneqq \begin{pmatrix} x_u \\ x_s \\ x_{nc}\end{pmatrix} -
\begin{pmatrix}
\psi(x_s,x_{nc},t)\\
0\\
0
\end{pmatrix}.
\end{equation}
This function locally maps the stable manifold $\mathcal{S}$ to the subspace $\{(y,t): y_i = 0 \text{ for } i \in 1 \dots n_u\}:=\calU$, where $n_u$ is the number of unstable coordinates.

Next, we notice that
\begin{equation}\label{eqn:Dx-Phi}
    D_x \Phi(x,t) = \begin{pmatrix} I_{n_u} & D_{(x_s,x_{nc})} \psi \\ 0 & I_{M-n_u}\end{pmatrix},
\end{equation}
where $D_{(x_s,x_{nc})}$ denotes a single derivative bundling the $x_s$ and $x_{nc}$ coordinates.
Since $D\psi(0,t) = 0$ (Lemma 25 in \citep{SMKP-TAC2020}; see also \eqref{eq:z-to-zero} above)
and since $\psi$ is $C^1$ in $x$ uniformly in time (Lemma \ref{lemma:manifold-C2}), we may use the inverse function theorem to establish that there exist a $C^1$ function $x\mapsto\Phi^{-1}(\cdot,t)$ in some ball $B(0,r)$, for any time $t$. We emphasize that $\Phi^{-1}(\cdot,t)$ inverts the first argument given a time $t$.

Now, suppose that $\vx(t)$ satisfies the ODE $\dot \vx = H(\vx,t)$, with $H$ as given in \eqref{eq:H-vec-field-def}. If we let $\vw(t) = \Phi(\vx(t),t)$, then, by construction of $\Phi$, the space $\calU := \{x\in \R^{M}: x_i = 0 \text{ for } i \in 1 \dots n_u\}$ is (forward) invariant for $\vw$; i.e., $\vw(t_0) \in \calU \implies \vw(t) \in\calU$ for all $t\geq t_0$. A chain rule computation shows that $\vw$ satisfies the ODE
\begin{equation} \label{eq:ODE-straightened}
    \dot \vw = G(\vw,t) \coloneqq D_x [\Phi,\Phi^{-1}(\vw,t)] H(\Phi^{-1}(\vw,t),t) + D_t[ \Phi,(\Phi^{-1}(\vw,t),t)].
\end{equation}
In particular, note that $\calU$ is the ``straightened out'' stable manifold for the above ODE. Consequently, we sometimes refer to \eqref{eq:ODE-straightened} as the rectified ODE. 

Our first result in this section will be to show that $\Phi(\cdot,t)$ converges to a limit as $t\to\infty$. Equivalently, this may be thought of as showing that ``time slices'' of the stable manifold converge to a limit as $t\to\infty$.
To this end, consider the (autonomous) ODE 
\begin{equation} \label{eq:ODE-autonomous}
\ddt \tilde\vx(t) = -\nabla h\vert_\C(\tilde\vx(t)),
\end{equation}
with $\tilde\vx:[0,\infty)\to \R^{\dim \C}$. Supposing that $0$ is a regular saddle point of $h\vert_\C$, by the classical stable manifold theorem  there exists a stable manifold $\calS^\star$ for \eqref{eq:ODE-autonomous}, associated with the rest point at the origin \citep{chicone2006ordinary}. Let $\psi^\star:\R^{\dim(\C) - n_u} \to \R^{n_u}$ be the function defining $\calS^\star$, i.e.,
\begin{equation} \label{def:S-star}
\calS^{\star} = \left\{x\in \R^{\dim\C}:\, x_u = \psi^\star(x_s)\right\}.
\end{equation}
Let $\Phi^\star:\R^{\dim\C} \to \R^{\dim\C}$ be given by
\begin{equation} \label{def:Phi-star}
\Phi^\star(x) \coloneqq
\begin{pmatrix}
x_u\\
x_s
\end{pmatrix}
-
\begin{pmatrix}
\psi^\star(x_s)\\
0
\end{pmatrix},
\end{equation}
where here $0$ denotes the zero vector in $\R^{\dim\C-n_u}$.
Here, $\Phi^\star$ serves an analogous role to $\Phi$ in \eqref{eq:phi-def}, straightening out the stable manifold of the autonomous system into the stable eigenspace of the autonomous system.

Until now, we have been able to restrict our analysis of \eqref{eq:ODE1} to a neighborhood of $x^*$ (where, by Assumption \ref{a:h-C3}, $h$ is smooth). In the following lemma, we consider global behavior of \eqref{eq:ODE1}. Consequently, we must treat  \eqref{eq:ODE1} as a differential \emph{inclusion} rather than a differential equation. (This will be the only point in the paper where treating the differential inclusion for the continuous dynamics is explicitly required.) In particular, instead of \eqref{eq:ODE1}, consider the differential inclusion
\begin{equation}\label{eq:DI-1}
\dot \vx \in - \gamma_t Q\vx -\partial h (\vx).
\end{equation}
Note that this is equivalent to \eqref{eq:ODE1} in a neighborhood of a saddle point satisfying Assumption \ref{a:h-C3}. We will say that $\vx:\R\to\R^M$ is a solution to \eqref{eq:DI-1} if it is absolutely continuous and satisfies \eqref{eq:DI-1} for almost all $t$. Under Assumption \ref{a:h-loc-lip}, the set $\partial h(x)$ is a nonempty, convex, and compact \citep{clarke1990optimization}. Consequently, solutions to \eqref{eq:DI-1} in the above sense exist (though they may not be unique) \citep{aubin2012differential}.

The next lemma shows that
that we obtain uniform convergence to $\C$ for initializations in a neighborhood of the origin. 
\begin{lemma}[Uniform convergence to $\C$] \label{lemma:consensus-uniform}
Suppose Assumptions \ref{a:h-loc-lip}, \ref{a:coercive-h}, and \ref{a:Q-PSD1} hold. For any open neighborhood $\calN$ of $0$ and any $\e>0$ there exists a $\bar t>0$ such that for any solution $\vx(t)$ of \eqref{eq:DI-1} with initial condition $\vx(t_0) = x_0\in\calN$, $t_0\geq 0$, there holds $\dist(\vx(t),\C) \leq \e$ for all $t\geq \bar t$.
\end{lemma}
\begin{remark}
The assumption that $h$ is coercive (Assumption \ref{a:coercive-h}) in Theorem \ref{thrm:saddle-instability} stems from this lemma. 
\end{remark}
\begin{proof}
Without loss of generality, let $\C$ be as given in \eqref{eq:C-rotation} and let $\vx(t)$ be decomposed as
$$
\vx(t) =
\begin{pmatrix}
\vx_c(t)\\
\vx_{nc}(t)
\end{pmatrix}.
$$
By Assumption \ref{a:coercive-h} there exists a compact set $K \ni 0$, $K \subset \R^M$ that is invariant under \eqref{eq:ODE1}. Without loss of generality, assume that $\calN \subset K$. Let $\overline{M} = \sup\{\|z\|:z\in \partial h(x),~ x\in K\}$. Let $\lambda_{\text{min}}$ be the smallest positive eigenvalue of $Q$. Choose $t_1$ so that $\gamma_t \lambda_{\text{min}} \e > 2\overline{M}$ for all $t\geq t_1$. Then for all $t\geq t_1$ and all $\vx$ with $\| \vx_{nc} (t)\| > \e$ we have
\begin{align}
\ddt \| \vx_{nc}(t)\| &= \frac{ \vx_{nc}(t)}{\| \vx_{nc}(t)\|} \cdot \ddt \vx_{nc}(t) \leq \overline{M} - \gamma_t \lambda_{\text{min}} \e\\
& \leq -\overline{M}.
\end{align}
Hence, for $t\geq t_1$, $\|\vx_{nc}(t)\| \leq \|\vx_{nc}(t_1)\| - \bar M(t-t_1)$. But since $K$ is invariant we have $\|\vx_{nc}(t)\| \leq C$ for $C= \sup_{x\in K} \|x\|$. Setting $\bar t = t_1 +\frac{C - \e}{\bar M}$ yields the desired result.
\end{proof}

The following lemma shows that $x\mapsto \Phi(x,t)$ has a limit as $t\to\infty$. We will require the following definition. Let 
$$\Proj_\calC := (I_{\dim\C} ~ 0)\in\R^{\dim\C\times M}$$ 
be the orthogonal projection onto $\C = \myspan\{e_1,\ldots,e_{\dim\C}\}$, where here $0$ is the zero matrix of appropriate dimension.
\begin{lemma} \label{lemma:Phi-limit}
Suppose Assumptions \ref{a:coercive-h}, \ref{a:Q-PSD1} and \ref{a:eigvec-continuity-h}--\ref{a:h-C3} hold and that $0$ is a regular saddle point of $h\vert_\C$. Let $\Phi$ and $\Phi^\star$ be the maps \eqref{eq:phi-def} and \eqref{def:Phi-star} respectively. Let $x$ be decomposed as $x = (x_u,x_s,x_{nc})$. For all $x$ in a neighborhood of $0$ we have that $\lim_{t \to \infty} P_\C \Phi(x_u,x_s,x_{nc},t) = \Phi^\star(x_u,x_s)$.
\end{lemma}
\begin{proof}
Recall that $\Phi$ and $\Phi^\star$ are defined using the stable manifolds of \eqref{eq:ODE1} and \eqref{eq:ODE-autonomous}. The stable manifolds for each ODE, in turn, are constructed using appropriate integral equations (e.g., \eqref{eq:integral-eq0} and \citep[Ch. 4]{chicone2006ordinary}). We will prove the theorem by recasting the result in terms of the integral equations defining $\calS$ and $\calS^\star$.

Let $U(t)\in\R^{M\times M}$ be the diagonalization of $A(t)$ defined after \eqref{def:Lambda-t}. We remind the reader that $A(t)$ is the linear part of the evolution of $\vy$, which is governed by the nonautonomous differential equation after recentering at the perturbed saddle point $g(\gamma_t)$; see equations \eqref{eq:A-matrix-def}--\eqref{def:Lambda-t}.
Let $B$ be given by 
\begin{equation}\label{def:B}
B := -\Hess h\vert_\calC(0),
\end{equation}
so that $B$ represents the linearization of \eqref{eq:ODE-autonomous} about the origin,
and let $U\in\R^{\dim\C\times \dim\C}$ be a unitary matrix that diagonalizes $B$ so $B = U\Lambda^\star U^\T$,
where $\Lambda^\star = \myDiag(\Lambda^{\star,u},\Lambda^{\star,s})$, and where $\Lambda^{\star,s}\in R^{(\dim\C-n_u)\times (\dim\C-n_u)}$ has positive diagonal entries and $\Lambda^{\star,u}\in \R^{n_u\times n_u}$ has negative diagonal entries. Thus far, we have assumed coordinates to be rotated so that $\calC = \myspan\{e_1,\ldots,e_{\dim\C}\}$. Without loss of generality, we will now assume a rotation of coordinates within $\C$; namely, we will assume that $U=I_{\dim\C}$. 

Analogous to \eqref{eq:V-def}, define
\begin{equation} \label{eq:hat-V-def}
\hat V^u(t_2,t_1) \coloneqq
\begin{pmatrix}
e^{ \Lambda^{\star,u}(t_2-t_1)} & 0\\
0 & 0\\
\end{pmatrix}, \quad\quad
\hat V^s(t_2,t_1) \coloneqq
\begin{pmatrix}
0 & 0\\
0 & e^{\Lambda^{\star,s}(t_2-t_1)}\\
\end{pmatrix}.
\end{equation}
Finally, for $x_c\in \R^{\dim\C}$, let $\hat F(x_c) := -\nabla h\vert_\C(x_c) - B_cx$.
Solutions to the following equation define the ``classical'' stable manifold of the ($\calC$-restricted) gradient system  \eqref{eq:ODE-autonomous}   \citep[Ch. 4]{chicone2006ordinary}
\begin{equation} \label{eq:int-eq-in-constraint}
\vw(t,a^s) = \hat V^s(t,t_0)
\begin{pmatrix}
0\\
a_c^s
\end{pmatrix}
+
\int_{t_0}^t \hat V^s(t,\tau)\hat F(\vw(\tau,a_c^s))\dtau - \int_{t}^\infty V^u(t,\tau)\hat F(\vw(\tau,a_c^s))\dtau,
\end{equation}
where $a^s_c\in \R^{\dim\C-n_u}$. Note that this is the classical analog of \eqref{eq:integral-eq0}. We emphasize that in this classical setting one does not have any $t$ dependence (outside of $\vw$) within the integrals; the inclusion of $t$ dependence in $\tilde F$ necessitated the analysis in \citep{SMKP-TAC2020}.

Let $T:\R^{M}\times[0,\infty)\to\R^M$ be given by
\begin{equation} \label{eq:T-coord-change}
T(u,t) := U(t)\left(u - g(\gamma_t)\right)
\end{equation}
where $g(\gamma)$ and $U(t)$ are defined in Section \ref{sec:stable-man-key-defs}. Note that $T$ is the coordinate
transformation used to recenter and diagonalize in the computation of the nonautonomous stable manifold (see Lemma 24 in \citep{SMKP-TAC2020}). Note that $T^{-1}(u,t) :=
U^\T (t)u + g(\gamma_t)$
is well defined and (i) by Assumption \ref{a:eigvec-continuity-h} we have $U(t)\to U= I$
and (ii) by Lemma \ref{lemma:g-existence} we have $g(\gamma_t)\to 0$ as $t\to\infty$.

Given $a^s \in \R^{n_s}$ (where $n_s=M-n_u$) 
let $\vu(t,a^s)$ be the solution to \eqref{eq:integral-eq0}. Let $a^s_c = P_\C a^s$ (the subscripts indicating the `in-constraint' component of $a^s$) and let $\vw(t,a_c^s)$ be the solution to \eqref{eq:int-eq-in-constraint} given $a_c^s$.

Recall that $\Phi$ and $\Phi^\star$ (\eqref{eq:phi-def} and \eqref{def:Phi-star}) are defined in terms of $\psi$ and $\psi^\star$.
In turn, $\psi$ is defined in \eqref{eq:psi-S-def} and $\psi^\star$ is defined as
\begin{equation}\label{eq:psi-star-def}
\psi^\star_j(a_{c}^s) := \vw(t_0,a_c^s),
\end{equation}
$j=1,\ldots,n_u$, where $t\mapsto\vw(t,a_c^s)$ is the solution to \eqref{eq:int-eq-in-constraint} given $a_c^s$ and initialization time $t_0$ \citep[Ch. 4]{chicone2006ordinary}.\footnote{We remark that, in an abuse of notation $\psi$ and $\psi^\star$ have been defined using the natural indexing associated with $\vw$ and $\vu$. This has been done to sidestep the distracting minutia of indexing. We emphasize that these are simply maps with domain/codomain $\psi:\R^{n_s}\to\R^{n_u}$ and $\psi^\star:\R^{\dim\C - n_u}\to\R^{n_u}$.} Thus, loosely speaking, to show the claim in the lemma, it suffices to show that the relevant coordinates of $\vu(t_0,a_s)$ and $\vw(t_0,a_c^s)$ converge in some appropriate sense as $t_0\to\infty$. 
However, since $\vu$ was constructed with respect to the coordinate change \eqref{eq:T-coord-change},
in order to compare $\vu$ and $\vw$ we first need to move them into comparable coordinate systems. 

To make this precise,  
recall that in \eqref{eq:integral-eq0}, $\vu$ was defined using the arguments $\vu(t,t_0,a^s)$ (though we have typically suppressed the argument $t_0$ for brevity). Below, we will require explicit mention of the initial time argument $t_0$. The lemma holds if we show that
\begin{equation} \label{eq:lemma3-eq1}
\lim_{t_0\to\infty} \|\Proj_\calC T^{-1}\Big(\vu\big(t_0,t_0,\tilde a^s\big),t_0\Big) - \vw(t_0,a_c^s)\| = 0,
\end{equation}
where $\tilde a^s = T((a^s,0),t_0)$, $a_s\in \R^{M-n_u}$, $0\in \R^{n_u}$ (this simply shifts $a_s$ into the appropriate coordinate system).
In words, \eqref{eq:lemma3-eq1} may be interpreted as follows: 
\begin{enumerate}
\item We are given some ``stable'' initialization $a_s\in\R^{\dim\C}$ and a time $t_0$. 
\item  $a_c^s$ extracts the in-constraint components of $a^s$ and $\vw(t_0,a_s^c)$ returns the unique point in $\R^{\dim\C}$ lying on $\calS^\star$ corresponding to $a_c^s$ (see \eqref{eq:psi-star-def} and \eqref{def:S-star}). 
\item  $\tilde a^s$ is the transformation of $a_s$ into the alternate coordinate system (via \eqref{eq:T-coord-change}). $\vu(t_0,(t_0,\tilde a^s)))$ returns the unique point in $\R^M$ on the stable manifold $\calS$ (defined with respect to the alternate coordinate system) corresponding to $\tilde a^s$. Afterwards, $T^{-1}(\cdot, t_0)$ moves this point back into the regular coordinate system in $\R^M$. Finally, the projection $\calP_\C$ projects this down to the $d$-dimensional space $\calC$ so it may be compared with $\vw(t_0,a_c^s)$.
\end{enumerate}




Given that we wish to show \eqref{eq:lemma3-eq1}, for brevity of notation we will again return to expressing $\vu$ as a function of two arguments so that $\vu(t,a^s)$ means $\vu(t,t_0,a^s)$.
Explicitly expanding $T^{-1}(\vu(t_0,\tilde a^s),t_0)$ and using \eqref{eq:integral-eq0} we have
\begin{align}
T^{-1}(\vu(t_0,\tilde a^s),t_0) = & U^\T(t_0)
\begin{pmatrix}
\tilde a^s\\
0
\end{pmatrix}
 - U^\T (t_0)\int_{t_0}^\infty V^u(t_0,\tau)\tilde F(\vu(\tau,\tilde a^s),\tau)\dtau \\
& - U^\T (t_0)\int_{t}^\infty V^u(t_0,\tau)U(\tau) g'(\gamma_\tau)\dot\gamma_\tau\dtau + g(\gamma_{t_0}).
\end{align}
Using this expression along with \eqref{eq:int-eq-in-constraint} and the triangle inequality,  we obtain the bound
\begin{equation} \label{eq:lemma2-eq-main-bound}
\|\Proj_\calC T^{-1}(\vu(t_0,\tilde a^s),t_0) - \vw(t_0,a^s_c)\| \leq (a) + (b) + (c),
\end{equation}
where
\begin{align}
(a) & = \Big\|\Proj_\calC\left[ U^\T (t_0)\begin{pmatrix}
\tilde a^s\\
0
\end{pmatrix} + g(\gamma_{t_0})\right] - \begin{pmatrix}
a^s_c\\
0
\end{pmatrix} \Big\|\\
(b) & = \left\|\int_{t_0}^\infty \Proj_\calC U^\T (t_0)V^u(t_0,\tau)\tilde F(\vu(t_0,\tilde a^s),t_0)\dtau - \int_{t_0}^\infty \hat V^u(t_0,\tau) \hat F(\vw(t_0,a^s_c))\dtau\right\|\\
(c) & = \left\|  U^\T(t_0)\int_{t_0}^\infty V^u(t_0,\tau)U(\tau) g'(\gamma_\tau)\dot\gamma_\tau\dtau \right\|
\end{align}

We will bound each of these in turn. Beginning with $(a)$, recall that $\tilde a^s = T((a^s,0),t_0)$ and note that the bracketed term simply gives $T^{-1}(\tilde a^s,t_0) = (a^s,0)$, so, term $(a)$ is zero.


We now consider $(b)$. Let $\hat V_{aug}(t_0,\tau) = (\hat V(t_0,\tau) ~ 0)\in\R^{d\times M}$ be an augmented version of $\hat V$. Suppressing arguments we have,

\begin{align} \label{eq:lemma2-eq2}
\left|\int_{t_0}^{\infty} \left( \Proj_\C U^\T V^u\tilde F\dtau - \hat V^u \hat F\right)\dtau\right| & \leq 
\int_{t_0}^\infty \left| (\Proj_\C U^\T V^u - \hat V_{aug}^u)\tilde F +\hat V_{aug}^u\tilde F - \hat V^u\hat F\right|\dtau\\
& \leq \int_{t_0}^\infty \left| (\Proj_\C U^\T V^u - \hat V_{aug}^u)\tilde F\right|\dtau + \int_{t_0}^\infty \left| \hat V^u(P_\calC\tilde F - \hat F) \right|\dtau,
\end{align}
where in the last line we simply observe that $\hat V_{aug}^u = (\hat V^u ~ 0) \tilde F = \hat V^u P_\C \tilde F$. 

We now bound the right hand side of \eqref{eq:lemma2-eq2}, beginning with the second term.
By construction, we have $\tilde F(0,t_0) = 0$ and $\hat F(0)= 0$. Moreover, by construction $\tilde F(\cdot,t_0)$ and $\hat F$ are uniformly Lipschitz in a neighborhood of 0.
By \eqref{eq:u-bound} (and a similar argument for $\vw$), we have $\vu(\tau,a_s) \leq c e^{-\alpha(\tau-t_0)}$ and $\vw(\tau,a_s) \leq ce^{-\alpha(\tau-t_0)}$ for some constants $c>0$ and $\alpha>0$, independent of $t_0$.
Hence, 
$$
|\Proj_\C  \tilde F(\vu(t_0,\tilde a_s),t_0) - \hat F(\vw(t_0,a_c^s))| \leq 2ce^{-\alpha(t-t_0)},
$$
and, since $\hat V^u(t_0,\tau)$ is uniformly bounded for $\tau\geq t_0$, we have
\begin{align}
\lim_{t_0\to\infty} \int_{t_0}^{\infty}|\hat V^u(t_0,\tau)(\Proj_\C \tilde F(\vu(\tau,a^s),\tau) - \hat F(\vw(\tau,a^s_c)))|\dtau = 0.
\end{align}
We now bound the first term in on the RHS of \eqref{eq:lemma2-eq2}.
Let $\Lambda(t)$ be as defined in \eqref{def:Lambda-t}.
Let the elements $\lambda_1(t),\ldots,\lambda_M(t)$ of $\Lambda(t)$ be ordered respecting the ordering assumed for $x$ earlier 
(see \eqref{eq:phi-def}) and similarly for elements of $\Lambda$.
By Lemma \ref{lemma:eigvalue-convergence} we see that $\lambda_i(t)\to \lambda_i$, $i=1,\ldots,\dim\C$. Thus, 
$$
\lim_{t_0\to\infty} \sup_{\tau\geq t_0} \left|e^{\lambda_i(\tau-t_0)} - e^{\int_{t_0}^{\tau} \lambda_i(s)\ds}\right| =0
$$
for $i=1,\ldots,n_s$. Since $\Proj_\C U^\T(t_0) \to (I_{\dim\C} ~ 0)$ this implies
\begin{equation} \label{eq:lemma2-eq3}
\lim_{t_0\to\infty} \sup_{\tau\geq t_0}
\Big|\Proj_\calC U^\T(t_0) V^s(t_0,\tau)\begin{pmatrix}
\tilde a^s\\
0
\end{pmatrix} - \hat V^s(t_0,\tau)\begin{pmatrix}
a^s_c\\
0
\end{pmatrix}\Big| = 0,
\end{equation}
where we recall that $\tilde a^s$ is given after \eqref{eq:lemma3-eq1} and $a^s_c = P_\C a^s$. 

By \eqref{eq:lemma2-eq3}, we have that $|\Proj_\C U^\T(t_0) V^u(t_0,\tau) - \hat V^u(t_0,\tau))| \leq c$ for some constant $c>0$ for all $t_0$ sufficiently large and all $t\geq t_0$.
Again using 
the facts that $\tilde F(0,t) = 0$ for all $t$, $\tilde F(\cdot,t)$ is uniformly Lipschitz in a ball about zero, and 
the estimate that
$|\vu(t,a^s)| \leq e^{-\alpha(t-t_0)}$ holds in a ball about the origin, uniformly for all $t$, we get that
$$
\lim_{t_0\to\infty} \int_{t_0}^{\infty} (\Proj_\C U^\T(t_0) V^u(t_0,\tau) - \hat V^u(t_0,\tau))\tilde F(\vu(t_0,a^s),\tau)\dtau = 0.
$$


Finally, handling $(c)$, we note that, as $\int_{t_0}^\infty |g'(\gamma_\tau) \dot \gamma_\tau|d\tau < \infty$ by Lemma \ref{lemma:g-existence}, and as $V^u$ is bounded, we can deduce that
\begin{equation}\label{eq:hand-waved}
 \lim_{t_0\to \infty} U^\T (t_0)\int_{t_0}^\infty V^u(t_0,\tau)U(\tau) g'(\gamma_\tau)\dot\gamma_\tau\dtau = 0.
\end{equation}

This accounts for all terms on the right hand side of \eqref{eq:lemma2-eq-main-bound}, thus giving \eqref{eq:lemma3-eq1}.
\end{proof}

\begin{remark}
The intuitive interpretation of Lemma \ref{lemma:Phi-limit} is that the stable manifolds of \eqref{eq:ODE1} and \eqref{eq:ODE-autonomous} asymptotically align, in some sense. A technical but important detail to note when interpreting this lemma is that $\Phi$ is defined with respect to the change of coordinates $T$ (see \eqref{eq:T-coord-change}) while $\Phi^*$ is not. However, since $g(\gamma_t)\to 0$ and the rotation matrix $U(t)$ asymptotically aligns with $U$, $\Phi$ and $\Phi^*$ are \emph{asymptotically} defined with respect to the same coordinate system.
\end{remark}


When approximating D-SGD with DGF, we will make the following assumption about the growth rate of $\gamma_t$ for the continuous-time process. When we approximate D-SGD, this assumption will easily be satisfied as a consequence of Assumption \ref{a:weights-h}.

\begin{newassumption} \label{ass-gamma-rate}
$\gamma_t$ takes the form $\gamma_t = \Theta(t^r)$ for some $r>0$.
\end{newassumption}
\begin{lemma} \label{lemma:gamma-condition}
Suppose Assumption \ref{ass-gamma-rate} holds.
The following technical condition holds: For fixed $t_0$, 
$$\int_{t_0}^t \gamma_t e^{-\int_{\tau}^t\gamma_s \ds} e^{-\alpha(\tau-t_0)}d\tau \to 0 \quad \mbox{ as } \quad t\to\infty$$
where $\alpha>0$.
\end{lemma}
\begin{proof}
Given the assumed form for $\gamma_t$, 
there will exist a $\kappa>0$ so that for any $t>\tau>t_0$ satisfying $t-\tau > \kappa$ (or, equivalently, $\tau < t-\kappa$) we have that $e^{- (t-\tau)} \geq e^{-\int_\tau^t \gamma_s \,ds}$. We may then estimate
\begin{align*}
\gamma_t \int_{t_0}^t e^{-\int_\tau^t \gamma_s \,ds} e^{-\alpha(\tau - t_0)} \,d\tau \leq & \gamma_t \int_{t_0}^{t-\kappa} e^{- (t-\tau) - \alpha(\tau - t_0)}\,d\tau + \gamma_t \int_{t-\kappa}^t e^{-\alpha(\tau - t_0)} \,d\tau \\
\leq & \gamma_t \int_{t_0}^{t} e^{- (t-\tau) - \alpha(\tau - t_0)}\,d\tau + \gamma_t \int_{t-\kappa}^t e^{-\alpha(\tau - t_0)} \,d\tau \\
 \leq & Ct^r e^{-\alpha t} \to 0 \mbox{ as } t\to\infty,
\end{align*}
for some constant $C>0$, depending on $t_0$. 
\end{proof}

\begin{lemma} \label{lemma:D_t-Phi-to-0}
Suppose Assumptions \ref{a:Q-PSD1} and \ref{a:eigvec-continuity-h}--\ref{ass-gamma-rate}
hold and that $0$ is a regular saddle point of $h\vert_\C$. Then
$D_t \Phi(0,t)\to 0$ and $D_{xt} \Phi(0,t) \to 0$ as $t\to\infty$.
\end{lemma}
\begin{proof}
Recalling that $\Phi$ is defined in \eqref{eq:phi-def}, the result is equivalent to $\frac{\partial}{\partial t} \psi(t,0) \to 0$ and $\frac{\partial^2}{\partial t \partial x} \psi(t,0) \to 0$, where 
$\psi$ is defined componentwise in \eqref{eq:psi-S-def} and where $\vu(t,a^s)$ denotes a solution to \eqref{eq:integral-eq0} with stable initialization $a^s$.

Thus, the claim holds if $\ddt \vu(t,0) \to 0$ and $\frac{\partial^2}{\partial x \partial t} \vu(t,0) \to 0$, where $\vu$ satisfies
\begin{align} \label{eq:lemma4-eq1}
\vu(t,a^s) = & V^s(t,t_0)
\begin{pmatrix}
a^s\\
0
\end{pmatrix}\\
& + \int_{t_0}^t V^s(t,\tau)\left(\tilde F(\vu(\tau,a^s),\tau) -U(\tau)g'(\gamma_\tau)\dot\gamma_\tau \right)\,d\tau\\
& - \int_{t}^\infty V^u(t,\tau)\left( \tilde F(\vu(\tau,a^s),\tau) -U(\tau)g'(\gamma_\tau)\dot\gamma_\tau  \right)\,d\tau,
\end{align}
We begin by estimating $\ddt \vu(t,0) \to 0$. Note that the first term on the RHS of \eqref{eq:lemma4-eq1} is zero (as $a_s = 0$) and so the time derivative of this term is also zero. For the second term on the RHS of \eqref{eq:lemma4-eq1}, taking a derivative in $t$ we obtain that this term is equal to
\begin{equation}\label{eq:lemma4-eq2}
\tilde F(\vu(t,0),t)-U(t)g'(\gamma_t)\dot \gamma_t + \int_{t_0}^t \Lambda(t)V^s(t,\tau) \left( \tilde F(\vu(\tau,a^s),\tau) -U(\tau)g'(\gamma_\tau)\dot\gamma_\tau  \right)\,d\tau.
\end{equation}
The first term above clearly goes to zero as $\vu(0,t) \to 0$ (by \eqref{eq:u-bound}) and $\tilde F$ satisfies \eqref{eq:Lip-F}, with $\tilde F(0,t) = 0$.

The second term in \eqref{eq:lemma4-eq2} goes to zero as $t\to\infty$ by the fact that $U(t)$ converges to a limit (by Assumption \ref{a:eigvec-continuity-h}) and $g'(\gamma_t)\dot\gamma_t \to 0$ by Lemma \ref{lemma:g-existence}.

Note that for $\lambda_i(t)$ in the stable block of $\Lambda(t)$ (see \eqref{eq:Lambda-decomposition}), either $\lambda_i(t)$ converges to a limit (by Lemma \ref{lemma:eigvalue-convergence}) or $\lambda_i(t)\to-\infty$ at rate $\lambda_i(t) = \Theta(\gamma_t)$. By \eqref{eq:Lip-F}, \eqref{eq:u-bound},  and Lemma \ref{lemma:g-existence} we see that $\|\tilde F(\vu(\tau,a^s),\tau) - U(\tau) g'(\gamma_\tau)\dot \gamma_\tau\|$ remains bounded for $\tau \geq t_0$. 
Thus, by Lemma \ref{lemma:gamma-condition} the third term in \eqref{eq:lemma4-eq2} also goes to zero. 

The third term on the RHS of \eqref{eq:lemma4-eq1} is bounded similarly, with the simplification that elements in the unstable block of $\Lambda(t)$ are actually bounded (i.e., they converge to a finite limit by Lemma \ref{lemma:eigvalue-convergence}).
The $D_{xt}$ terms are handled in a completely analogous way.
\end{proof}

Let
\begin{equation} \label{eq:T-t-def}
\linG_t \coloneqq D_x G(0,t),
\end{equation}
where $G$ is given in Equation \eqref{eq:ODE-straightened}.

The next lemma considers basic properties of ``approximate'' eigenvalues and eigenvectors of a matrix that will be used to characterize the eigenstructure of $\linG_t$. 
\begin{lemma}\label{approximate-eig-est}
Suppose that $A$ is a symmetric, $m \times m$ matrix and for some $\|x\| = 1$ we have that $(A-\lambda I)x = y$ and  $\|y\|_2 = \epsilon$. Then $\dist(\lambda,\sigma(A)) \leq \e\sqrt{m}$.
Suppose, moreover, that there is a set of $k$ orthogonal vectors $\{x^i\}$ satisfying $(A-\lambda I)x^i = y_i$ with $\|y_i\| \leq \epsilon$. Then $A$ has at least $k$ mutually orthogonal eigenvectors whose eigenvalues satisfy $|\lambda-\lambda_i| \leq C \epsilon$, where $C>0$ depends on $k$ and $m$.
\end{lemma}
\begin{proof}
Let $\lambda_i\in \R$ and $v_i\in \R^m$ be the eigenvalue/eigenvector pairs of $A$ satisfying $\|v_i\|=1$. Because $A$ is symmetric it possesses orthogonal eigenvectors and we have $x = \sum_{i=1}^m (v_i\cdot x)v_i$. Thus, we may write
\[
(A-\lambda I) x = \sum_{i=1}^m (\lambda_i - \lambda) (v_i\cdot x) v_i = y.
\]
In turn, we have that, for all $i$, $v_i \cdot x = \frac{y \cdot v_i}{\lambda_i - \lambda}$. As $|y \cdot v_i| \leq \epsilon$, and as $|v_i \cdot x|$ must be greater than $1/\sqrt{m}$ for at least one $i$, we then have that there exists an $i$ such that $|\lambda_i - \lambda| \leq \e\sqrt{m}$.

Let $P_{\lambda,\epsilon}$ be the projection onto the eigenspace associated with all the eigenvalues within distance $K\epsilon$ of $\lambda$. Using the same argument as above, on can verify that
$\|(I-P_{\lambda,\e})x^i\|^2 \leq \frac{m}{K}$, and hence $\|P_{\lambda,\epsilon}x^i\| \geq  \sqrt{1-\frac{m}{K}} \geq 1-\frac{m}{2K}$ for each of the $x^i$. As the $x^i$ are orthonormal, we can use this to infer that $\|P_{\lambda,\epsilon}x^i \cdot P_{\lambda,\epsilon}x^j|| \leq \frac{3m}{2K}$ for $i \neq j$. Thus for $K$ large enough, one has that the $P_{\lambda,\epsilon}x^i$ are linearly independent, and hence the projection has rank at least $k$, which completes the proof. 
\end{proof}

The next lemma characterizes the spectral gap of $\linG_t$. This is the key ingredient establishing that \eqref{eq:ODE-straightened} has an exponential dichotomy and that the stable manifold is a linearly unstable set. 
\begin{lemma} (Spectral gap of $\linG_t$) \label{lemma:T-eigvals}
Suppose that $h\in C^2$, Assumptions \ref{a:coercive-h}, \ref{a:Q-PSD1}, \ref{a:eigvec-continuity-h}, and \ref{ass-gamma-rate} hold, and that $0$ is a regular saddle point of $h\vert_\C$.
The following two properties hold:\\
(i) For all $t$ sufficiently large, $\linG_t$ has precisely $M-n_u$ negative eigenvalues and $n_u$ positive eigenvalues.\\
(ii) There exists a $t^*$ such that $\inf\{\lambda\in \sigma(\linG_t): \,\lambda>0, t\geq t^*\}>0$.
\end{lemma}
\begin{proof}
By Equation \eqref{eqn:Dx-Phi} and  Lemma 25 in \citep{SMKP-TAC2020} we have $D_x \Phi(0,t) = I$. By Lemma \ref{lemma:Phi-limit} we have  $\Phi^{-1}(0,t) \to 0$ as $t\to\infty$.
By Lemma \ref{lemma:D_t-Phi-to-0} we have $D_t \Phi(0,t)\to 0$ as $t\to\infty$.

From \eqref{eq:ODE-straightened} and \eqref{eq:T-t-def} we see that
\begin{align}
    \linG_t &=  D_x\left[ D_x[\Phi, (\Phi^{-1}(0,t),t)] H(\Phi^{-1}(0,t),t) - D_t[\Phi,(\Phi^{-1}(0,t),t)], (\Phi^{-1}(0,t),t)\right]\\ 
    &= D_x^2[\Phi, (\Phi^{-1}(0,t),t)]( H(\Phi^{-1}(0,t),t),\cdot) + D_x \Phi(0,t) D_x H(\Phi^{-1}(0,t),t) D_x \Phi^{-1}(0,t)\\ 
    & \quad + D_{xt}\Phi(\Phi^{-1}(0,t),t) D_x \Phi^{-1}(0,t).
\end{align}

By Lemmas \ref{lemma:Phi-limit} and \ref{lemma:D_t-Phi-to-0} we obtain that
\[
 \linG_t =   D_x H(\Phi^{-1}(0,t),t) + o(1) = -D_x^2 h(0) - \gamma_t Q + o(1).
\]

Let
\[
 -D_x^2 h(0) - \gamma_t Q = \begin{pmatrix} A_1 & A_2 \\ A_3 & A_4 \end{pmatrix} - \gamma_t \begin{pmatrix} 0 & 0 \\ 0 & Q_{nc}, \end{pmatrix},
\]
where $Q$ has the above structure by our assumption on the ordering of coordinates (and the fact that $\C$ is the nullspace of $Q$) and where we have ordered coordinates so that the first column represents the action in on-constraint directions and the second column represents the action in off-constraint directions. In particular, we have $A_1\in \R^{\dim\C\times \dim\C}$ and $A_4\in \R^{(M-\dim\C)\times(M-\dim\C)}$. 

We remark that $A_1$ is precisely the Hessian of $h|_\C(0)$.
Let $e = (0,v)$ satisfy $|e| =1$ and $Q_{nc} v = \lambda v$. 
Then $(\linG_t + \gamma_t\lambda I)e = (A_2v,A_4v) + o(1)$. Dividing these matrices by $\gamma_t$ and applying the previous lemma gives that, for large $t$, $\linG_t$ has at least $\textup{rank}(Q)$ linearly independent eigenvectors with eigenvalue given by $-\gamma_t\lambda + O(1)$, where $\lambda$ is an eigenvalue of $Q_{nc}$.

Similarly, if we let $v$ be a unit length eigenvector of $A_1$ with eigenvalue $\lambda$, then we have
\[
(\linG_t -\lambda I) \begin{pmatrix} v \\ \gamma_t^{-1}Q_{nc}^{-1}A_3v \end{pmatrix} = O(1).
\]
Applying Lemma \ref{approximate-eig-est} to these approximate eigenvectors, and using the fact that $A_1$ has a spectral gap (since we assume that $0$ is a regular saddle point of $h\vert_\C$)
then completes the proof.
\end{proof}

Having established that $\linG_t$ has a spectral gap, the next lemma shows that, under our choice of coordinates, $\linG_t$ has a convenient structure. 
\begin{lemma} \label{lemma:T-block-diag}
Suppose that $h\in C^2$, Assumptions \ref{a:coercive-h}, \ref{a:Q-PSD1}, \ref{a:eigvec-continuity-h}, and \ref{ass-gamma-rate} hold and that $0$ is a regular saddle point of $h\vert_\C$. Then for all $t$ sufficiently large, $\linG_t$ has the block diagonal form
$$
\linG_t =
\begin{pmatrix}
P_t & 0\\
0 & Q_t
\end{pmatrix},
$$
where $P_t$ is positive definite and $Q_t$ is negative definite.
\end{lemma}
We remark that the block diagonal structure of $\linG_t$ noted above is a direct consequence of the fact that we are dealing with the vector field for which the stable manifold has been rectified. We now prove the lemma.
\begin{proof}
First, recall that, by construction, $\calU := \{x\in \R^M:x_1=\cdots= x_{n_u} = 0\}$ is invariant under \eqref{eq:ODE-straightened}.
Note that, if an eigenvector of $\linG_t$ has positive eigenvalue, then it must lie in $\calU$.
If this were false, then the space $\calU$ would not be stable under \eqref{eq:ODE-straightened}.
By Lemma \ref{lemma:T-eigvals}, for $t$ sufficiently large, $\linG_t$ has precisely $n_u$ positive eigenvalues and $M-n_u$ negative eigenvalues. 
Let the eigenstructure be arranged so that $\lambda_1,\ldots,\lambda_{n_u}$ are positive and the remaining eigenvalues are negative for $t$ sufficiently large.  (This does not conflict with previous assumptions about ordering or rotation of components.) The corresponding eigenvectors of $v_1,\ldots,v_M$ of $\linG_t$ are divided into two sets so that $\myspan\{v_1,\ldots,v_{n_u}\} = \myspan\{e_1,\ldots,e_{n_u}\} = \calU$ and $\myspan\{v_{n_u+1},\ldots,v_M\} = \myspan\{e_{n_u+1},\ldots,e_M\} = \calU^\perp$.

Letting $V = [v_1,\ldots, v_M]$ be the matrix formed by taking the eigenvectors as columns, by orthogonality $V$ has block diagonal structure
$$
V =
\begin{pmatrix}
V_1 & 0\\
0 & V_2
\end{pmatrix},
$$
and for $\Lambda = \myDiag(\lambda_1,\ldots,\lambda_M)$,  we have
$$
\linG_t =
V\Lambda V^\T  =
\begin{pmatrix}
V_1 & 0\\
0 & V_2
\end{pmatrix}
\begin{pmatrix}
\Lambda_1 & 0\\
0 & \Lambda_2
\end{pmatrix}
\begin{pmatrix}
V_1^\T  & 0\\
0 & V_2^\T
\end{pmatrix}
=
\begin{pmatrix}
V_1\Lambda_1V_1^\T  & 0\\
0 & V_2\Lambda_2V_2^\T
\end{pmatrix}.
$$
\end{proof}

\bigskip
We now define a function that gives a convenient notion of distance to the stable manifold.
Let
\begin{equation} \label{eq:d-def}
d^2(x) \coloneqq \sum_{i=1}^{n_u} x_i^2,
\end{equation}
and
\begin{equation}\label{eta-def}
\eta(x,t) \coloneqq d(\Phi(x,t)).
\end{equation}

The following lemma characterizes the manner in which taking a step in \eqref{eq:ODE1} pushes away from the set $\calU = \{x\in\R^M:x_1=\cdots=x_{n_u}=0\}$ (i.e., roughly, the ``straightened-out'' version of $\calS$).
\begin{remark} (Use of $h\in C^3$ assumption) We remark that the following lemma is the only point in the paper at which Assumption \ref{a:h-C3} ($h\in C^3$) and Lemma \ref{lemma:manifold-C2} are (directly) used. All other uses of Assumption \ref{a:h-C3} in the paper propagate from this lemma.
\end{remark}
\begin{lemma}  \label{lemma:d-ineq}
Suppose Assumptions \ref{a:coercive-h}, \ref{a:Q-PSD1} and \ref{a:eigvec-continuity-h}--\ref{a:h-C3} hold and that $0$ is a regular saddle point of $h\vert_\C$.  Then there exists a constant $c>0$ and a $\delta>0$ such that
$$
d^2(x + \e G(x,t)) \geq (1+\e c)d^2(x)
$$
for all $\e\in[0,1]$, $x\in B_\delta(0)$, and all $t$ sufficiently large, where $G$ is given in Equation \eqref{eq:ODE-straightened}.
\end{lemma}

\begin{proof}
Recall that $\calU=\{x:x_{1}=\cdots=x_{x_u}=0\}$ is an invariant set for \eqref{eq:ODE-straightened}.
This implies that $G_i(x,t) = 0$ for $x\in \calU, ~i=1,\ldots,n_u.$
In turn, this implies that, for $x\in \calU$,
\begin{equation} \label{eq:G_i_part-zero}
    \frac{\partial G_i(x,t)}{\partial x_j\partial x_k} = 0, \quad \textup{ if } j,k \in \{n_u+1,\ldots,M\}
\end{equation}
(i.e., if $j,k$ are both ``stable'' coordinates). 
By Taylor's theorem, for $i=1,\ldots,n_u$ we have
\begin{equation} \label{eq:lemma1-eq1}
G_i(x,t) = (\linG_t x)_i  
+ R_i(x,t),
\end{equation}
where, given an integer $n$, we use the notation $[n]:=\{1,\ldots,n\}$, and
where $R_i(x,t)$ denotes the remainder term. Noting that $G\in C^2$ since $H \in C^2$ (which follows from Assumption \ref{a:h-C3}), we may express the remainder term as
$$
R_i(x,t) = \sum_{|\alpha| = 2} R_{i,\alpha}(x,t)x^\alpha
$$
where
$$
R_{i,\alpha}(x,t) = \int_0^1(1-s)D^\alpha G_i(sx,t)\ds, 
$$
and where here we use the notation $\alpha = (\alpha_1,\ldots,\alpha_M)$ to denote a multi-index, $|\alpha| = \alpha_1+\ldots+\alpha_M$ denotes the cardinality, $D^\alpha G_i$ denotes $\frac{\partial^2 G_i(x,t)}{\partial x_1^{\alpha_1}\cdots x_M^{\alpha_M}}$ and $x^\alpha = x_1^{\alpha_1}\cdots x_M^{\alpha_M}$ \citep{konigsberger2013analysis}. 
Recalling \eqref{eq:G_i_part-zero} we see that $R_i(x,t)$ takes the form
$$
R_i(x,t) = \sum_{\substack{j\in [n_u]\\ k\in [M]}} c_{i,j}(x,t) x_j x_k.
$$
By Lemma \ref{lemma:manifold-C2} we have that $\|D^2_x G_i(x,t)\| \leq C$ for some $C>0$, for all $t\geq t_0$ and all $x$ in a neighborhood of zero.
Using the integral form for the remainder above, we see that the bounded second derivative for $G$ implies that for all $x$ in a neighborhood of zero and all $t\geq 0$ we have
\begin{equation} \label{eq:R-decay1}
|R_i(x,t)| \leq c\Big|\sum_{\substack{j\in [M]\\ k\in [n_u]}} x_jx_k \Big|.
\end{equation}
for some $c>0$. Note that if we restrict $x\in B_\delta(0)$ then clearly we have the componentwise estimate $|x_j|<\delta$, $j\in[M]$. Implementing this simple estimate in \eqref{eq:R-decay1} we see that 
\begin{equation}\label{eq:R-decay2}
R_i(x,t) \leq \delta c d(x)
\end{equation}
for some $c>0$ and $x\in B_\delta(0)$. 

For $\e\in [0,1]$, we now compute
\begin{align}
d^2(x+ \e G(x,t)) & = \sum_{i=1}^{n_u} (x+\e G(x,t))_i^2\\
& = \sum_{i=1}^{n_u} ((x+\e \linG_t x)_i + \e R_i(x,t))^2\\
& = \sum_{i=1}^{n_u} (x+\e \linG_tx)_i^2 + \underbrace{\sum_{i=1}^{n_u} \left( 2(x+\e \linG_t x)_i\e R_i(x,t)+ \e^2 R_i(x,t)^2\right)}_{\coloneqq R(x,t)},
\label{eq:lemma-d-eq1}
\end{align}
and we define $R(x,t)$ as in the last line above.
Note that $W_t$ is bounded in unstable directions and so $\|W_tx\| \leq d(x)$. 
By \eqref{eq:R-decay2} and the definition of $R(x,t)$
we see that
we may choose a constant $\hat c>0$ and $\delta > 0$ such that 
\begin{equation} \label{eq:R-bound4}
R(x,t) \leq \hat c d^2(x)
\end{equation}
for all $x\in B_\delta(0)$ and $t\geq t_0$. 

We now focus on estimating the first term on the right-hand side (last line) of \eqref{eq:lemma-d-eq1}.
For $t$ sufficiently large, $\linG_t$ has the block diagonal structure indicated in Lemma \ref{lemma:T-block-diag}. Let $P_t$ be the positive definite block.
Let $\lambda^*_t$ denote smallest positive eigenvalue of $\linG_t$ at time $t$ and note that by Lemma \ref{lemma:T-eigvals} there exists a time $t^*$ such that $\lambda := \inf_{t\geq t^*} \lambda^*_t >0$.
Thus we see that
\begin{align}
\sum_{i=1}^{n_u} (x + \e\linG_t x)_i^2 = (x_u+\e P_t x_u)^\T (x_u+\e P_t x_u) = x_u^\T (I+\e P_t)^\T (I+\e P_t)x_u \geq (1+\e \lambda)^2d^2(x),
\end{align}
for all $t$ sufficiently large.
Choose $\hat c\in (0,\e\lambda)$ in \eqref{eq:R-bound4} and a corresponding $\delta>0$. 
Letting $c=\lambda-\frac{\hat c}{\e}$, noting that $(1+\e\lambda)^2\geq (1+\e\lambda)$, and returning to \eqref{eq:lemma-d-eq1} this implies that
$$
d^2(x+ \e G(x,t)) \geq (1+\e c)d^2(x)
$$
for all $x\in B_\delta(0)$ and $t$ sufficiently large.
\end{proof}

The following lemma reviews the properties of $d$ and $\eta$ defined in \eqref{eq:d-def} and \eqref{eta-def}. More to the point, the lemma characterizes the relationship between taking discretized steps of \eqref{dynamics_DT3} and the stable manifold $\calS$, in particular, showing that \eqref{dynamics_DT3} is repelled from $\calS$. The properties demonstrated in this lemma will be used in the following section to prove Theorem \ref{thrm:saddle-instability}.

Before stating the lemma, we recall that we use the following notational convention: If $g$ is $C^1(R^m;R^n)$, we use the notation $D[g,x]$ to denote the derivative of $g$ at the point $x$. Treating $D[g,x]:\R^m\to\R^n$ as a linear operator, we use the notation $D[g,x](y)$ or $D[g,x]\circ y$, $y\in \R^m$ to indicate the action of $D[g,x]$ on $y$. We also recall that $J$ is the vector field defined in \eqref{eq:J-def}.
\begin{proposition} \label{prop:eta-properties}
Suppose Assumptions \ref{a:coercive-h}, \ref{a:Q-PSD1} and \ref{a:eigvec-continuity-h}--\ref{a:h-C3} hold and that $0$ is a regular saddle point of $h\vert_\C$. Assume that $t\mapsto \gamma_t$ is $C^2$. Let $J$ be defined in \eqref{eq:J-def}. Then
$d(\cdot)$ and $\eta(\cdot,\cdot)$ have the following properties.
\begin{enumerate}
  \item $d(cx) = cd(x)$ for all $c>0$
  \item $d(\cdot)$ is convex
  \item $d(\cdot)$ is Lipschitz continuous
  \item There exist constants $c_2,c_3>0$ and a $\delta>0$ such that
$$
\eta(T(x+\e J(x,t),t+\e),t+\e)) \geq (1+c_2\e)\eta(T(x,t),t) - c_3\e^2
$$
for $\e\in[0,1]$, $x\in B_\delta(0)$, and $t$ sufficiently large.
  \item Let $\tilde \eta(x,t) = \eta(T(x,t),t)$. For $\tilde \eta(x,t)\not= 0$, $x$ in a neighborhood of 0 and all $t\geq 0$, we have
      $$
      D[\tilde \eta,(x,t)](J(x,t),1)
       > 0.
      $$
\end{enumerate}
\end{proposition}
\begin{remark}[$\calS$ as the repelling object]
In this paper we have discussed stable manifolds for continuous-time systems. Discrete-time systems (with constant step size) also possess stable manifolds \citep{shub2013global}, and these manifold generally differ from their continuous-time counterparts (when the discrete-time processes is obtained by discretization of a continuous-time process). It is important to note that, in each case, the associated stable manifold is precisely the set that the process (be it discrete or continuous-time) is repelled from. Note that in Proposition \ref{prop:eta-properties}, we study the continuous-time stable manifold $\calS$ as a repelling object for a discretization of \eqref{eq:ODE1} with step size $\e$. Because we are using the ``wrong'' stable manifold, these dynamics are not perfectly repelled from $\calS$; this is captured by the error term at the end of property 4 above, indicating that arbitrarily close to $\calS$, the discretization may not step away from $\calS$. However, as $\e\to0$, $\calS$ approximates the stable manifold of the discretized system with higher fidelity, and this error term goes to zero. Since Theorem \ref{thrm:saddle-instability} considers a discretization of \eqref{eq:ODE1} with \emph{decaying} step size (i.e., \eqref{dynamics_DT3}), $\calS$ is asymptotically repelling for these dynamics.
\end{remark}

\begin{proof}
The proof of this proposition is similar to the proof of Proposition 3 in \citep{pemantle1990nonconv}. Properties 1--3 follow readily from the definition of $d$. Property 5 follows from Property 4 by taking $\e\to0$. 
Property 4 is proved as follows. First, note that if $\dot \vx = J(\vx,t)$, so that $\vx$ satisfies \eqref{eq:ODE1}, and $\vy = T(\vx,t)$, then 
$$
\dot \vy = DT \circ
\begin{pmatrix}
J(\vx,t)\\
1
\end{pmatrix} = DT \circ
\begin{pmatrix}
J(T^{-1}(\vy,t),t)\\
1
\end{pmatrix},
$$
where here we use $DT$ as shorthand for $D[T,(\vx,t)]$. Note that, by construction, the right hand side above coincides with the vector field $H$ defined in \eqref{eq:H-vec-field-def}. (This may also be verified directly using \eqref{eq:H-vec-field-def} and \eqref{eq:T-coord-change}.) 

In the computations below we take derivatives of $\Phi$ with respect the base point $(T(x,t),t)$, so $D_x \Phi$ represents $D_x[\Phi,(T(x,t),t)]$ and likewise for $D_t\Phi$. Derivatives of $T$ are taken with respect to the base point $(x,t)$. 
Taking the Taylor expansion of $\Phi\circ T$ we have
\begin{align}
    \Phi(T(x+\e J(x,t),t+\e),t+\e) & = \Phi(T(x,t),t) + \e D_x \Phi\circ DT \circ 
    \begin{pmatrix}
    J(x,t) \\
    1
    \end{pmatrix} + \e D_t\Phi + O(\e^2)\\
    & = \Phi(y,t) + \e D_x \Phi\circ DT \circ 
    \begin{pmatrix}
    J(T^{-1}(y,t),t) \\
    1
    \end{pmatrix} + \e D_t\Phi + O(\e^2)\\
    & = \Phi(y,t) + \e D_x \Phi\circ H(y,t) + \e D_t\Phi + O(\e^2)\\
    & = w + \e D_x \Phi\circ H(\Phi^{-1}(w,t),t) + \e D_t\Phi + O(\e^2)\\
    & = w + G(w,t) + O(\e^2)
\end{align}
where in the first line, the $O(\e^2)$ term follows from Taylor's theorem using the fact that $\Phi$ is $C^2$ (this follows from \eqref{eq:phi-def} and Lemma \ref{lemma:manifold-C2}) and $T$ is $C^2$ (which follows by \eqref{eq:T-coord-change} and the fact that $\gamma_t$ is assumed to be $C^2$), in the second line 
we let $y = T(x,t)$, in the third line we use the form of $H$ given above, in the fourth line we let $w= \Phi(y,t)$, and in the last line we use the definition of $G$ in \eqref{eq:ODE-straightened}. Thus, we see that
\begin{align}
    \eta(T(x+\e J(x,t),t+\e),t+\e)) & = d(\Phi(T(x+\e H(x,t),t+\e),t+\e))\\
    & = d\left( \Phi(T(x,t),t) + G(y,t) + O(\e^2) \right)\\
    & \geq (1+c_2\e)d(\Phi(T(x,t),t)) -c_3\e^2\\
    & = (1+c_2\e)\eta(T(x,t),t) -c_3 \e^2,
\end{align}
for constants $c_2,c_3>0$, where the inequality follows by Lemma \ref{lemma:d-ineq} and the fact that $d$ is Lipschitz.
\end{proof}


\section{Stochastic Analysis} \label{sec:stochastic-analysis-D-SGD}
We now prove Theorem \ref{thrm:saddle-instability}. Our analysis follows a similar approach to \citep[Sec. 4]{pemantle1990nonconv}. 
The strategy of our analysis will rely on the observation that \eqref{dynamics_DT3} is a noisy discretization of the continuous-time process \eqref{eq:ODE1}. As a consequence, we will see that solutions to \eqref{dynamics_DT3} are asymptotically repelled from the stable manifold of \eqref{eq:ODE1}.

To be more precise, suppose that the hypotheses of Theorem \ref{thrm:saddle-instability} hold. Note that the stable manifold constructed in Theorem \ref{thrm:main-continuous} depends not only on $h$ and $Q$, but also on the (continuous-time) weight parameter
$\gamma_t$. In order to construct appropriate continuous-time weight parameters given discrete-time weights $\alpha_k$
and $\gamma_k$,
let $t\mapsto\gamma_t$ be a smooth interpolation of $\gamma_k$ so that $\gamma_t$ and $\gamma_k$ coincide when $t=k$, $k\in \{1,2,\ldots\}$ and $\gamma_t\in C^2$. (Note that if $\gamma_k$ satisfies \ref{a:weights-h}, then $\gamma_t$ can be constructed to satisfy Assumption \ref{ass-gamma-rate}.) 
Per Theorem \ref{thrm:main-continuous}, let $\calS$ be the stable manifold associated with the process \eqref{eq:ODE1} at the given saddle point, given $\gamma_t$. We will see that solutions to \eqref{dynamics_DT3} are repelled from $\calS$, thus constructed.

For $k\geq 1$ let
\begin{equation}
\zeta_k \coloneqq \sum_{j=1}^{k} \alpha_j.
\end{equation}
Informally, \eqref{dynamics_DT3} may be thought of as an Euler approximation of \eqref{eq:ODE1} with (diminishing) step size $\alpha_k$. In this interpretation, $\zeta_k$ represents the time elapsed by iteration $k$. 

Let $d(\cdot)$ and $\eta(\cdot,\cdot)$ be as in \eqref{eq:d-def} and \eqref{eta-def}. With $T$ defined in \eqref{eq:T-coord-change}, let $z(k) = T(x(k),\zeta_k)$, let
\begin{displaymath}
S_k \coloneqq \eta(z(k),\zeta_k),
\end{displaymath}
let $X_k \coloneqq S_k-S_{k-1}$, and let $\calF_k \coloneqq \sigma\left(\{x(j),\xi(j)\}_{j=1}^k \right)$, for $k\geq 1$, where $\sigma(\cdot)$ denotes the $\sigma$-algebra generated by a random variable \citep{durrett2005probability}.
Here, $S_k$ represents the distance of the D-SGD process, $x(k)$, from the stable manifold at iteration $k$, and $X_k$ represents the incremental process.\footnote{$\eta$ takes as inputs points in the alternate coordinate system defined by $T$. Hence, we pass $z(k)$ to $\eta$ rather than $x(k)$.}

To show Theorem \ref{thrm:saddle-instability} it is sufficient to show that $\P(S_k\not\to 0) = 1$.
Intuitively, the proof of Theorem \ref{thrm:saddle-instability} may be broken down into two parts. First, the nondegenerate nature of the noise sequence $\{\xi(k)\}$ (Assumption \ref{a:noise3}) ensures that $S_k$ will eventually wander far from zero (Lemma \ref{lemma:wander-far} below). Second, due to the instability of $\calS$ under \eqref{eq:ODE1},
$S_k$ has positive drift so that, if $S_k$ wanders far from 0, it is unlikely to return (Lemma \ref{lemma:doesnt-return} below).
After proving Lemmas \ref{lemma:wander-far}--\ref{lemma:doesnt-return}, a brief proof of Theorem \ref{thrm:saddle-instability} is given in Section \ref{sec:last-proof}.

\begin{lemma} \label{lemma:wander-far}
Suppose that the hypotheses of Theorem \ref{thrm:saddle-instability} hold.
Then there exists a constant $c_4>0$ such that for all $k$ sufficiently large,
$$
\P\left(\sup_{j\geq k} S_j > c_4 k^{1/2 - \tau_\alpha}\Big\vert \calF_k\right) \geq 1/2.
$$
\end{lemma}
\begin{proof}
The proof is similar to the proof of Lemma 1 in \citep{pemantle1990nonconv}, but adapted to the nonautonomous case. Throughout the proof, without loss of generality, we will assume that $\sup_{x,t}|J(x,t)|< \infty$, where $J$ is the vector field defined in \eqref{eq:J-def}.\footnote{The analysis is easily extended to the case where this property only holds locally near $x^*$ using a simple coupling argument, identical to the end of \citep[Sec. 4]{pemantle1990nonconv}.}

Let
$$\calT := \inf\left\{m \geq k:S_m> c_4k^{1/2-\tau_\alpha}\right\},$$
be a stopping time indicating the first time (after time $k$) that $S_m$ attains the value $c_4k^{1/2-\tau_\alpha}$, where $\tau_\alpha$ is the decay rate of $\alpha_k$ assumed in Assumption \ref{a:weights-h}.
We will prove the result by considering the growth of the second moment of $\E(S_m^2\vert \calF_k)$.
To that end, for $m\geq k$, we begin by estimating the incremental growth
\begin{align}
\E(S^2_{\calT \wedge (m+1)}\vert \calF_k) - \E(S^2_{\calT \wedge m}\vert \calF_k) = & \E(\ones_{\calT > m}(2X_{m+1}S_m + X_{m+1}^2)\vert \calF_k)\\
= & \E(\E(\ones_{\calT > m}2X_{m+1}S_m\vert \calF_m)\vert \calF_k)\\
& \quad + \E(\E(\ones_{\calT > m}X_{m+1}^2\vert\calF_m)\calF_k).
\label{eq:S_n-increment}
\end{align}

We will estimate both of the terms on the right hand side above, beginning with the term $\E(\ones_{\calT > m}2X_{m+1}S_m\vert \calF_m)$. Note that $\ones_{\calT > m}$ and $S_m$ are $\calF_m$-measurable and so may be pulled out of the conditional expectation. 

We will now estimate a lower bound on $\E(X_{m+1}\vert \calF_m)$. Let $\tilde \Phi(x,t) := \Phi(T(x,t),t)$. Observe that
\begin{align} \label{eq:drift-growth}
\E(X_{m+1}\vert \calF_m) = & \E\big(\eta\left(z(m+1),\zeta_{m+1}\right) - \eta\left(x(m),\zeta_m\right)\big\vert \calF_m\big)\\
= & \E\big( d(\tilde \Phi(x(m+1),\zeta_{m+1})\big\vert \calF_m \big) - S_m\\
\geq & d\big(\E( \tilde \Phi(x(m+1),\zeta_{m+1} \big\vert \calF_m)\big) - S_m, \label{eq:drift-growth1}
\end{align}
where the inequality follows from Jensen's inequality and the convexity of $d$. Taking the first-order Taylor approximation of $\tilde \Phi$  and continuing from \eqref{eq:drift-growth1} above we obtain
\begin{align}
\ldots= & d\bigg(\E\Big( \tilde \Phi(x(m),\zeta_{m+1}) + D_x[\tilde \Phi,(x(m),\zeta_{m+1})](x(m+1) - x(m)) \\
& \quad + O(|x(m+1) - x(m)|^2) \Big\vert \calF_m\Big)\bigg) - S_m  \\
= & d\Big( \tilde \Phi(x(m),\zeta_{m+1}) + D_x[\tilde \Phi,(x(m),\zeta_{m+1})] \E\big(x(m+1) - x(m)\Big\vert\calF_m\big)\Big)\\
& \quad + O(\E(\|x(m+1) - x(m)\|^2\,\big\vert \calF_m)) - S_m, 
\label{eq:drift-growth2}
\end{align}
where the second line follows using the fact that $D^2 \Phi$ is uniformly bounded in $t$ (Lemma \ref{lemma:manifold-C2}), so that $D_x\tilde \Phi$ is Locally Lipschitz in $x$ with a constant that holds uniformly across time. 

Using \eqref{dynamics_DT3} and Assumption \ref{a:noise2}, we see that  
$$\E\left(\|x(m+1) - x(m)\|^2\big\vert \calF_m\right) \leq \alpha_m^2 C$$ for some $C>0$. Thus, continuing from \eqref{eq:drift-growth2} above, employing this estimate, the Lipschitz estimate from Proposition \ref{prop:eta-properties}, and the step-size assumption \ref{a:weights-h}, we obtain,
\begin{align}
\ldots= & d\big( \tilde \Phi(x(m),\zeta_{m+1}) + D_x[\tilde \Phi,(x(m),\zeta_{m+1})] \alpha_m J(x(m),\zeta_m)\big) + O(m^{-2\tau_\alpha}) - S_m,
\end{align}
where we recall that $J$ is the right hand side of \eqref{eq:ODE1} (see \eqref{eq:J-def}). Next we ``undo'' the Taylor approximation to obtain
\begin{align}
\ldots = & d\big( \tilde \Phi(x(m) + \alpha_m J(x(m),\zeta_m),\zeta_{m+1})\big) + O(|m^{-\tau_\alpha} J(x(m),\zeta_m)|^2) + O(m^{-2\tau_\alpha}) - S_m \\
= & \eta\big(T\left( x(m) + \alpha_m J(x(m),\zeta_m),\zeta_m \right),\zeta_{m+1}\big) + O(m^{-2\tau_\alpha}) - S_m\\
\geq & (1+c_2\alpha_m)\eta(T(x(m),\zeta_m),\zeta_m) - c_3m^{-2\tau_\alpha} - S_m \\
\geq & Cm^{-\tau_\alpha}S_m  - c_3m^{-2\tau_\alpha}, \label{eq:drift-growth5}
\end{align}
where the first inequality follows from Property 4 of Proposition \ref{prop:eta-properties} and the second inequality follows from Assumption \ref{a:weights-h}.
Thus we see that
\begin{equation} \label{eq:some-equation}
\E(2X_{m+1}S_m\vert\calF_m) \geq Cm^{-\tau_\alpha} S_m^2 + c_3m^{-2\tau_\alpha} S_m,
\end{equation}
for some $C>0$. 

We now estimate the second term on the right hand side of \eqref{eq:S_n-increment}. 
We will use the following convention: 
At $x=0$, where $d$ is not differentiable, we will take $D[d,0]$ to be the particular subgradient of $d$ given by $D[d,0]=\lim_{\delta\to 0} D[d,\delta \hat x]$ where $\hat x = (\ones_{n_u}, ~ \textbf{0}_{M-n_u})$, where here $\textbf{1}$ and $\textbf{0}$ denote vectors of ones and zeros of appropriate size,
so that, by \eqref{eq:d-def} we have $D[d,0]= (\ones_{n_u}, ~ \textbf{0}_{m-n_u})$. Similarly, at points where $\Phi(x,t) = 0$, we define $D[\eta,(x,t)]$ in terms of the previously mentioned definition of $D[d,0]$.

Define $\tilde \eta(x,t) := \eta(T(x,t),t) = d(\tilde \Phi(x,t))$. 
Observe that
\begin{align}
X_{m+1} & = \tilde \eta(x(m+1),\zeta_{m+1}) - \tilde \eta(x(m),\zeta_m)\\
& = d(\tilde \Phi(x(m+1),\zeta_{m+1})) - d(\tilde \Phi(x(m),\zeta_m))\\
& \geq D[d,(\tilde \Phi(x(m),\zeta_m))](\tilde \Phi(x(m+1),\zeta_{m+1}) - \tilde \Phi(x(m),\zeta_m))\\
& = D[d,(\tilde \Phi(x(m),\zeta_m))]\Big(D_x[\tilde \Phi,(x(m),\zeta_{m})](x(m+1) - x(m))\\
&\quad\quad + D_t[\tilde \Phi,(x(m),\zeta_m)]\alpha_m
+  O(\|(x(m+1) - x(m), \alpha_m)\|^2) \Big)\\
& = \alpha_m D[\tilde \eta,(x(m),\zeta_m)]
\begin{pmatrix}
J(x(m),\zeta_m) + \xi(m+1)\\
1
\end{pmatrix} +  O(m^{-2\tau_\alpha})
%
\end{align}
for some $C>0$, where the first inequality follows by convexity of $d$, the fourth line follows by smoothness of $\Phi$, and in the last line we use the fact that $\zeta_{m+1} - \zeta_m = \alpha_m$ and again use the fact that $D^2 \Phi$ is uniformly bounded in $t$ (Lemma \ref{lemma:manifold-C2}) to obtain the $O(m^{-2\tau_\alpha})$ bound on the error.

There exists a constant $c_5>0$ such that, for all $x$ in a neighborhood of zero, all $t$ sufficiently large,
\begin{equation} \label{eq:eta-bound}
\|D[\tilde \eta,(x,t)]\| \geq c_5.
\end{equation}
This follows using the facts that $\|D[d,x]\|= \sqrt{n_u}$ for all $x$ (by our earlier choice of a conventional subgradient for $d$), $\Phi \in C^1$ (by \eqref{eq:phi-def} and Lemma \ref{lemma:manifold-C2}), $D_t\Phi(0,t)\to 0$ as $t\to\infty$ (by Lemma \ref{lemma:D_t-Phi-to-0}), and $D_x \Phi(0,t) \to I$ as $t\to\infty$.
From here we get
\begin{align}
\E(X_{m+1}^+\Big\vert \calF_m) & \geq \alpha_m\E\left( \left(  D[\tilde \eta,(x(m),\zeta_m)]
\begin{pmatrix}
J(x(m),\zeta_m) + \xi(m+1)\\
1
\end{pmatrix}
\right)^+\Big\vert\calF_m\right) + O(m^{-2\tau_\alpha})\\
& \geq \alpha_m\E\left( \left( D[\tilde \eta,(x(m),\zeta_m)]
\begin{pmatrix}
\xi(m+1)\\
0
\end{pmatrix}
\right)^+\Big\vert \calF_m\right) + O(m^{-2\tau_\alpha})\\
& \geq \alpha_m c_5 \E\left(\left(\frac{D[\tilde \eta,x]}{\|D[\tilde \eta,x]\|}\cdot \begin{pmatrix}
\xi(m+1)\\
0
\end{pmatrix}\right)^+\Big\vert \calF_m\right) + O(m^{-2\tau_\alpha})\\
& \geq c_4m^{-\tau_\alpha} + O(m^{-2\tau_\alpha}), \label{eq:drift-growth6}
\end{align}
where the second line follows from Proposition \ref{prop:eta-properties}, property 5, the third line follows from \eqref{eq:eta-bound} and the fourth line follows from Assumption \ref{a:noise3}.

The above inequality implies that
\begin{equation} \label{eq:X-squared-estimate}
\E\left( X_{m+1}^2\big\vert \calF_m \right) \geq Cm^{-2\tau_\alpha}
\end{equation}
for some $C>0$. 

We may now complete the estimate in \eqref{eq:S_n-increment}. Suppose that $S_m = O(m^{-\tau_\alpha})$. Then the r.h.s. of \eqref{eq:X-squared-estimate} dominates the r.h.s. of \eqref{eq:some-equation} and we have
\begin{equation}\label{eq:Sm+1_estimate}
\E(2X_{m+1}S_m + X_{m+1}^2\vert \calF_m) \geq Cm^{-2\tau_\alpha}
\end{equation}
for some $C>0$. On the other hand, if $S_m$ is not $O(m^{-\tau_\alpha})$, then the r.h.s. of \eqref{eq:some-equation} grows at least as fast as $m^{-2\tau_\alpha}$, and \eqref{eq:Sm+1_estimate} still holds. 
Returning to \eqref{eq:S_n-increment}, this gives
\begin{align}
\E\left( S_{\calT\wedge (m+1)}^2\vert \calF_k \right) - \E\left( S_{\calT\wedge m}^2\vert \calF_k \right) & \geq \E\left( \ones_{\calT > m} \frac{C}{m^{2\tau_\alpha}} \Big\vert \calF_k \right)\\
& \geq Cm^{-2\tau_\alpha}
\P(\calT = \infty\vert \calF_k).
\end{align}
By induction we have
\begin{align}
\E\left( S_{\calT \wedge m}^2 \vert \calF_k\right) & \geq S_k^2 + C\P( \calT= \infty\vert \calF_k)\sum_{j=k}^{m-1} \frac{1}{j^{2\tau_\alpha}}\\
&\geq  C \P(\calT=\infty\vert \calF_k)\left( \frac{1}{k^{2\tau_\alpha-1}} -  \frac{1}{m^{2\tau_\alpha-1}} \right). \label{eq:pemantle-lemma1}
\end{align}

We will now compute an upper bound on $\E(S_{\calT\wedge m}^2\big\vert \calF_k)$. Observe that
\begin{align}
    \E\left(S_{\calT\wedge m}^2 \big\vert \calF_k\right) & = \E\left( (S_{(\calT\wedge m) - 1} + X_{\calT\wedge m})^2 \Big\vert \calF_k \right)\\
    & = \E\left(S_{(\calT\wedge m) - 1}^2 \Big\vert \calF_k \right)  + 2\E\left( S_{(\calT\wedge m) - 1}X_{\calT\wedge m} \Big\vert \calF_k \right) +\E\left( X_{\tau\wedge m }^2 \Big\vert \calF_k \right)\\
    & \leq \frac{c_4^2}{k^{2\tau_\alpha-1}} + 2\frac{c_4^2}{k^{\tau_\alpha - \frac{1}{2}}}\E\left(X_{\calT\wedge m} \Big\vert \calF_k \right) + \E\left( X_{\calT\wedge m}^2 \Big\vert \calF_k \right),
\end{align}
where the last line follows by using the definition of $\calT$. Observe that 
$$
\E\left( X^2_{\calT\wedge m}\big|\calF_k \right) 
=  \E\left( \ones_{(m-1)> \calT}  \E\left(  X_m \big|\calF_{m-1}\right)\big\vert\calF_k \right) \leq \alpha_k^2C,
$$
for some $C>0$, where the last inequality follows by using Assumption \ref{a:noise2}, \eqref{dynamics_DT3}, and the fact that $\eta$ is Lipschitz. Likewise, we see that $\big\|\E\left(X_{\calT\wedge m} \Big\vert \calF_k \right) \big\| \leq \alpha_k C$ for some $C>0$.

Using the step size bound in Assumption \ref{a:weights-h} we obtain
\begin{align}
 \E\left(S_{\calT\wedge m}^2 \big\vert \calF_k\right) & \leq 
 \frac{c_4^2}{k^{2\tau_\alpha-1}} + 2\frac{c_4^2}{k^{2\tau_\alpha - \frac{1}{2}}}C + \frac{1}{k^{2\tau_\alpha}}C \leq \frac{2c_4^2}{k^{2\tau_\alpha-1}},
 \label{eq:pemantle-lemma2}
\end{align}
for all $k$ sufficiently large.

Combining \eqref{eq:pemantle-lemma1} and \eqref{eq:pemantle-lemma2} we get
$$
\frac{2c_4^2}{k^{2\tau_\alpha-1}} \geq C\P(\calT=\infty\vert \calF_k)\left( \frac{1}{k^{2\tau_\alpha-1}} -  \frac{1}{m^{2\tau_\alpha-1}} \right).
$$
Letting $m\to\infty$ we get that $\P(\calT=\infty\vert\calF_k)$ is bounded by a constant times $c_4^2$. This can be made smaller than $\frac{1}{2}$ by choosing $c_4$ small enough (which is permissible as we are setting $c_4$ in this lemma)
in which case we have
$$
\P\left(\sup_{m\geq k} S_m \geq c_4 k^{1/2-\tau_\alpha}\vert\calF_k\right) = 1-\P(\calT=\infty\vert\calF_k) \geq 1/2.
$$
\end{proof}

\begin{lemma} \label{lemma:doesnt-return}
Suppose that the hypotheses of Theorem \ref{thrm:saddle-instability} hold. Then there exists a constant $a>0$ such that
$$
\P\left(\inf_{j\geq k} S_j \geq \frac{c_4}{2} k^{1/2-\tau_\alpha}\Big\vert \calF_k,~S_k\geq c_4k^{1/2-\tau_\alpha} \right) > a.
$$
\end{lemma}
The proof of this lemma is nearly identical to the proof of Lemma 2 in \citep{pemantle1990nonconv}, and is omitted for brevity.

\begin{remark}We remark that the only difference in the proofs is that in Lemma \ref{lemma:doesnt-return} above we are required to handle noise with bounded variance rather than noise that is bounded almost surely. In particular, Equation (16) in \citep{pemantle1990nonconv} must be obtained using Assumption \ref{a:noise2} rather than a bounded noise assumption. This is accomplished by observing that for $Y_k = X_k - \E(X_k|\calF_{k-1})$, we have $\E(Y_k|\calF_{k-1}) = 0$. Thus, for the stopping time $\calT := \{j\geq k: S_j\leq \frac{c_4}{2}k^{\frac{1}{2}-\tau_\alpha}\}$ we have
\begin{align}
\E\left[\left(\sum_{j=k}^\calT Y_j \right)^2\right] = \sum_{j=k}^\calT \E(Y_j^2) \leq C\sum_{j=k}^\infty \alpha_j \leq C\frac{1}{k^{\tau_\alpha-1}} ,
\end{align}
for some $C>0$, where the first equality follows by observing that for cross terms $j< k$ we have
$$
\E(Y_k Y_j) = \E(Y_j \E(Y_k|\calF_{j})) = 0,
$$
where the last inequality from Assumption \ref{a:weights-h}.
\end{remark}

Informally, we note that the lemma follows from the observation that for $m\geq k$ we have
$\E(X_{m+1}\vert\calF_m) > 0,$ 
locally, so that $S_k$ has positive drift. This follows from the estimate derived in \eqref{eq:drift-growth1}--\eqref{eq:drift-growth5} and the fact that we condition on the event $S_k \geq c_4 k^{1/2-\tau_\alpha}$ in Lemma \ref{lemma:doesnt-return}.

\subsection{Proof of Theorem \ref{thrm:saddle-instability}} \label{sec:last-proof}
Theorem \ref{thrm:saddle-instability} follows from Lemmas \ref{lemma:wander-far} and \ref{lemma:doesnt-return} as follows.
Note that the event that $S_k$ converges to zero, i.e., the event $E:=\{S_k\to 0\}$, is contained in $\calF_\infty := \bigcup_{k\geq 1} \calF_k$. Suppose, for the sake of contradiction, that $\P(E)>0$. Then, there exists a sequence of events $(E_k)_k$, $E_k\in\calF_k$ such that $E_k\supset E$ and $\P(E_k\backslash E)\to 0$ as $k\to\infty$. Thus, for all $k$ sufficiently large, there exists a neighborhood $\calN$ of 0 such that the probability of leaving $\calN$ conditioned on $E_k$ may be made arbitrarily large; in particular, larger than $1-\frac{a}{2}$. But by Lemmas \ref{lemma:wander-far} and \ref{lemma:doesnt-return} the probability of leaving leaving $\calN$ conditioned on any event in $\calF_k$ is greater than $a/2$ for all $k$ sufficiently large. Thus, by contradiction, we see that $\P(S_k\to 0) = 0$.

\section{Conclusions} \label{sec:conclusions}
The paper considered distributed stochastic gradient descent (D-SGD) for nonconvex optimization. In order to obtain convergence guarantees in the presence of noise, the paper considered decaying step size algorithms. It was shown that D-SGD achieves local-minimum convergence guarantees similar to those known to hold in the centralized setting. In particular, it was shown that D-SGD converges to critical points for nonconvex nonsmooth loss functions and avoids saddle points  when the function is smooth in a neighborhood of the saddle point and the saddle point is regular (see Definition \ref{def:strict saddle}). The assumption that a saddle is regular is similar to (but slightly stronger than) the commonly employed assumption that a saddle is ``strict'' \citep{ge2015escaping}. Proof techniques relied on the method of stochastic approximation. In particular, D-SGD was approximated with a continuous-time distributed gradient flow (referred to as DGF). Nonconvergence to saddle points was obtained by studying the (nonclassical) stable manifold for DGF. Because the dynamics are nonclassical, standard theoretical tools could not be applied, and new tools for handling nonautonomous systems were developed.

We note several potential directions for future research. First, we note that the paper makes the assumption that saddle points are regular (Definition \ref{def:strict saddle}). It may be worthwhile considering extensions to strict saddle points \citep{ge2015escaping}, or higher-order saddle points e.g., \citep{anandkumar2016efficient}. The paper assumes that gradient noise is bounded (Assumption \ref{a:grab-bag}). This assumption is simply made for analytical convenience and can likely be relaxed. The paper also makes the strong assumption that agents communicate over an undirected time-invariant graph. This assumption was made for analytical convenience, and can likely also be relaxed. It may be worthwhile to study applications of the results obtained for the general framework of Section \ref{sec:gen-framework} to time-varying and directed communication graphs. Another important future research direction may be to study the role of stable manifolds and the techniques developed in this paper in order to understand saddle point nonconvergence in important practical scenarios such as D-SGD with compressed inter-agent message passing, quantization, or other errors. 
Finally future work may also consider extending the techniques developed here to more sophisticated distributed first-order algorithms, such as those with momentum or adaptive step sizes.

\appendix
\section{}

The following lemma characterizes the asymptotic properties of the linearization of \eqref{eq:ODE1} near saddle points.
\begin{lemma} \label{lemma:eigvalue-convergence}
Let $A(t)$ be given by \eqref{eq:A-matrix-def} and let $B$ be given by \eqref{def:B}.
Let $d=\dim\C$, let $\{\lambda_1(t),\ldots,\lambda_M(t)\}$ and $\{\lambda_1,\ldots,\lambda_d\}$ denote the eigenvalues of $A(t)$ and $B$ respectively, and assume that $\lambda_i(t) \leq \lambda_j(t)$, $i< j$, and likewise for the $(\lambda_i)_{i=1}^d$. Assume $h$ is $C^2$ and Assumption \ref{a:eigvec-continuity-h} holds.
Then $\lambda_i(t)\to\lambda_i$, $i=1,\ldots,d$, and $\lambda_i(t)\to\infty$, $i=d+1,\ldots,M$. 
\end{lemma}

\begin{proof}
This follows by the continuity of eigenvalues as a function of matrices which holds under Assumption \ref{a:eigvec-continuity-h} and the assumption that $h$ is $C^2$, e.g., see \citep[p. 110]{katoBook}.
\end{proof}

\bigskip
\noindent \textbf {Summary of Some Common Notational Conventions.}
\begin{itemize}
  \item $f=$ sum function \eqref{eqn:f-def}
  \item $N=$ number of agents
  \item $\calC =$ constraint subspace (see \eqref{eq:constraint-def})
  \item $d=$ dimension of domain of $f$ or dimension of $\C$
  \item $M = $ dimension of ambient space in general setup (see Section \ref{sec:general-setup})
  \item $h:\R^M\to \R$ is general objective function (see \eqref{eq:opt-prob})
  \item $Q\in\R^{M\times M}$ is quadratic penalty function matrix (see \eqref{eq:opt-prob})
  \item $x^*=$ critical point of interest
  \item $n_u=$ number of positive (i.e., ``unstable'') eigenvalues of $\Hess h\vert_\C(x^*)$
  \item $\ns= M-n_u = $ dimension of stable eigenspace of $\Hess h(x^*) + \gamma Q$, for $\gamma>0$ large (see above \eqref{eq:Lambda-decomposition})
\end{itemize}

\bibliography{myRefs}

\end{document}